\documentclass[a4paper,11pt]{article}
\usepackage{makeidx}
\usepackage[mathscr]{eucal}
\usepackage[dvipdf]{color} 
\usepackage{amssymb, amsfonts, amsmath, amsthm, graphics}

\newtheorem{definition}{Definition}[section]
\newtheorem{proposition}{Proposition}[section]
\newtheorem{theorem}{Theorem}[section]
\newtheorem{lemma}{Lemma}[section]
\newtheorem{remark}{Remark}[section]
\newtheorem{corollary}{Corollary}[section]
\newtheorem{notation}{Notation}[section]

\numberwithin{equation}{section}

\title{Uniqueness of ground states for combined power-type nonlinear scalar field equations involving the Sobolev critical exponent and  
 a large frequency parameter in three and four dimensions
}
\author{Takafumi Akahori\footnote{The corresponding author; Faculty of Engineering, 
 Shizuoka University, 
 Jyohoku 3-5-1, Hamamatsu-Shi, 
Shizuoka, 432-8561, Japan,~
E-mail: akahori.takafumi@shizuoka.ac.jp}
,~
Miho Murata\footnote{Faculty of Engineering,
Shizuoka University, Jyohoku 3-5-1, Hamamatsu-Shi, 
Shizuoka, 432-8561, Japan,~ 
E-mail: murata.miho@shizuoka.ac.jp}
}
\date{}
\newcommand{\R}{\mathbb{R}}

\begin{document}
\maketitle

\begin{abstract}
We prove the  uniqueness of ground states for 
 combined power-type nonlinear scalar field equations involving  the Sobolev critical exponent and a large frequency parameter. 
 This study is motivated by  the paper \cite{AIIKN} and aims to remove the restriction on dimension imposed there. 
 In this paper,  we  employ  the fixed-point argument developed in \cite{CG} to prove the uniquness.  
 Hence,  the linearization around  the Aubin-Talenti function plays a key role.  Furthermore, 
  we need some estimates for the associated perturbed resolvents (see Proposition \ref{18/11/17/07:17}).
\end{abstract}


\tableofcontents

%%%%%%%%%%%%%%%%%%%%%%%%%%%%%%%%%%%%%%%%%%%%%%%%%%%%%%%

\section{Introduction}

In this paper, we consider the following equation:  
\begin{equation}\label{eq:1.1} 
 - \Delta u + \omega u  
-|u|^{p-1}u-|u|^{\frac{4}{d-2}}u =0 
\quad 
\mbox{in $\mathbb{R}^{d}$}
,
\end{equation}
where $d\ge 3$,  $1< p < \frac{d+2}{d-2}$ and $\omega >0$.  
\par 
This study  is motivated by  the paper \cite{AIIKN} and aims  to  remove the restriction on dimension imposed there.  More precisely,  our aim is to prove the uniqueness of ground states to \eqref{eq:1.1} 
 when $d=3,4$ and  $\omega$ is large (see Theorem \ref{18/09/09/17:09}).  
  We remark that when $\omega$ is small,  the uniqueness had been proved in \cite{AIKN3}  for all $d\ge 3$.  
 We also remark that  the properties of ground states to \eqref{eq:1.1}  are important    
  in the study of the following nonlinear Schr\"odinger and Klein-Gordon equations 
  (see \cite{AIKN3, AIKN2, Grillakis-Shatah-Strauss1}):  
\begin{align}
\label{nls}
i \frac{\partial u }{\partial t }  + \Delta u 
+ |u|^{p-1}u + |u|^{\frac{4}{d-2}}u &=0,
\\[6pt] 
\label{nkg}
\frac{\partial^{2}u }{\partial t^{2}}  - \Delta u 
+ u -|u|^{p-1}u-|u|^{\frac{4}{d-2}}u &=0
.
\end{align}

In order to describe our main result, we introduce a few symbols: 
\begin{notation}\label{20/11/21/10:48}
\begin{enumerate}
\item 
For each $d\ge 3$,   we use $2^{*}$ to denote the Sobolev critical exponent , namely   
\begin{equation}\label{20/8/10/17:8}
2^{*}:=\frac{2d}{d-2}
.
\end{equation}

\item
For $d\ge 3$, $1<p<\frac{d+2}{d-2}$ and $\omega>0$, we define functionals $\mathcal{S}_{\omega}$ and $\mathcal{N}_{\omega}$ on $H^{1}(\mathbb{R}^{d})$ by 
\begin{align}
\label{eq:1.5}
\mathcal{S}_{\omega}(u)
&:=
\frac{\omega}{2}\|u\|_{L^{2}}^{2}
+
\frac{1}{2}\|\nabla u \|_{L^{2}}^{2}
-
\frac{1}{p+1}\|u\|_{L^{p+1}}^{p+1}
-
\frac{1}{2^{*}}\|u\|_{L^{2^{*}}}^{2^{*}}
,
\\[6pt]
\label{20/8/5/10:4}
\mathcal{N}_{\omega}(u)
&:=
\omega \|u\|_{L^{2}}^{2}
+
\|\nabla u \|_{L^{2}}^{2}
-
\|u\|_{L^{p+1}}^{p+1}
-
\|u\|_{L^{2^{*}}}^{2^{*}}
.
\end{align}

\end{enumerate}
\end{notation}

Next, we clarify terminology: 
\begin{definition}\label{20/11/21/11:3}
\begin{enumerate}
\item
Let $d\ge 3$, $1<p<\frac{d+2}{d-2}$ and $\omega>0$. 
Then, by a ground state to \eqref{eq:1.1}, we mean a minimizer for the following variational problem:   
\begin{equation}\label{18/12/07/09:20}
\inf\{ 
\mathcal{S}_{\omega}(u) \colon  
u \in H^{1}(\mathbb{R}^{d})\setminus \{0\},~\mathcal{N}_{\omega}(u)=0  
\}
.
\end{equation}

%%%

\item 
We use the term ``radial'' to mean radially symmetric about $0 \in \mathbb{R}^{d}$. 

\end{enumerate}
\end{definition}

We refer to a known result for the existence of a ground state to \eqref{eq:1.1} (see, e.g., \cite{AIIKN}):  
\begin{proposition}\label{proposition:1.1}
Assume either $d= 3$ and $3<p<5$, or else $d\ge4$ and $1<p<\frac{d+2}{d-2}$. 
 Then, for any $\omega >0$, there exists a ground state to \eqref{eq:1.1}.  Furthermore, for any $\omega >0$, the following hold:
\begin{enumerate}
\item
A ground state to \eqref{eq:1.1}  becomes  a solution to \eqref{eq:1.1} and is in $C^2(\mathbb{R}^{d})$. 

\item
For any ground state $Q_{\omega}$ to \eqref{eq:1.1}, there exist  $y \in \mathbb{R}^{d}$,  $\theta \in \R$ and a positive radial function  $\Phi_{\omega}$ such that $Q_{\omega}(x) = e^{i \theta} \Phi_{\omega} (x-y)$;  In particular,  $\Phi_{\omega}$ is a positive radial ground state to \eqref{eq:1.1}. 

\item 
Let $\Phi_{\omega}$ be a positive radial ground state to \eqref{eq:1.1}. Then, 
 $\Phi_{\omega}=\Phi_{\omega}(x)$ is  strictly decreasing as a function of $|x|$.   
 In particular,  $\|\Phi_{\omega}\|_{L^{\infty}}=\Phi_{\omega}(0)$. 

\item
For any positive radial ground state $\Phi_{\omega}$ to \eqref{eq:1.1}, 
 there exist $\nu>0$ and $C>0$ such that  for any $x\in \mathbb{R}^{d}$, 
\begin{equation}\label{20/8/10/12:16}
|\Phi_{\omega}(x)| 
+
|\nabla \Phi_{\omega}(x)|
+
|\Delta \Phi_{\omega}(x)|
\le
C  e^{-\nu |x|} 
.
\end{equation}
In particular, $\Phi_{\omega} \in H^{2}(\mathbb{R}^{d})$.  
\end{enumerate}
\end{proposition}
\begin{remark}\label{20/8/10/15:34}
In \cite{AIKN6}, it is proved that if $d=3$ and $1<p\le 3$, then there exists $\omega_{c}>0$ such that for any $\omega > \omega_{c}$, 
 there is no ground state to \eqref{eq:1.1}. 
\end{remark}

By Proposition \ref{proposition:1.1},  we may assume that  a ground state to \eqref{eq:1.1} is positive and radial.

Now, we state the main result of this paper: 
\begin{theorem}\label{18/09/09/17:09}
Assume $d=3,4$ and $\frac{4}{d-2}-1<p<\frac{d+2}{d-2}$; namely either $d=3$ and $3<p<5$, or $d=4$ and $1<p<3$. 
Then, there exists $\omega_{*}>0$ such that for any $\omega > \omega_{*}$, 
 a positive radial ground state to \eqref{eq:1.1} is unique. 
\end{theorem}

\begin{remark}\label{20/8/17/7:7}
\begin{enumerate}
\item 
By the result of Pucci and Serrin~\cite{PS} (see also Appendix C of \cite{AIIKN}),  we find that 
   if $3 \leq d \leq 6$ and $\frac{4}{d-2} \leq p < \frac{d+2}{d-2}$, then for any $\omega>0$,   
  a positive radial solution to \eqref{eq:1.1} is unique; 
 in particular, a positive radial ground state is unique.  

%%%
  
\item 
In \cite{AIIKN}, it had been proved that if $d\ge 5$ and $1<p<\frac{d+2}{d-2}$, then  there exists $\omega_{*}>0$ such that  for any $\omega \ge \omega_{*}$, a positive radial ground state to \eqref{eq:1.1} is unique.  
\end{enumerate}
 \end{remark}

We give a proof of Theorem \ref{18/09/09/17:09} in Section \ref{18/12/17/07:07}. 

\subsection*{Acknowledgements}
The authors would like to thank Professor Hiroaki Kikuchi for  helpful discussion. This work was supported by JSPS KAKENHI Grant Number 20K03697. 

%%%%%%%%%%%%%%%%%%%%%%%%%%%%%%%%%%%%%

\section{Preliminaries}\label{20/8/18/11:39}

We collect symbols used  in  this paper, with auxiliary results:  
\begin{enumerate}
\item
For $d\ge 3$, $1<p<\frac{d-2}{d+2}$ and $\omega>0$,  
 the symbol $\mathcal{G}_{\omega}$ denotes the set of positive radial ground states to \eqref{eq:1.1}. 

\item
$L_{\rm real}^{2}(\mathbb{R}^{d})$ denotes the real Hilbert space of 
 complex-valued $L^{2}$-functions equipped with the inner product  
\[
\langle f, g \rangle :=\Re\int_{\mathbb{R}^{d}}f(x) \overline{g(x)}\,dx 
.
\]

\item
For $u,v \in L^{2}(\mathbb{R}^{d})$, we define    
\begin{equation}\label{17/08/15/06:01}
( u,v )
:=
\int_{\mathbb{R}^{d}}u(x)\overline{v(x)} 
\,dx 
.
\end{equation}

\item
We define the Fourier and inverse Fourier transformations to be that 
 for $u\in L^{1}(\mathbb{R}^{d})$,    
\begin{equation}\label{20/8/9/15:35}
\mathcal{F}[u](\xi):=\int_{\mathbb{R}^{d}}e^{-i \xi \cdot x}u(x)\,dx ,
\qquad 
\mathcal{F}^{-1}[u](\xi):=\frac{1}{(2\pi )^{d}}\int_{\mathbb{R}^{d}}e^{i \xi \cdot x}u(x)\,dx
.
\end{equation}

%%%%%%%
%%%%%%%

Then, the relationship between Fourier transformation and convolution (see Theorem 0.1.8 of \cite{Sogge}), 
 and  Parsevarl's identity (see Theorem 0.1.11 of \cite{Sogge}) are  stated as      
\begin{align}
\label{21/4/6/10:15}
\mathcal{F}\big[\int_{\mathbb{R}^{d}} u(\cdot-y)  v(y) \,dy \bigm]
&=
\mathcal{F}[u]\mathcal{F}[v], 
\\[6pt]
\label{20/8/9/15:50}
(u,v)
&=\frac{1}{(2\pi)^{d}}(\mathcal{F}[u], \mathcal{F}[v]).
\end{align}
Furthermore,  it is known (see, e.g., Theorem 2.4.6 of \cite{Grafakos}, or Theorem 5.9 of \cite{Lieb-Loss}\footnote{Notice that the definition of Fourier transformation in \cite{Grafakos} and \cite{Lieb-Loss} is different from ours.}) that  for any $0<r <d$,  
\begin{equation}\label{18/12/22/10:07}
\mathcal{F}[|x|^{-r}](\xi)
=
C_{r}|\xi|^{-(d-r)},
\qquad 
\mbox{with}
\quad 
C_{r}:=\frac{2^{d-r}\pi^{\frac{d}{2}}\Gamma(\frac{d-r}{2})}{\Gamma(\frac{r}{2})}.
\end{equation}

%%%

\item 
We use $W$ to denote the Aubin-Talenti function with $W(0)=1$, namely  
\begin{equation}\label{eq:1.16}
W(x)
:=\Big(1+\frac{|x|^{2}}{d(d-2)} \Big)^{-\frac{d-2}{2}}
.
\end{equation}
Note that 
\begin{equation}\label{20/9/13/9:55}
W(x)\sim (1+|x|)^{-(d-2)},
\qquad 
\|W\|_{L_{\rm weak}^{\frac{d}{d-2}}} <\infty, 
\end{equation}
where the implicit constant depends only on $d$; in particular, if $d=3,4$, then $W  \not \in L^{2}(\mathbb{R}^{d})$. 
It is known that for $d\ge 3$, $W$ is a solution to the following equation: 
\begin{equation}\label{eq:1.15}
\Delta u +|u|^{\frac{4}{d-2}}u=0.
\end{equation}

%%%

\item 
For $d\ge 3$, we define  a potential function $V$ as 
\begin{equation}\label{18/11/05/10:05}
V:=-\frac{d+2}{d-2}W^{\frac{4}{d-2}}.
\end{equation}
Note that $V$ is negative. 
 Observe from \eqref{20/9/13/9:55} that for any $d\ge 3$, 
\begin{equation}\label{20/9/13/9:58}
|V(x)| \lesssim (1+|x|)^{-4}, 
\end{equation}
where the implicit constant depends only on $d$. 

\item 
Let $\lambda>0$. Then, we define the $H^{1}$-scaling operator $T_{\lambda}$ by 
\begin{equation}\label{20/9/25/7:1}
T_{\lambda}[v](x):=\lambda^{-1}v(\lambda^{-\frac{2}{d-2}}x)
.
\end{equation}

\item
We use $\Lambda W$ to denote the function defined as  
\begin{equation}\label{eq:1.19}
\Lambda W
:= 
\frac{d-2}{2}W+x\cdot \nabla W
=
\Big( \frac{d-2}{2} - \frac{|x|^{2}}{2d} \Big) 
\Big( 1+\frac{|x|^{2}}{d(d-2)}\Big)^{-\frac{d}{2}}
.
\end{equation}
Observe that 
\begin{align}
\label{20/9/13/9:52}
&(-\Delta + V )\Lambda W =0 ,
\\[6pt]
\label{20/9/13/10:2}
&|\Lambda W(x)| \lesssim (1+|x|)^{-(d-2)},
\end{align} 
where the implicit constant depends only on $d$.  
In particular, \eqref{20/9/13/10:2} shows that if $d=3,4$, then $\Lambda W  \not \in L^{2}(\mathbb{R}^{d})$.
A computation involving integration by parts shows that 
 for any $\max\{ 1, \frac{2}{d-2}\} <r<\frac{d+2}{d-2}$,  
\begin{equation}\label{20/12/22/14:18}
\langle W^{r},\, \Lambda W \rangle 
=
\frac{-\{ 4-(d-2)(r-1)\}}{2(r+1)}\|W\|_{L^{r+1}}^{r+1} 
<0.  
\end{equation} 
The function $\Lambda W$ can be obtaind from the $H^{1}$-scaling of $W$; Precisely,  we see that for any $\lambda>0$, 
\begin{equation}\label{20/9/25/11:51}
\frac{d}{d\lambda} T_{\lambda}[W] 
=
-\lambda^{-1} T_{\lambda}[W] 
-\frac{2}{d-2} T_{\lambda}[ x\cdot \nabla W]
=
-\frac{2}{d-2}\lambda^{-1}T_{\lambda}[\Lambda W]
, 
\end{equation}
so that $\Lambda W = -\frac{2}{d-2} \frac{d}{d\lambda} T_{\lambda}[W] |_{\lambda=1}$.

Observe from the fundamental theorem of calculus and \eqref{20/9/25/11:51} that for any $\mu>0$, 
\begin{equation}\label{20/9/25/11:45}
T_{\mu}[W] - W 
=
\int_{1}^{\mu} \frac{d}{d\lambda}T_{\lambda}[W]\,d\lambda
=
-\frac{2}{d-2}
\int_{1}^{\mu} 
\lambda^{-1} 
T_{\lambda}[ \Lambda W]
\,d\lambda
. 
\end{equation}
Furthermore,   by \eqref{20/9/25/11:45} and the fundamental theorem of calculus, we see that  for any $\mu>0$, 
\begin{equation}\label{20/9/26/8:41}
\begin{split}
&T_{\mu}[W] - W 
+
\frac{2}{d-2} (\mu -1 )\Lambda W 
\\[6pt]
&=
-\frac{2}{d-2}
\int_{1}^{\mu} 
\int_{1}^{\nu}
\frac{d}{d\lambda}
\Big(
\lambda^{-1} T_{\lambda}[ \Lambda W]
\Big)
 \,d\lambda
d \nu 
\\[6pt]
&=
\frac{2}{d-2}
\int_{1}^{\mu} 
\int_{1}^{\nu} \lambda^{-2}
T_{\lambda}\Big[
2 \Lambda W
+
\frac{2}{d-2} x\cdot \nabla \Lambda W
\Big]
\,d\lambda
d \nu . 
\end{split}
\end{equation}

%%%

\item

We use $\Gamma$ to denote the gamma function. 
 Recall that $\Gamma(\frac{5}{2})=\frac{3\sqrt{\pi}}{4}$ and $\Gamma(3)=2$.

%%%

\item 
For $d\ge 3$, we use $(-\Delta)^{-1}$ to denote the operator defined to be that  for any function $f$ satisfying $(1+|x|)^{-(d-2)}f \in L^{1}(\mathbb{R}^{d})$,
\begin{equation}
\label{18/11/07/14:02}
(-\Delta)^{-1}f(x)ga
:=
\int_{\mathbb{R}^{d}}\frac{\Gamma(\frac{d}{2})}{(d-2) 2 \pi^{\frac{d}{2}}}
|x-y|^{-(d-2)} f(y)\,dy 
.
\end{equation}
Note that H\"older's ineqaultiy shows that if $1\le r<\frac{d}{2}$ and $f \in L^{r}(\mathbb{R}^{d})$, 
 then  $(1+|x|)^{-(d-2)}f \in L^{1}(\mathbb{R}^{d})$. 
 The Hardy-Littlewood-Sobolev inequality shows that
 for any $1< r < \frac{d}{2}$ ($r\neq 1$) and $f\in L^{r}(\mathbb{R}^{d})$,    
\begin{equation}\label{20/8/16/11:20}
\|  (-\Delta)^{-1} f \|_{L^{\frac{dr}{d-2r}}} 
\lesssim 
\|f\|_{L^{r}}
,
\end{equation}
where the implicit constant depends only on $d$ and $r$.  
 Furthermore,  the standard theory for Poisson's equation (see, e.g., Theorem 6.21 of \cite{Lieb-Loss}) shows that 
 for any $1\le r < \frac{d}{2}$ and $f\in L^{r}(\mathbb{R}^{d})$, 
 \begin{equation}\label{18/11/14/09:07}
(-\Delta) (-\Delta)^{-1} f  =f
\quad 
\mbox{in the distribution sense}
.
\end{equation}

By \eqref{18/11/14/09:07} and $(-\Delta +V)\Lambda W =0$ (see \eqref{20/9/13/9:52}), we see that
\begin{equation}\label{20/9/12/14:32}
(-\Delta) (-\Delta)^{-1} V\Lambda W = V\Lambda W = \Delta \Lambda W
\quad 
\mbox{in the distribution sense}
,
\end{equation}
which implies that there exists a harmonic function $h$ on $\mathbb{R}^{d}$ such that 
\begin{equation}\label{20/9/12/14:38}
(-\Delta)^{-1} V\Lambda W = -\Lambda W + h
.
\end{equation}
Note here that  \eqref{20/8/16/11:20} shows that   
\begin{equation}\label{20/8/15/15:57}
\| (\Delta)^{-1} V\Lambda W \|_{L^{2^{*}}}
\lesssim 
\| V\Lambda W\|_{L^{\frac{2d}{d+2}}}  \lesssim 1,  
\end{equation}
whereas the maximum principle shows that a non-trivial harmonic function does not belong to $L^{2^{*}}(\mathbb{R}^{d})$. 
Thus,  $h \equiv 0$ in \eqref{20/9/12/14:38} and therefore    
\begin{equation}\label{20/8/15/15:47}
(-\Delta)^{-1} V\Lambda W = -\Lambda W
.
\end{equation}

%%%

\item 
For $1\le q \le \infty$, we use  $L_{\rm{rad}}^{q}(\mathbb{R}^{d})$ to denote the set of 
 radially symmetric functions about $0 \in \mathbb{R}^{d}$ in $L^{q}(\mathbb{R}^{d})$.

%%%

\item 
For $1\le q \le \infty$, we define a set $X_{q}$ as  
\begin{equation}\label{18/11/06/13:57}
X_{q}
:=\{ 
f\in L_{\rm{rad}}^{q}(\mathbb{R}^{d}) \colon 
\langle f, V \Lambda W \rangle =0 \}
.
\end{equation}
We can verify that $X_{q}$ is a closed subspace of $L^{q}(\mathbb{R}^{d})$. Hence, $X_{q}$ is a Banach space.

\item 
We define a function $\delta$ on $(0,\infty)$ by 
\begin{equation}\label{19/01/27/16:58}
\delta(s)
:=\left\{ \begin{array}{cc}
s^{\frac{1}{2}} & \mbox{if $d=3$},
\\[6pt]
\displaystyle{\frac{1}{\log{(1+s^{-1})}}} &\mbox{if $d=4$}.
\end{array} \right. 
\end{equation}
A computation involving the spherical coordinates shows that 
 for any $s>0$,  
\begin{equation}\label{20/12/12/12:3}
\int_{|\xi|\le 1} 
\frac{1}{(|\xi|^{2}+s )|\xi|^{2}} 
\,d\xi 
=
\frac{d \pi^{\frac{d}{2}}}{\Gamma(\frac{d}{2}+1)}  
\left\{ \begin{array}{lc}
s^{-\frac{1}{2}}\arctan{(s^{-\frac{1}{2}})} 
& \mbox{if $d=3$},
\\[12pt]
\frac{1}{2}\log{(1+s^{-1})} & \mbox{if $d=4$}
.
\end{array} \right. 
\end{equation}
Note here that $\frac{d \pi^{\frac{d}{2}}}{\Gamma(\frac{d}{2}+1)}$ is the surface area of unit sphere  in $\mathbb{R}^{d}$. 
 Furthermore, we see that for any $s>0$, 
\begin{equation}\label{20/12/12/11:12}
\Big|\frac{\pi}{2} -\arctan{(s^{-\frac{1}{2}})} \Big|
\le 
\int_{s^{-\frac{1}{2}}}^{\infty}\frac{1}{1+t^{2}}\,dt 
\le 
\int_{s^{-\frac{1}{2}}}^{\infty}t^{-2} \,dt 
\le 
s^{\frac{1}{2}}
.
\end{equation}
Observe from  \eqref{20/12/12/12:3} and \eqref{20/12/12/11:12}
 that for $d=3,4$ and any $s>0$, 
\begin{equation}\label{20/9/30/9:12}
\Big|
\int_{|\xi|\le 1} 
\frac{1}{(|\xi|^{2}+s )|\xi|^{2}} 
\,d\xi 
-
(5-d)\pi^{2}
\delta(s)^{-1}
\Big|
\lesssim 1,  
\end{equation}
where the implicit constant depends only on $d$.

\item
We define a function $\beta$ on $(0,\infty)$ by 
\begin{equation}\label{19/01/15/08:25}
\beta(s):=\delta(s)^{-1}s
,
\end{equation}
where $\delta$ is the function given in \eqref{19/01/27/16:58}. 

Note that $\beta$ is strictly increasing on $(0,\infty)$, so that the inverse exists.

\item 
We use $\alpha$ to denote the inverse function of $\beta$.

 When $d=3$, $\beta(s) =s^{\frac{1}{2}}$; Hence, the domain of $\alpha$ is $(0,\infty)$.  

 When $d=4$,  the image of $(0,\infty)$ by $\beta$ is $(0,1)$ (see Lemma \ref{20/12/30/8:13}); Hence, the domain of $\alpha$ is $(0,1)$.   
 
Note that  for any $t$ in the domain of $\alpha$,  
\begin{equation}\label{20/11/2/9:40}
t=\beta(\alpha(t)) =\delta(\alpha(t))^{-1}\alpha(t);
\quad  
 \mbox{or equivalently, }
\quad  
\delta\big( \alpha(t) \big) = t^{-1}\alpha(t)
.
\end{equation}
Since $\beta$ is strictly increasing and $\lim_{s\to 0}\beta(s)=0$, 
  we see that  $\alpha$ is strictly increasing  and
\begin{equation}\label{19/01/14/16:26}
\lim_{t \to 0}\alpha(t)=0,
\qquad 
\lim_{t \to 0}\delta(\alpha(t))
=0
.
\end{equation}
Furthermore,  the following holds (see Lemma \ref{20/12/30/8:13}): 
\begin{equation}\label{20/10/15/16:46}
\alpha(t)
\left\{ 
\begin{array}{ll}
=t^{2} & \mbox{for $d=3$ and $0< t <\infty$}, 
\\[6pt]
\sim 
\delta(t) t & \mbox{for $d=4$ and $0< t \le  T_{0}$}
, 
\end{array}
\right.  
\end{equation}
where $T_{0}>0$ is some constant.

\item
Let $d\ge 3$ and $p>1$. Furthermore, let $t\ge 0$, $s\ge 0$,  and let $\eta$ be a function on $\mathbb{R}^{d}$. Then, we define $N(\eta;t)$ and $F(\eta;s,t)$ as 
\begin{align}
\label{19/01/13/15:16}
N(\eta;t)
&:=
|W+\eta|^{\frac{4}{d-2}}(W+\eta) - W^{\frac{d+2}{d-2}} 
- 
\frac{d+2}{d-2}W^{\frac{4}{d-2}} \eta 
\\[6pt]
\nonumber 
&\qquad + 
t \big\{ |W+\eta|^{p-1}(W+\eta) - W^{p}  \big\},
\\[12pt]
\label{19/01/13/15:17}
F(\eta; s, t)& :=-s W + t W^{p} + N(\eta; t).
\end{align}  

\item 
For $d\ge 3$ and functions $\eta_{1},\eta_{2}$ on $\mathbb{R}^{d}$, 
 we define $D(\eta_{1},\eta_{2})$  as  
\begin{equation}\label{20/10/23/1:1}
\begin{split} 
D(\eta_{1},\eta_{2})
&:=|W+\eta_{1}|^{\frac{4}{d-2}}
(W+\eta_{1}) 
-
|W+\eta_{2}|^{\frac{4}{d-2}}
(W+\eta_{2})
- 
\frac{d+2}{d-2}W^{\frac{4}{d-2}} (\eta_{1} -\eta_{2})
.
\end{split}
\end{equation}

%%%

\item 
For $d\ge 3$, $p>1$ and functions $\eta_{1},\eta_{2}$ on $\mathbb{R}^{d}$,
 we define $E(\eta_{1},\eta_{2})$ as 
\begin{equation}\label{20/10/23/1:2}
E(\eta_{1},\eta_{2})
:=
|W+\eta_{1}|^{p-1}
(W+\eta_{1}) 
-
|W+\eta_{2}|^{p-1}
(W+\eta_{2})
.
\end{equation}
\end{enumerate}

Next, we make some remarks about $N(\eta;t)$, $D(\eta_{1},\eta_{2})$ and $E(\eta_{1},\eta_{2})$(see \eqref{19/01/13/15:16}, 
 \eqref{20/10/23/1:1} and \eqref{20/10/23/1:2}). 
 Let $d\ge 3$, $p>1$, $t_{1},t_{2}\ge 0$, and let  $\eta_{1},\eta_{2}$ be functions on $\mathbb{R}^{d}$. 
 Then, the following hold:
\begin{enumerate}
\item 
\begin{equation}\label{20/11/8/11:35}
N(\eta_{1};t_{1}) - N(\eta_{2};t_{2})
=
D(\eta_{1},\eta_{2}) 
+ 
t_{2} E(\eta_{1},\eta_{2})
+
(t_{1}-t_{2}) E(\eta_{1},0)
.
\end{equation} 
Note that $D(\eta,\eta)=0$ and $E(\eta,\eta)=0$. 

Observe from $N(0;0)=0$ that  
\begin{equation}\label{20/11/8/11:39}
N(\eta_{1};t_{1}) =D(\eta_{1},0) + t_{1} E(\eta_{1},0)
.
\end{equation} 
Furthermore, observe from  the fundamental theorem of calculus that 
\begin{align}
\label{19/02/19/11:43}
|D(\eta_{1},\eta_{2})|
&\lesssim 
\big( W+|\eta_{1}|+|\eta_{2}| \big)^{\frac{6-d}{d-2}}
\big( |\eta_{1}|+|\eta_{2}| \big) 
|\eta_{1}-\eta_{2}|,
\\[6pt]
\label{19/02/19/17:02} 
|E(\eta_{1},\eta_{2})|&\lesssim
\big( W+|\eta_{1}|+|\eta_{2}| \big)^{p-1}
|\eta_{1}-\eta_{2} |
,
\end{align}
where the first implicit constant depends only on $d$, and the second one does only on $d$ and $p$.
\end{enumerate}

%%% 

We record well-known results for the resolvent of the Laplace operator. 
 The standard theory for Yukawa's equation (see, e.g., Theorem 6.23 and Lemma 9.11 of \cite{Lieb-Loss}) shows the following:  
\begin{enumerate}
\item
For any $d\ge 1$, $s>0$, $1\le q \le \infty$ and $f\in L^{q}(\mathbb{R}^{d})$, 
\begin{equation}\label{18/11/10/16:21}
(-\Delta+s)^{-1} f(x)
=
\int_{\mathbb{R}^{d}}
\int_{0}^{\infty} (4\pi t)^{-\frac{d}{2}} e^{-\frac{|x-y|^{2}}{4t}} e^{-s
t}\,dt   f(y)\,dy 
. 
\end{equation}
Observe from \eqref{18/11/10/16:21} that if $|f|\le |g|$, then  
\begin{equation}\label{20/9/15/14:54}
|
(-\Delta + s)^{-1} f 
|
\le 
(-\Delta + s)^{-1} |g|  
.
\end{equation}
Using integration by substitution, we see 
 that for each $x \in \mathbb{R}^{d}\setminus \{0\}$, 
\begin{equation}\label{20/8/16/15:23}
\int_{0}^{\infty} (4\pi t)^{-\frac{d}{2}} e^{-\frac{|x|^{2}}{4t}} e^{-s t}\,dt
\le 
\int_{0}^{\infty} (4\pi t)^{-\frac{d}{2}} e^{-\frac{|x|^{2}}{4t}} \,dt  
\lesssim   
|x|^{-(d-2)} ,
\end{equation}
where the implicit constant depends only on $d$. 

\item 
For any $d\ge 1$, $s>0$, $1\le q \le \infty$ and $f\in L^{q}(\mathbb{R}^{d})$, 
\begin{align}
\label{20/8/13/17:5} 
&
(-\Delta+s)^{-1} f \in L^{q}(\mathbb{R}^{d}),
\\[6pt]
\label{18/11/14/09:06}
&
(-\Delta+s)(-\Delta+s)^{-1} f 
=f 
\quad 
\mbox{in the distribution sense}
.
\end{align}
Observe from \eqref{20/8/13/17:5} and \eqref{18/11/14/09:06} that $(-\Delta+s)^{-1} f\in W^{2,q}(\mathbb{R}^{d})$.  

%%%
\item For any $d\ge 1$, $s>0$, $1\le q \le \infty$ and $f \in W^{2,q}(\mathbb{R}^{d})$,    
\begin{equation}\label{20/8/13/15:11}
(-\Delta+s)^{-1}(-\Delta+s)f 
=f 
.
\end{equation}

%%%
\item 
For any $s>0$ and $\xi \in \mathbb{R}^{d}$,
\begin{equation}\label{20/8/18/16:27}
\mathcal{F}\big[
\int_{0}^{\infty} (4\pi t)^{-\frac{d}{2}} e^{-\frac{|x|^{2}}{4t}} e^{-s t}\,dt
\big](\xi)
=
(|\xi|^{2}+s)^{-1}
.
\end{equation}
\end{enumerate}

%%%

At the end of this section, we collect known results about ground states:
\begin{lemma}[Lemma 2.3 of \cite{AIIKN}]\label{proposition:2.3}
Assume either $d= 3$ and $3<p <5$, or else $d\ge 4$ and $1< p<\frac{d+2}{d-2}$. Then, we have 
\begin{align}
\label{eq:2.16}
&\lim_{\omega \to \infty} \inf_{\Phi_{\omega}\in \mathcal{G}_{\omega}} \Phi_{\omega}(0) = \infty,
\\[6pt]
\label{eq:2.17}
&\lim_{\omega \to \infty} \sup_{\Phi_{\omega}\in \mathcal{G}_{\omega}}
 \omega \Phi_{\omega}(0)^{-(2^{*}-2)} 
= 
\lim_{\omega \to \infty} \sup_{\Phi_{\omega}\in \mathcal{G}_{\omega}}
\Phi_{\omega}(0)^{-\{2^{*}-(p+1)\}}
=0.
\end{align} 
\end{lemma}

\begin{lemma}[Proposition 2.1 and Corollary 3.1 of \cite{AIIKN}]
\label{theorem:3.1}
Assume either $d= 3$ and $3<p <5$, or else $d\ge 4$ and $1<p<\frac{d+2}{d-2}$. 
Then, the following hold:    
\begin{align}
\label{20/12/15/11:11}
\lim_{\omega \to \infty} \sup_{\Phi_{\omega}\in \mathcal{G}_{\omega}}
\|T_{\Phi_{\omega}(0)} [\Phi_{\omega}]  - W  \|_{\dot{H}^{1}} 
&=0,
\\[6pt]
\label{eq:3.2}
\lim_{\omega \to \infty} \sup_{\Phi_{\omega}\in \mathcal{G}_{\omega}}
\| T_{\Phi_{\omega}(0)} [\Phi_{\omega}] -W \|_{L^{r}}
&=0
\quad 
\mbox{for all $\frac{d}{d-2} <r <\infty$}
.
\end{align}
\end{lemma}

\begin{lemma}[Proposition 3.1 of \cite{AIIKN}]\label{18/09/05/01:05}
Assume either $d=3$ and $3<p <5$, or else $d\ge 4$. Let $1<p<\frac{d+2}{d-2}$. Then, for any sufficiently large $\omega$, and any $x\in \mathbb{R}^{d}$, 
\begin{equation}\label{eq:3.1}
\sup_{\Phi_{\omega}\in \mathcal{G}_{\omega}}
T_{\Phi_{\omega}(0)} [\Phi_{\omega}](x) 
\lesssim \left(1 +|x| \right)^{-(d-2)}
, 
\end{equation}
where the implicit constant depends only on $d$ and $p$. 
\end{lemma}

%%%%%%%%%%%%%%%%%%%%%%%%%%%%%%%%%%%%%%%%%%%%%%%%%%%%%%%%%%%%%%%%%%%%%%%%%%%%%%%%
\section{Uniform estimates for perturbed resolvents}
\label{18/11/05/09:33}

Our aim in this section is to prove  the following proposition which 
 plays a vital role  for the proof of  Theorem \ref{18/09/09/17:09}: 
\begin{proposition}\label{18/11/17/07:17}
Assume $d= 3,4$, and let $\frac{d}{d-2} < q < \infty$.  
If $s>0$ is sufficiently small dependently on $d$ and $q$, 
 then  the inverse of the operator $1+(-\Delta + s)^{-1}V \colon L_{\rm rad}^{q}(\mathbb{R}^{d})\to L_{\rm rad}^{q}(\mathbb{R}^{d})$ exists;
 and the following estimates hold:  
\begin{enumerate} 
\item
If $f \in L_{\rm rad}^{q}(\mathbb{R}^{d})$, then  
\begin{equation}\label{19/02/02/11:47}
\|\{ 1+(-\Delta + s)^{-1}V\}^{-1} f \|_{L^{q}}
\lesssim 
\delta(s) s^{-1} 
\| f \|_{L^{q}}
,
\end{equation} 
where the implicit constant depends only on $d$ and $q$.  
\item 
If $f \in X_{q}$, then 
\begin{equation}\label{18/11/11/15:50}
\|\{ 1+(-\Delta + s )^{-1}V\}^{-1} f \|_{L^{q}}
\lesssim 
\| f \|_{L^{q}}
,
\end{equation} 
where the implicit constant depends only on $d$ and $q$. 
\end{enumerate}
\end{proposition}
\begin{remark}\label{20/9/13/10:11}
\begin{enumerate}
\item
When $d=3$, the case $q=\infty$ can be included in Proposition \ref{18/11/17/07:17} (see Lemma 2.4 of \cite{CG}).
On the other hand, from the point of view of  Remark \ref{20/9/13/9:1} below, 
  when $d=4$, it seems difficult to include the case $q=\infty$. 
 
\item
By the substitution of variables $\tau=\beta(s):=\delta(s)^{-1}s$, we may write \eqref{19/02/02/11:47} as 
 \begin{equation}\label{20/12/6/11:56}
\|\{ 1+(-\Delta + \alpha(\tau))^{-1}V\}^{-1} f \|_{L^{q}}
\lesssim 
\tau^{-1} 
\| f \|_{L^{q}}
.
\end{equation}

\item   
Observe from $V=-\frac{d+2}{d-2}W^{\frac{4}{d-2}}$, $W^{\frac{d+2}{d-2}}=(-\Delta) W$ and $\Delta \Lambda W = V\Lambda W$ 
 that 
\begin{equation}\label{20/12/19/16:29}
\langle W, V \Lambda W \rangle 
=
-\frac{d+2}{d-2}\langle W^{\frac{d+2}{d-2}} , \Lambda W \rangle 
=
\frac{d+2}{d-2} \langle  \Delta W, \Lambda W \rangle
=
\frac{d+2}{d-2} \langle W, V \Lambda W  \rangle,
\end{equation}
which implies $\langle W, V \Lambda W \rangle=0$. Thus, we see that   $W \in X_{q}$ for all $\frac{d}{d-2}<q\le \infty$.

\end{enumerate}  
\end{remark}

We will prove Proposition \ref{18/11/17/07:17} in Section \ref{18/11/11/10:47}. 

%%%%%%%%%%%%%%%%%%%%%%%%%%%%%%%%%%%%%%%%%%%%%%%%%%%%%%%%%%%%%%%%%%%%%%%%%%%%%%%

\subsection{Preliminaries}\label{18/11/11/10:42}

In this subsection, we give basic estimates for free and perturbed resolvents. 
 
Let us begin by recalling the following fact (see, e.g., Section 5.2 of \cite{DM}, and Proposition 2.2 of \cite{CMR}): 
\begin{lemma}\label{18/11/14/10:43}
Assume $d\ge 3$. Then, there exists $e_{0}>0$ such that $-e_{0}$ is the only one negative eigenvalue of the operator $-\Delta +V \colon H^{2}(\mathbb{R}^{d}) \to L^{2}(\mathbb{R}^{d})$. Furthermore, the essential spectrum of $-\Delta+V$ 
 equals $[0,\infty)$ (hence $\sigma(-\Delta +V)=\{-e_{0}\}\cup [0,\infty)$), and 
\begin{equation}\label{18/11/11/11:01}
\{ u \in \dot{H}^{1}(\mathbb{R}^{d}) \colon \mbox{$(-\Delta +V)u =0$ in the distribution sense}\} ={\rm span}\,\{\Lambda W, \partial_{1}W, \ldots, \partial_{d}W \},
\end{equation}
where $\partial_{1}W, \ldots \partial_{d}W$ denote the partial derivatives of $W$. 
\end{lemma}

\begin{lemma}\label{18/11/19/09:41}
Assume $d= 3,4$. Let $\frac{d}{d-2} < r \le \infty$. Then, 
\begin{equation}\label{18/11/19/09:43}
\{ u \in L_{\rm rad}^{r}(\mathbb{R}^{d}) \colon \mbox{$(-\Delta +V )u =0$ in the distribution sense} \} ={\rm span}\,\{\Lambda W \}
.
\end{equation}
\end{lemma}
\begin{proof}[Proof of Lemma \ref{18/11/19/09:41}]
Let $u\in L^{r}(\mathbb{R}^{d})$ be a function such that $\Delta u=Vu$ in the distribution sense.  Notice that
$(1+|x|)^{-(d-2)}Vu \in L^{1}(\mathbb{R}^{d})$.  Hence, the formula \eqref{18/11/14/09:07} is available.  Furthermore, we see from \eqref{18/11/14/09:07}, \eqref{18/11/07/14:02}, the Hardy-Littlewood-Sobolev inequality and H\"older's inequality that  
\begin{equation}\label{18/11/19/09:04}
\begin{split}
\|\nabla u\|_{L^{2}}
&=
\|\nabla (-\Delta)^{-1} V u \|_{L^{2}}
\\[6pt]
&\lesssim
\| \int_{\mathbb{R}^{d}} \frac{x-y}{|x-y|^{d}} V(y) u(y) \,dy\|_{L^{2}}
\le 
\| \int_{\mathbb{R}^{d}} |x-y|^{-(d-1)} V(y) |u(y)| \,dy \|_{L^{2}}
\\[6pt]
&\lesssim 
\| V u \|_{L^{\frac{2d}{d+2}}}
\lesssim  
\|u\|_{L^{r}}
,
\end{split} 
\end{equation}
so that $u \in \dot{H}^{1}(\mathbb{R}^{d})$. Thus, we find from Lemma \ref{18/11/14/10:43} that the claim \eqref{18/11/19/09:43} is true. 
\end{proof} 

We record well-known inequalities for radial functions (see \cite{Strauss}): 
\begin{lemma}\label{18/11/16/06:41}
Assume $d\ge 3$, and let $g$ be a radial function on $\mathbb{R}^{d}$. Then, we have 
\begin{align}
\label{13/12/26/21:32}
|x|^{\frac{d-2}{2}}|g(x)| 
&\lesssim 
\|\nabla g \|_{L^{2}},
\\[6pt]
\label{13/12/26/21:32}
|x|^{\frac{d-1}{2}}|g(x)| 
&\lesssim 
\|g \|_{L^{2}}^{\frac{1}{2}}\|\nabla g \|_{L^{2}}^{\frac{1}{2}}
.
\end{align}
\end{lemma}
\begin{remark}\label{18/12/14/16:37}
We see from Lemma \ref{18/11/16/06:41} that for $d=3,4$, 
\begin{equation}\label{18/12/14/16:38}
|x| |g(x)| 
\lesssim 
\|g \|_{L^{2}}^{\frac{4-d}{2}}\|\nabla g \|_{L^{2}}^{\frac{d-2}{2}}
.
\end{equation}
\end{remark}

\begin{lemma}\label{18/11/05/10:29}
Assume $d \ge 1$. 
\begin{enumerate}
\item 
Let $1\le q_{1} \le q_{2} \le \infty$ 
 and $d(\frac{1}{q_{1}}-\frac{1}{q_{2}})<2$. 
 Then, the following holds for all $s>0$: 
\begin{equation}\label{18/11/05/10:41}
\|(-\Delta+s)^{-1} \|_{L^{q_{1}}(\mathbb{R}^{d}) \to L^{q_{2}}(\mathbb{R}^{d})}
\lesssim 
s^{\frac{d}{2}(\frac{1}{q_{1}}-\frac{1}{q_{2}})-1}
,
\end{equation}
where the implicit constant depends only on $d$, $q_{1}$ and $q_{2}$. 
\item 
Let $1< q_{1} \le q_{2} < \infty$ and $d(\frac{1}{q_{1}}-\frac{1}{q_{2}})<2$.
 Then, the following holds for all $s>0$: 
\begin{equation}\label{18/11/24/13:56}
\|(-\Delta+s)^{-1} \|_{L_{\rm weak}^{q_{1}}(\mathbb{R}^{d}) \to L^{q_{2}}(\mathbb{R}^{d})}
\lesssim 
s^{\frac{d}{2}(\frac{1}{q_{1}}-\frac{1}{q_{2}})-1}
,
\end{equation}
where the implicit constant depends only on $d$, $q_{1}$ and $q_{2}$.
\end{enumerate} 
\end{lemma}
\begin{proof}[Proof of Lemma \ref{18/11/05/10:29}]
First, we prove \eqref{18/11/05/10:41}. Let $1\le q_{1} \le q_{2} \le \infty$ and $f \in L^{q_{1}}(\mathbb{R}^{d})$. 
Then, it follows from Young's inequality and elementary computations that   
\begin{equation}\label{18/11/17/09:02}
\begin{split}
\|(-\Delta+s)^{-1}f \|_{L^{q_{2}}}
&\lesssim
\| \int_{0}^{\infty}(4\pi t)^{-\frac{d}{2}}e^{-\frac{|x|^{2}}{4t}}e^{-s t}\,dt
\|_{L^{\frac{q_{1} q_{2}}{q_{1}q_{2}-(q_{2}-q_{1})}}} 
\|f\|_{L^{q_{1}}} 
\\[6pt]
&\lesssim 
\int_{0}^{\infty} t^{-\frac{d}{2}+\frac{d}{2}(1+\frac{1}{q_{2}}-\frac{1}{q_{1}})} e^{-s t}\,dt
\|f\|_{L^{q_{1}}} 
\\[6pt]
&=
s^{\frac{d}{2}(\frac{1}{q_{1}}-\frac{1}{q_{2}})-1}
\int_{0}^{\infty} r^{-\frac{d}{2}(\frac{1}{q_{1}}-\frac{1}{q_{2}})}
e^{-r}\,dr\|f\|_{L^{q_{1}}}
,
\end{split}
\end{equation}
where the implicit constants depend only on $d$, $q_{1}$ and $q_{2}$.  The above right-hand side is finite if and only if 
 $\frac{d}{2}(\frac{1}{q_{1}}-\frac{1}{q_{2}})<1$.  Hence, we see that \eqref{18/11/05/10:41} holds.  Similarly, 
 by the weak Young inequality  (see, e.g., (9) in Section 4.3 of \cite{Lieb-Loss}), we see that  if $1<q_{1}<q_{2}<\infty$ and $\frac{d}{2}(\frac{1}{q_{1}}-\frac{1}{q_{2}})<1$, then \eqref{18/11/24/13:56} holds. 
 \end{proof}
 
Note that Lemma \ref{18/11/05/10:29} does not deal with the case $d(\frac{1}{q_{1}}-\frac{1}{q_{2}})=2$. 
 The following lemma  compensates for Lemma \ref{18/11/05/10:29}:
\begin{lemma}\label{18/11/23/17:17}
Assume $d \ge 3$, and  let $s_{0} \ge 0$. If $\frac{d}{d-2} < q < \infty$, then  
\begin{equation}\label{18/11/23/17:19}
\|(-\Delta+s_{0})^{-1} \|_{L^{\frac{dq}{d+2q}}(\mathbb{R}^{d}) \to L^{q}(\mathbb{R}^{d})}
\lesssim 
1
,
\end{equation}
where the implicit constant depends only on $d$ and $q$. 
\end{lemma}
\begin{remark}\label{20/8/16/12:09}
 Put $r:=\frac{dq}{d+2q}$. Then,  $q=\frac{dr}{d-2r}$; and   the condition $\frac{d}{d-2} < q < \infty$ is equivalent to $1< r <\frac{d}{2}$. 
\end{remark}
\begin{proof}[Proof of Lemma \ref{18/11/23/17:17}]
When $s_{0}=0$,  the claim is identical to \eqref{18/11/23/17:19}.  
 Assume $s_{0}>0$.  Then, by \eqref{20/8/16/15:23} and Hardy-Littlewood-Sobolev's inequality (see \eqref{20/8/16/11:20}), we see that 
\begin{equation}\label{18/11/23/17:18}
\|(-\Delta+s_{0})^{-1}f \|_{L^{q}}
\lesssim  
\| |x|^{-(d-2)}* |f|  \|_{L^{q}}
\lesssim 
\|f\|_{L^{\frac{dq}{d+2q}}}
,
\end{equation}
where the implicit constants depend only on $d$ and $q$.  Thus, we have proved the lemma.  
\end{proof}

\begin{lemma}\label{20/8/8/14:56}
Assume $d\ge 3$. Let $s_{0} \ge 0$ and $0<  \varepsilon \le d-2$. 
Then, the following inequality hods:     
\begin{equation}\label{10/8/8/7:1}
\|(-\Delta+s_{0})^{-1} 
\|_{L^{\frac{d+\varepsilon}{2}} \cap  L^{\frac{d-\varepsilon}{2}} \to L^{\infty}}
\lesssim 
1,
\end{equation}
where the implicit constant depends only on $d$ and $\varepsilon$. 
\end{lemma}
\begin{proof}[Proof of Lemma \ref{20/8/8/14:56}]
Let $f \in L^{\frac{d+\varepsilon}{2}}(\mathbb{R}^{d}) \cap  L^{\frac{d-\varepsilon}{2}}(\mathbb{R}^{d})$.  
 Then, by \eqref{20/8/16/15:23} and H\"older's inequality, we see  that 
\begin{equation}\label{20/8/8/14:59}
\begin{split}
&
\|(-\Delta+s_{0})^{-1} f\|_{L^{\infty}}
\lesssim 
\|  |x|^{-(d-2)} * f \|_{L^{\infty}}
\\[6pt]
&
\le 
\sup_{x\in \mathbb{R}^{d}}
\int_{|y|\le 1} |y|^{-(d-2)}|f(x-y)|\,dy 
+
\sup_{x\in \mathbb{R}^{d}}
\int_{1\le |y|}  |y|^{-(d-2)}|f(x-y)|\,dy
\\[6pt]
&
\lesssim  
\|f\|_{L^{\frac{d+\varepsilon}{2}}}
+ 
\|f\|_{L^{\frac{d-\varepsilon}{2}}}
\lesssim 
\|f\|_{L^{\frac{d+\varepsilon}{2}\cap \frac{d-\varepsilon}{2}}},
\end{split} 
\end{equation}
where the implicit constant depends only on $d$ and $\varepsilon$.  Thus, the claim is true. 
\end{proof}

Let us recall the resolvent equation:
\begin{equation}\label{18/11/11/09:35}
(-\Delta+s)^{-1}-(-\Delta+s_{0})^{-1}
=
-(s-s_{0}) 
(-\Delta+s)^{-1} (-\Delta+s_{0})^{-1}
.
\end{equation}
This identity  holds in the following sense: 
\begin{lemma}\label{20/8/16/15:45}
Assume $d\ge 3$. Let $s>0$ and $s_{0}\ge 0$.
\begin{enumerate}
\item 
Let $\frac{d}{d-2}<q < \infty$.  Then,  the resolvent identity \eqref{18/11/11/09:35} holds as an operator from $L^{\frac{dq}{d+2q}}(\mathbb{R}^{d})$ to $L^{q}(\mathbb{R}^{d})$.
\item  
Let $0< \varepsilon \le d-2$. Then, the resovlent identity \eqref{18/11/11/09:35} holds as an operator from $L^{\frac{d+\varepsilon}{2}}(\mathbb{R}^{d})\cap L^{\frac{d-\varepsilon}{2}}(\mathbb{R}^{d})$ to 
 $L^{\infty}(\mathbb{R}^{d})$.  
\end{enumerate} 
\end{lemma}
\begin{remark}\label{20/9/12/14:9}
Let $\frac{d}{d-2}<q \le \infty$. Then, Lemma \ref{20/8/16/15:45} together with  $|V(x)|\lesssim (1+|x|)^{-4}$ shows that 
 the following identity holds as an operator from $L^{q}(\mathbb{R}^{d})$ 
 to $L^{q}(\mathbb{R}^{d})$:    
\begin{equation}\label{19/01/27/18:15}
1+(-\Delta+s)^{-1}V
=
1+(-\Delta)^{-1}V -s (-\Delta+s)^{-1}(-\Delta)^{-1}V 
.
\end{equation}
\end{remark}
\begin{proof}[Proof of Lemma \ref{20/8/16/15:45}]
By \eqref{18/11/14/09:07}, \eqref{18/11/14/09:06} and \eqref{20/8/13/15:11}, 
 we see that  
\begin{equation}\label{20/8/16/15:47}
\begin{split}
&(-\Delta+s )^{-1}-(-\Delta+s_{0})^{-1}
\\[6pt]
&=
\{(-\Delta+s )^{-1} (-\Delta+s_{0}) -1 \} (-\Delta+s_{0})^{-1}
\\[6pt]
&=
\{(-\Delta+s)^{-1}
(-\Delta + s)- (s-s_{0})  (-\Delta+s)^{-1}  -1 \}
(-\Delta+s_{0})^{-1}
\\[6pt]
&= -(s-s_{0}) (-\Delta+s)^{-1} (-\Delta+s_{0})^{-1}
.
\end{split}
\end{equation}
Note that each identity in \eqref{20/8/16/15:47} can be justified by Lemma \ref{18/11/05/10:29}, Lemma \ref{18/11/23/17:17} and Lemma \ref{20/8/8/14:56}. Thus,  the clamis of the lemma are true. 
\end{proof} 

\begin{lemma}\label{18/11/13/11:08}
Assume $3\le d \le 7$.  Let $2\le  q \le \infty$ and $s>0$. 
Then, the following inequality hods:  
\begin{equation}\label{18/11/13/11:09}
\| (1-\Delta) (-\Delta+s)^{-1} V  \|_{L^{q} \to L^{2} \cap L^{q}}
\lesssim
1+|1-s| s^{-1}
, 
\end{equation}
where the implicit constant depends only on $d$ and $q$. 
 In particular,  if  $f\in L^{q}(\mathbb{R}^{d})$, then  
\begin{equation}\label{18/11/14/10:50}
(-\Delta +s)^{-1} V f \in H^{2}(\mathbb{R}^{d}).
\end{equation} 
\end{lemma}
\begin{proof}[Proof of Lemma \ref{18/11/13/11:08}]
Let $f\in L^{q}(\mathbb{R}^{d})$.  
Observe from $|V(x)| \lesssim (1+|x|)^{-4}$ (see \eqref{20/9/13/9:58}) that if $3\le d \le 7$, then 
\begin{equation}\label{20/11/27/9:24}
\|Vf\|_{L^{2}\cap L^{q}} 
\lesssim \|f\|_{L^{q}}
,
\end{equation}   
where the implicit constant depends only on $d$ and $q$.  
 Then, by \eqref{18/11/14/09:06},  Lemma \ref{18/11/05/10:29} and \eqref{20/11/27/9:24}, 
 we see that 
 \begin{equation}\label{18/11/13/11:14}
\begin{split}
&
\| (1-\Delta)  (-\Delta+s)^{-1} V f \|_{L^{2} \cap L^{q}}
=
\| \big\{ (-\Delta+s)+ (1-s)  \big\} 
(-\Delta+s)^{-1} 
V f \|_{L^{2}\cap L^{q}}
\\[6pt]
&\le 
\| V f \|_{L^{2}\cap L^{q}}
+
|1-s| \| (-\Delta+s)^{-1} V f  \|_{L^{2}\cap L^{q}}
\\[6pt]
&\lesssim  
\| V f \|_{L^{2}\cap L^{q}}
+
| 1-s | s^{-1} \| V f \|_{L^{2}\cap L^{q}}
\lesssim 
( 1+ |1-s| s^{-1} ) 
\|f\|_{L^{q}}
,
\end{split}
\end{equation}
where the implicit constants depend only on $d$ and $q$.  
 Thus, we have obtained the desired estimate. 
\end{proof}

\begin{lemma}\label{18/11/14/11:15}
Assume $3\le d\le 5$, and let $\frac{d}{d-2}< q  \le \infty$. 
 Then, $(-\Delta)^{-1}V$ is a compact operator from $X_{q}$ to itself. 
 In particular, $(-\Delta)^{-1}V \colon X_{q}\to X_{q}$ is bounded.    
\end{lemma}
\begin{proof}[Proof of Lemma \ref{18/11/14/11:15}]
We prove the lemma in two steps. 
\\ 
\noindent 
{\bf Step 1.}~We shall show that for any $f\in L^{q}(\mathbb{R}^{d})$, 
\begin{align}
\label{18/11/13/14:47}
\|(-\Delta)^{-1}V f\|_{L^{q}}
&\lesssim 
\|f\|_{L^{q}},
\\[6pt]
\label{18/11/15/16:27}
\|\nabla (-\Delta)^{-1}V f\|_{L^{q}}
&\lesssim 
\|f\|_{L^{q}},
\end{align}
where the implicit constants depend only on $d$ and $q$. 

Assume $\frac{d}{d-2}<q <\infty$, and let $f \in L^{q}(\mathbb{R}^{d})$.   
 Then, applying \eqref{20/8/16/11:20} as $r=\frac{dq}{2q+d}$,  we see that 
  that   
\begin{equation}\label{20/8/14/15:17}
\|(-\Delta)^{-1}V f\|_{L^{q}}
\lesssim 
\| V f\|_{L^{\frac{dq}{2q+d}}}
\lesssim 
\|f\|_{L^{q}},
\end{equation}
where the implicit constants depend only on $d$ and $q$.  
 Furthermore,  by \eqref{18/11/07/14:02} and the Hardy-Littlewood-Sobolev inequality,
 we see that 
\begin{equation}\label{20/8/14/15:9}
\|\nabla (-\Delta)^{-1}V f\|_{L^{q}}
\lesssim 
\| \int_{\mathbb{R}^{d}} |x-y|^{-(d-1)} V(y) |f(y)|\,dy \|_{L^{q}}
\lesssim 
\| V f\|_{L^{\frac{dq}{q+d}}}
\lesssim 
\|f\|_{L^{q}},
\end{equation}
where the implicit constants depend only on $d$ and $q$.
 
Next, assume $q=\infty$, and let $f \in L^{\infty}(\mathbb{R}^{d})$.  Then, Lemma \ref{20/8/8/14:56} shows that  
\begin{equation}\label{20/8/12/11:27}
\|(-\Delta)^{-1}V f\|_{L^{\infty}}
\lesssim 
\| V f\|_{L^{\frac{d-1}{2}}\cap L^{\frac{d+1}{2}}}
\lesssim  
\| V \|_{L^{\frac{d-1}{2}}\cap L^{\frac{d+1}{2}}} 
\|f\|_{L^{\infty}}
\lesssim 
\|f\|_{L^{\infty}},
\end{equation}
where the implicit constants depend only on $d$.  
 Furthermore,  by \eqref{18/11/07/14:02} and H\"odler's inequality, 
 we see that   
\begin{equation}\label{20/8/14/15:14}
\begin{split}
&\|\nabla (-\Delta)^{-1}V f\|_{L^{\infty}}
\lesssim 
\| \int_{\mathbb{R}^{d}} |y|^{-(d-1)} V(x-y) \,dy \|_{L^{\infty}}
\|f\|_{L^{\infty}}
\\[6pt]
&\lesssim 
\Big\{
\int_{|y|\le 1} |y|^{-(d-1)} \,dy  
+
\Big(
\int_{1\le |y|} |y|^{-\frac{d(d-1)}{d-2}}\,dy 
\Big)^{\frac{d-2}{d}}  \|V\|_{L^{\frac{d}{2}}} 
\Big\} 
\|f \|_{L^{\infty}}
\lesssim 
\|f\|_{L^{\infty}}
,
\end{split} 
\end{equation}
where the implicit constants depend only on $d$.  
 Thus, we have proved \eqref{18/11/13/14:47} and \eqref{18/11/15/16:27}. 

\noindent 
{\bf Step 2.}~We shall finish the proof of the lemma. 

We will use Lemma \ref{18/11/10/13:29}.  
 Let $\{f_{n}\}$ be a bounded sequence in $X_{q}$.  
 
Observe from \eqref{18/11/13/14:47} that 
\begin{equation}\label{20/8/14/17:29}
\sup_{n\ge 1} \| (-\Delta)^{-1} V f_{n}\|_{L^{q}}
\lesssim 
\sup_{n\ge 1} \|f_{n}\|_{L^{q}} <\infty. 
\end{equation} 
Furthermore, observe from the fundamental theorem of calculus and \eqref{18/11/15/16:27} that 
\begin{equation}\label{18/11/15/12:01}
\begin{split}
&
\lim_{y \to 0} \sup_{n\ge 1} \| (-\Delta)^{-1} V f_{n} - (-\Delta)^{-1} V f_{n}(\cdot+y) 
\|_{L^{q}}
\\[6pt]
&=
\lim_{y \to 0}\sup_{n\ge 1}
\| 
\int_{0}^{1} \frac{d}{d\theta}  (-\Delta)^{-1} V f_{n}  (\cdot+\theta y)\,d\theta
\|_{L^{q}} 
\\[6pt]
&=
\lim_{y \to 0}\sup_{n\ge 1}
\| 
\int_{0}^{1} \big[y\cdot  \nabla (-\Delta)^{-1} V f_{n} \big] (\cdot+\theta y)\,d\theta 
\|_{L^{q}} 
\\[6pt]
&\le 
\lim_{y \to 0}  
\sup_{n\ge 1} |y| \| \nabla (-\Delta)^{-1} V f_{n} 
\|_{L^{q}}
\lesssim 
\lim_{y \to 0} |y| 
\sup_{n\ge 1}\| f_{n} \|_{L^{q}}
=0.
\end{split}
\end{equation}

Note that  $(-\Delta)^{-1}V f_{n}$ is radial, as so is $f_{n}$.  
 Let $\varepsilon_{0}:= \min\{ \frac{1}{2}(q-\frac{d}{d-2}), 1 \}$. 
 When $q\neq \infty$, Lemma \ref{18/11/16/06:41} shows that 
\begin{equation}\label{18/11/15/12:02}
\begin{split}
&
\lim_{R\to \infty} \sup_{n\ge 1}
\|(-\Delta)^{-1} V f_{n}(x) \|_{L^{q}(|x|\ge R)}
\\[6pt]
&\le 
\lim_{R\to \infty} \sup_{n\ge 1} R^{-\frac{d-2}{2} \frac{\varepsilon_{0}}{d+\varepsilon_{0}}}
\| |x|^{\frac{d-2}{2}} (-\Delta)^{-1} V f_{n}  \|_{L^{\infty}}^{\frac{\varepsilon_{0}}{d+\varepsilon_{0}}} 
\| (-\Delta)^{-1} V f_{n} \|_{L^{\frac{dq}{d+2}}}^{\frac{d}{d+\varepsilon_{0}}}
\\[6pt]
&\lesssim  
\lim_{R\to \infty} \sup_{n\ge 1} 
R^{-\frac{(d-2)\varepsilon_{0}}{2(d+\varepsilon_{0})}}
\| \nabla (-\Delta)^{-1} V f_{n}  \|_{L^{2}}^{\frac{\varepsilon_{0}}{d+\varepsilon_{0}}} 
\| (-\Delta)^{-1} V f_{n} \|_{L^{\frac{dq}{d+\varepsilon_{0}}}}^{\frac{d}{d+\varepsilon_{0}}}
,
\end{split}
\end{equation}
where the implicit constant depends only on $d$ and $q$. Similarly, when $q=\infty$, we see that 
\begin{equation}\label{21/1/10/15:25}
\begin{split}
&
\lim_{R\to \infty} \sup_{n\ge 1}
\|(-\Delta)^{-1} V f_{n}(x) \|_{L^{\infty}(|x|\ge R)}
\\[6pt]
&\le 
\lim_{R\to \infty} \sup_{n\ge 1} 
R^{-\frac{d-2}{2}}
\| |x|^{\frac{d-2}{2}} (-\Delta)^{-1} V f_{n}  \|_{L^{\infty}}
\\[6pt]
&\lesssim  
\lim_{R\to \infty} \sup_{n\ge 1} 
R^{-\frac{d-2}{2}}
\| \nabla (-\Delta)^{-1} V f_{n}  \|_{L^{2}}
,
\end{split}
\end{equation}
where the implicit constant depends only on $d$.

 Observe from $3\le d \le 5$ and $\frac{d}{d-2}<q \le \infty$ that  
 $\frac{2d}{d+2} <q$ and $\frac{2dq}{(d+2)q-2d} > \frac{d}{4}$ ($\frac{2d}{d+2}>\frac{d}{4}$ if $q=\infty$). 
 Hence,  by \eqref{18/11/07/14:02}, the Hardy-Littlewood-Sobolev inequality, H\"older's one,  and $V(x) \sim (1+|x|)^{-4}$,
 we see  that  
 \begin{equation}\label{20/8/15/9:47}
 \begin{split}
\|\nabla (-\Delta)^{-1}V f_{n}\|_{L^{2}}
&\lesssim  
\||x|^{-(d-1)} *  V f_{n} \|_{L^{2}}
\lesssim 
\|V f_{n} \|_{L^{\frac{2d}{d+2}}}
\\[6pt]
&\le 
\|V\|_{L^{\frac{2dq}{(d+2)q -2d}}} \| f_{n}\|_{L^{q}}
\lesssim
\| f_{n}\|_{L^{q}} ,
 \end{split}
\end{equation}
where the implicit constants depend only on $d$ and $q$. 
 Similarly, if $q\neq \infty$, then  
\begin{equation}\label{18/11/16/07:07}
\|(-\Delta)^{-1}V f_{n}\|_{L^{\frac{dq}{d+\varepsilon_{0}}}}
\lesssim 
\| V f_{n}\|_{L^{\frac{dq}{d+2q+\varepsilon_{0}}}}
\le 
\| V \|_{L^{\frac{dq}{2q+\varepsilon_{0}}}} 
\|f_{n}\|_{L^{q}}
\lesssim
\|f_{n}\|_{L^{q}},
\end{equation} 
where the implicit constants depend only on $d$ and $q$.

 Plugging  the estimates  \eqref{20/8/15/9:47} and \eqref{18/11/16/07:07}
into \eqref{18/11/15/12:02} (\eqref{20/8/15/9:47} into \eqref{21/1/10/15:25} if $q=\infty$), we find that 
\begin{equation}\label{20/11/27/11:50}
\lim_{R\to \infty} \sup_{n\ge 1}
\|(-\Delta)^{-1} V f_{n}(x) \|_{L^{q}(|x|\ge R)}
=0
.
\end{equation}
Thus, it follows from Lemma \ref{18/11/10/13:29} that $\{ (-\Delta)^{-1}V f_{n}\}$ contains a convergent subsequence in $L^{q}(\mathbb{R}^{d})$. We denote   the convergent subsequence by the same symbol as the original one, and the limit by $g$. Note that $g \in L_{\rm rad}^{q}(\mathbb{R}^{d})$.   

It remains to prove that the limit $g$ satisfies the orthogonality condition 
$\langle g, V \Lambda W\rangle=0$.  
  Since $V \Lambda W = \Delta \Lambda W $ and $\langle f_{n}, V\Lambda W\rangle =0$, 
 we can verify that for any $n\ge 1$, 
\begin{equation}\label{18/11/13/14:35}
\langle (-\Delta)^{-1}V f_{n} , V \Lambda W \rangle 
=
\langle  (-\Delta)^{-1}V f_{n}, -(-\Delta)\Lambda W
\rangle
=
-
\langle V f_{n}, \Lambda W
\rangle 
=0
,
\end{equation}
which together with $V\Lambda W \in L^{1}(\mathbb{R}^{d})\cap L^{\infty}(\mathbb{R}^{d})$ yields  
\begin{equation}\label{20/8/15/11:46}
\langle g, V\Lambda W \rangle
=
\lim_{n\to \infty}
\langle (-\Delta)^{-1}V f_{n} , V \Lambda W \rangle
=0.
\end{equation}
Thus, we have completed the proof.
\end{proof}

\begin{lemma}\label{18/11/13/14:25}
Assume $3\le d\le 5$ and $\frac{d}{d-2}<q \le \infty$. Then, the inverse of $1+(-\Delta)^{-1}V$ exists as a bounded operator from $X_{q}$ to itself. 
\end{lemma}
\begin{proof}[Proof of Lemma \ref{18/11/13/14:25}]
We see from Lemma \ref{18/11/14/11:15} that $(-\Delta)^{-1} V$ is a compact operator from $X_{q}$ to itself.  
 Furthermore,  by a standard theory of compact operators, if $z\in \mathbb{C}$ is neither zero nor an eigenvalue of $(-\Delta)^{-1} V$, then 
  $z$ belongs to the resolvent. In particular, if $-1$ is not an eigenvalue of $(-\Delta)^{-1}V$, then  the inverse of $1+(-\Delta)^{-1}V$ exists as a bounded operator from  $X_{q}$ to itself. Hence, what we need to prove is that  $-1$ is not an eigenvalue.  Suppose for contradiction that there exists $f \in X_{q}$ such that 
$(-\Delta)^{-1} V f = - f$. Then, we can verify that for any $\phi \in C_{c}^{\infty}(\mathbb{R}^{d})$, 
\begin{equation}\label{18/12/10/21:27}
\langle (-\Delta +V) f, \phi \rangle 
= 
\langle (-\Delta) \{ f + (-\Delta)^{-1} V f \}, \phi \rangle
=
\langle f+(-\Delta)^{-1} V f, -\Delta \phi \rangle 
=0.  
\end{equation}
Lemma \ref{18/11/19/09:41} together with \eqref{18/12/10/21:27} implies that $f =\kappa \Lambda W$ for some $\kappa \in \mathbb{R}$. Furthermore, it follows from $\langle f, V \Lambda W\rangle = 0$ and $\langle \Lambda W, V \Lambda W \rangle \neq 0$ that $\kappa=0$. Thus, $f$ is trivial and therefore $-1$ is not an eigenvalue of $(-\Delta)^{-1}V$.  
\end{proof}

%%%%%%%%%%%%%%

\begin{lemma}\label{20/12/12/9:52}
Assume  $d=3,4$.   Then, the following holds 
 for all $\frac{d}{d-2}<q <\infty$, $f \in L_{\rm rad}^{q}(\mathbb{R}^{d})$ and $s>0$: 
\begin{equation}\label{20/12/12/10:9}
\big| \langle 
(-\Delta + s )^{-1} Vf, |x|^{-(d-2)}   
\rangle -   
\frac{(5-d)\pi^{2-\frac{d}{2}}}{2^{d-2}\Gamma(\frac{d-2}{2})}
\Re\mathcal{F}[Vf](0) \delta(s)^{-1}
\big| 
\lesssim 
\|f\|_{L^{q}},
\end{equation}
where the implicit contant depends only on $d$ and $q$. 
\end{lemma}
\begin{remark}\label{20/9/13/9:1}
When $d=4$,  it is essential that $q<\infty$ in Lemma \ref{20/12/12/9:52}; More precisely, since $V\not \in L^{1}(\mathbb{R}^{4})$, 
 $\mathcal{F}[Vf](0)$ makes no sense for $f\in L_{\rm rad}^{\infty}(\mathbb{R}^{d})$.
\end{remark}
\begin{proof}[Proof of Lemma \ref{20/12/12/9:52}]
Let $\frac{d}{d-2}<q <\infty$,  $f \in L_{\rm rad}^{q}(\mathbb{R}^{d})$ and $s>0$. 
 Then, by \eqref{20/8/9/15:50}, \eqref{18/12/22/10:07},
 \eqref{21/4/6/10:15}, \eqref{18/11/10/16:21}  and \eqref{20/8/18/16:27},  we see  that 
\begin{equation}\label{20/12/12/10:35}
\begin{split}
&
\langle (-\Delta+s)^{-1} V f, |x|^{-(d-2)} \rangle
=
\frac{1}{(2\pi)^{d}}\frac{4 \pi^{\frac{d}{2}}}{\Gamma(\frac{d-2}{2})} 
\langle (|\xi|^{2}+s)^{-1}\mathcal{F}[V f] ,\  |\xi|^{-2} 
\rangle
\\[6pt]
&= 
C_{d}
\Re 
\int_{|\xi|\le 1} 
\frac{\mathcal{F}[V f](0)}{(|\xi|^{2}+s)|\xi|^{2}} 
\,d\xi 
+
C_{d}
\Re\int_{|\xi|\le 1} 
\frac{\mathcal{F}[V f](\xi)-\mathcal{F}[V f](0)}{(|\xi|^{2}+s)|\xi|^{2}}  
\,d\xi 
\\[6pt]
&\quad +
C_{d} \Re\int_{1\le|\xi|} 
\frac{\mathcal{F}[V f](\xi)}{(|\xi|^{2}+s)|\xi|^{2}} 
\,d\xi ,
\end{split}
\end{equation}
where $C_{d}:=\{ 2^{d-2} \pi^{\frac{d}{2}}\Gamma(\frac{d-2}{2})\}^{-1}$.

Consider the first term on the right-hand side of \eqref{20/12/12/10:35}.  
 By \eqref{20/9/30/9:12} and $|V(x)| \lesssim (1+|x|)^{-4}$ (see \eqref{20/9/13/9:58}), we see that 
\begin{equation}\label{20/12/12/11:15}
\begin{split}
&\Big|
C_{d}\Re\int_{|\xi|\le 1} 
\frac{\mathcal{F}[V f](0)}{(|\xi|^{2}+s)|\xi|^{2}} 
\,d\xi -
C_{d} (5-d) \pi^{2} \Re\mathcal{F}[V f](0)
\delta(s)^{-1}
\Big|
\\[6pt]
&\lesssim |\mathcal{F}[V f](0)| \le \|Vf\|_{L^{1}}
\lesssim 
\|f\|_{L^{q}}, 
\end{split} 
\end{equation}
where the implicit constants depend only on $d$ and $q$. 
Next, consider  the second term on the right-hand side of \eqref{20/12/12/10:35}.  
Observe that  
\begin{equation}\label{20/12/12/1}
\int_{|\xi|\le 1} \frac{\xi}{(|\xi|^{2}+s)|\xi|^{2}}\,d\xi
=0.
\end{equation}
Furthermore,  observe that for any $x\in \mathbb{R}^{d}$ and $\xi \in \mathbb{R}^{d}$, 
\begin{equation}\label{20/12/12/2}
\big| e^{-ix\cdot\xi} -1 + ix\cdot \xi \big|
\le 2 \min\{|x||\xi| , \ |x|^{2}|\xi|^{2}\}
.
\end{equation}
Fix a number $\varepsilon_{0}$ such that $4-d< \varepsilon_{0} < 4- \frac{d(q-1)}{q}$; The dependence of a constant on $\varepsilon_{0}$ is absorbed into than on  $d$ and $q$.  Observe from $d=3,4$ and $\frac{d}{d-2}<q<\infty$ that 
 $0 <\varepsilon_{0} < 2$.
 Then, by \eqref{20/12/12/2}, $|V(x)| \lesssim (1+|x|)^{-4}$,  H\"older's inequality, and $\frac{q(4-\varepsilon_{0})}{q-1}>d$, 
 we see that  
\begin{equation}\label{20/12/12/3}
\begin{split} 
&
\big|  
\mathcal{F}[V f](\xi)-\mathcal{F}[V f](0)
-
\xi \cdot \nabla \mathcal{F}[V f](0)
\big| 
\\[6pt]
&=
\Big| 
\int_{\mathbb{R}^{d}} \big( e^{-ix\cdot\xi} -1 + ix\cdot \xi \big) V(x)f(x)\,dx \Big|
\\[6pt]
&\lesssim   
|\xi|^{\varepsilon_{0}} \int_{\mathbb{R}^{d}} 
|x|^{\varepsilon_{0}}(1+|x|)^{-4}|f(x)| 
\,dx
\\[6pt]
&\lesssim  
|\xi|^{\varepsilon_{0}}
\| (1+|x|)^{-(4-\varepsilon_{0})} \|_{L^{\frac{q}{q-1}}}
\|f\|_{L^{q}}
\lesssim 
|\xi|^{\varepsilon_{0}}
\|f\|_{L^{q}}
,
\end{split}
\end{equation}
where the implicit constants depend only on $d$ and $q$.  
 Hence,  by \eqref{20/12/12/1}, \eqref{20/12/12/3} and $d-\varepsilon_{0}<d$, we see  that 
\begin{equation}\label{20/12/12/5}
\begin{split}
&
\Big|  
\int_{|\xi|\le 1} 
\frac{\mathcal{F}[V f](\xi)-\mathcal{F}[V f](0)}{(|\xi|^{2}+s)|\xi|^{2}}  
\,d\xi
\Big| 
\\[6pt]
&=
\Big|  
\int_{|\xi|\le 1} 
\frac{\mathcal{F}[V f](\xi)-\mathcal{F}[V f](0)
-\xi\cdot \nabla \mathcal{F}[Vf](0)}{(|\xi|^{2}+s)|\xi|^{2}}  
\,d\xi
\Big| 
\\[6pt]
&\lesssim 
\|f\|_{L^{q}} 
\int_{|\xi|\le 1} 
\frac{|\xi|^{\varepsilon_{0}}}{(|\xi|^{2}+s)|\xi|^{2}}  
\,d\xi
\le 
\|f\|_{L^{q}} 
\int_{|\xi|\le 1} 
|\xi|^{-(4-\varepsilon_{0})}  
\,d\xi
\lesssim 
\|f\|_{L^{q}}
,
\end{split} 
\end{equation}
where the implicit constants depend only on $d$ and $q$. 

Consider  the last term on the right-hand side of \eqref{20/12/12/10:35}.   
 By \eqref{18/12/14/16:38}, Plancherel's theorem, $|V(x)|\lesssim (1+|x|)^{-4}$ and $q>2$, 
 we see that 
\begin{equation}\label{20/12/12/6}
\begin{split}
\| |\xi| \mathcal{F}[V f ] \|_{L^{\infty}} 
&\lesssim 
\| \mathcal{F}[V f ] \|_{L^{2}}^{\frac{4-d}{2}}
\| \nabla \mathcal{F}[V f ] \|_{L^{2}}^{\frac{d-2}{2}}
=
\| \mathcal{F}[V f ] \|_{L^{2}}^{\frac{4-d}{2}}
\| \mathcal{F}[x V f] \|_{L^{2}}^{\frac{d-2}{2}}
\\[6pt]
&\lesssim 
\| V f  \|_{L^{2}}^{\frac{4-d}{2}}
\| |x| V f  \|_{L^{2}}^{\frac{d-2}{2}}
\lesssim 
\|  f  \|_{L^{q}}
,
\end{split} 
\end{equation}
where the implicit constants depend only on $d$ and $q$.  Using \eqref{20/12/12/6}, we see  that  
\begin{equation}\label{20/12/12/7}
\Big|  
\int_{1\le |\xi|} 
\frac{\mathcal{F}[V f](\xi)}{(|\xi|^{2}+s)|\xi|^{2}}  
\,d\xi
\Big| 
\lesssim 
\|  f  \|_{L^{q}}\int_{1\le |\xi|} 
\frac{1}{|\xi|^{5}}  
\,d\xi
\lesssim \|  f  \|_{L^{q}},
\end{equation}
where the implicit constants depend only on $d$ and $q$. 

Now, combining \eqref{20/12/12/10:35}, \eqref{20/12/12/11:15},  \eqref{20/12/12/5} and \eqref{20/12/12/7}, 
 we find that \eqref{20/12/12/10:9} holds.  
\end{proof}

%%%%%%%%%%%%%%

\begin{lemma}\label{21/1/10/15:37}
Assume $d=3,4$. Then, the following holds for all $s>0$: 
\begin{equation}\label{21/1/10/7:15}
\| (-\Delta +s)^{-1} \Lambda W \|_{L^{\infty}}
\lesssim 
1+\delta(s)^{-1}, 
\end{equation}
where the implicit constant depends only on $d$.
\end{lemma}
\begin{proof}[Proof of Lemma \ref{21/1/10/15:37}]
Observe from \eqref{20/8/18/16:27} and $-\Lambda W = (-\Delta)^{-1} V \Lambda W$ (see \eqref{20/8/15/15:47}) that 
\begin{equation}\label{21/1/10/7:16} 
\begin{split}
&\| (-\Delta +s)^{-1} \Lambda W \|_{L^{\infty}}
\sim 
\| \int_{\mathbb{R}^{d}} e^{ix\cdot \xi} 
 \frac{\mathcal{F}[ V \Lambda W](\xi)}{(|\xi|^{2}+s)|\xi|^{2}}\,d\xi 
\|_{L_{x}^{\infty}}
\\[6pt]
&\le 
\int_{|\xi|\le 1} \frac{ | \mathcal{F}[ V \Lambda W](\xi)|}{(|\xi|^{2}+s)|\xi|^{2}}\,d\xi 
+
\int_{1\le |\xi|} \frac{ |\mathcal{F}[ V \Lambda W](\xi)|}{(|\xi|^{2}+s)|\xi|^{2}}\,d\xi 
,
\end{split}
\end{equation}
where the implicit constant depends only on $d$.

Consider the first term on the right-hand side of \eqref{21/1/10/7:16}.  
 By \eqref{20/9/30/9:12}, we see that 
\begin{equation}\label{21/1/10/15:42}
\int_{|\xi|\le 1} 
\frac{\mathcal{F}[ V \Lambda W](\xi)}{(|\xi|^{2}+s)|\xi|^{2}}\,d\xi 
\lesssim \{\delta(s)^{-1}+1\}
\|V\Lambda W\|_{L^{1}}
\lesssim \delta(s)^{-1}+1,
\end{equation}
where the implicit constants depend only on $d$.  
 Move on to the second term on the right-hand side of \eqref{21/1/10/7:16}. 
 By \eqref{18/12/14/16:38} and Plancherel's theorem, we see that 
\begin{equation}\label{21/1/10/15:47}
\begin{split}
&\| |\xi| \mathcal{F}[V \Lambda W ] \|_{L^{\infty}} 
\lesssim 
\| \mathcal{F}[V \Lambda W ] \|_{L^{2}}^{\frac{4-d}{2}}
\| \nabla \mathcal{F}[V \Lambda W ] \|_{L^{2}}^{\frac{d-2}{2}}
\\[6pt]
&=
\| \mathcal{F}[V \Lambda W ] \|_{L^{2}}^{\frac{4-d}{2}}
\| \mathcal{F}[x V \Lambda W] \|_{L^{2}}^{\frac{d-2}{2}}
\lesssim 
\| V \Lambda W  \|_{L^{2}}^{\frac{4-d}{2}}
\| |x| V \Lambda W  \|_{L^{2}}^{\frac{d-2}{2}}
\lesssim 1
,
\end{split} 
\end{equation}
where the implicit constants depend only on $d$. Hence, we find from 
 \eqref{21/1/10/15:47} and $d=3,4$ that  
\begin{equation}\label{21/1/10/15:48}
\int_{1\le |\xi|} 
\frac{ | \mathcal{F}[V \Lambda W](\xi)|}{(|\xi|^{2}+s)|\xi|^{2}}  
\,d\xi
\lesssim 
\int_{1\le |\xi|} 
\frac{1}{|\xi|^{5}}  
\,d\xi
\lesssim 1, 
\end{equation}
where the implicit constants depend only on $d$. Plugging  \eqref{21/1/10/15:42} and \eqref{21/1/10/15:48} into \eqref{21/1/10/7:16}, we obtain the desired estimate \eqref{21/1/10/7:15}. 
\end{proof}

%%%%%%%%%%%%%%

\begin{lemma}\label{18/12/02/11:25}
Assume $d=3,4$. 
 Then, there exists a constant $\mathscr{C}>0$ depending only on $d$ such that 
the following holds  for all $\frac{d}{d-2}<q <\infty$,   $g \in L_{\rm rad}^{q}(\mathbb{R}^{d})$ and $s>0$: 
\begin{equation}\label{18/12/02/11:26}
\big|
\langle (-\Delta+s)^{-1} Vg , \Lambda W \rangle
+ 
\mathscr{C}
\Re\mathcal{F}[Vg](0) \delta(s)^{-1} 
\big|
\lesssim 
 \|  g  \|_{L^{q}}
,
\end{equation}
where the implicit constant depends only on $d$ and $q$.
\end{lemma}
\begin{proof}[Proof of Lemma \ref{18/12/02/11:25}] 
Decompose $\Lambda W$ as 
\begin{equation}\label{18/12/03/10:17}
\Lambda W (x)
= 
-\frac{|x|^{2}}{2d}
\Big(\frac{|x|^{2}}{d(d-2)} \Big)^{-\frac{d}{2}}
+
\frac{d-2}{2}\Big(1+\frac{|x|^{2}}{d(d-2)} \Big)^{-\frac{d}{2}}
+
Z(x),
\end{equation}
where 
\begin{equation}\label{18/12/04/09:43}
Z(x):=
-
\frac{|x|^{2}}{2d}
\Big(1+\frac{|x|^{2}}{d(d-2)} \Big)^{-\frac{d}{2}}
+
\frac{|x|^{2}}{2d}
\Big(\frac{|x|^{2}}{d(d-2)} \Big)^{-\frac{d}{2}}
.
\end{equation}
Let $\frac{d}{d-2}<q <\infty$,   $g \in L_{\rm rad}^{q}(\mathbb{R}^{d})$ and $s>0$. 
 Then,  it follows from \eqref{18/12/03/10:17} that      
\begin{equation}\label{18/12/03/16:02}
\begin{split}
\langle (-\Delta+s)^{-1} V g, \Lambda W \rangle
&= 
-\frac{\{d(d-2)\}^{\frac{d}{2}}}{2d}
\langle (-\Delta+s)^{-1} V g ,\  
|x|^{-(d-2)} \rangle
\\[6pt]
&\quad +
\frac{d-2}{2}\langle (-\Delta+s)^{-1} V g, \  
\Big(1+\frac{|x|^{2}}{d(d-2)} \Big)^{-\frac{d}{2}}
 \rangle
\\[6pt]
&\quad + 
\langle (-\Delta+s)^{-1} V g, Z \rangle
.
\end{split}
\end{equation}
Consider the first term on the right-hand side of \eqref{18/12/03/16:02}. 
 Lemma \ref{20/12/12/9:52} shows that there exists $A>0$ depending only on $d$ 
 such that    
\begin{equation}\label{18/12/04/11:19}
\Big| 
 -\frac{\{d(d-2)\}^{\frac{d}{2}}}{2d}
 \langle (-\Delta+s)^{-1} V g, |x|^{-(d-2)} \rangle
 +
A
\Re\mathcal{F}[V g](0) \delta(s)^{-1}
\Big| 
\lesssim \|g\|_{L^{q}},
\end{equation}
where the implicit constant depends only on $d$ and $q$. 
 Consider the second term on the right-hand side of \eqref{18/12/03/16:02}.
  Observe from $d=3,4$ that 
 $\frac{4d}{2d-5}>\frac{d}{d-2}$, 
 $\frac{4d}{2d+5}>1$ and 
 $\frac{4d}{2d+3}<\frac{d}{d-2}<q$. Then,  
 by H\"older's inequality and Lemma \ref{18/11/23/17:17}, 
 we see that  
\begin{equation}\label{18/12/02/14:01}
\begin{split}
&\Big|
\frac{d-2}{2}\langle (-\Delta+s)^{-1} V g ,\  
\Big(1+\frac{|x|^{2}}{d(d-2)} \Big)^{-\frac{d}{2}}
 \rangle
\Big|
\\[6pt]
&\lesssim  
\|(-\Delta+s)^{-1} V g \|_{L^{\frac{4d}{2d-5}}} 
\| (1+|x|)^{-d}\|_{L^{\frac{4d}{2d+5}}}
\lesssim 
\|V g  \|_{L^{\frac{4d}{2d+3}}} 
\lesssim
\|g  \|_{L^{q}} 
,
\end{split} 
\end{equation}
where the implicit constants depend only on $d$ and $q$.
It remains to estimate the last term on the right-hand side of \eqref{18/12/03/16:02}.
 Observe from the fundamental theorem of calculus,  $\frac{4d}{2d+5}(d-1) <d$ and  
 $\frac{4d}{2d+5}>1$ that 
\begin{equation}\label{18/12/04/08:55}
\begin{split}
\|Z\|_{L^{\frac{4d}{2d+5}}}
&=
\Big\| \frac{|x|^{2}}{4}
\int_{0}^{1} 
\Big(\theta+\frac{|x|^{2}}{d(d-2)}\Big)^{-\frac{d}{2}-1} \,d \theta
\Big\|_{L^{\frac{4d}{2d+5}}}
\\[6pt]
&\lesssim  
\int_{0}^{1} \theta^{-\frac{1}{2}} \,d\theta
\big\| 
|x|^{-(d-1)} 
\big\|_{L^{\frac{4d}{2d+5}}(|x|\le 1)}
+
\big\| |x|^{-d}
\big\|_{L^{\frac{4d}{2d+5}}(1\le |x|)}
\lesssim 1,
\end{split} 
\end{equation}
where the implicit constants depend only on $d$. 
 Then, by \eqref{18/12/04/08:55} and the same computation as \eqref{18/12/02/14:01}, 
 we see that  
\begin{equation}\label{18/12/04/11:42}
\big|
\langle (-\Delta+s)^{-1} V g , Z  \rangle
\big|
\lesssim  
\|(-\Delta+s)^{-1} V g \|_{L^{\frac{4d}{2d-5}}} 
\| Z \|_{L^{\frac{4d}{2d+5}}}
\lesssim
\|g  \|_{L^{q}} , 
\end{equation}
where the implicit constants depend only on $d$ and $q$. 

Putting \eqref{18/12/03/16:02}, \eqref{18/12/04/11:19}, \eqref{18/12/02/14:01} and \eqref{18/12/04/11:42} together, 
 we find that \eqref{18/12/02/11:26} holds. 
\end{proof}

%%%%%%%%%%%%%%%%%%%

The following lemma can be proved  in a way similar to Lemma \ref{18/12/02/11:25}:   
\begin{lemma}[cf. Lemma 2.5 of \cite{CG}] \label{20/10/15/11:51}
Assume $d=3,4$.  Then, there exists a constant $\mathscr{C}>0$ dependeing only on $d$ such that the following holds for all  $\frac{d}{d-2}<q<\infty$, $g \in L_{\rm rad}^{q}(\mathbb{R}^{d})$ and $s>0$:  
\begin{equation}\label{20/11/15/1}
\big|
\langle 
(-\Delta +s)^{-1} Vg 
, 
W 
\rangle 
-
\mathscr{C}
\Re\mathcal{F}[Vg](0)
 \delta(s)^{-1} 
\big|
\lesssim 
\|g\|_{L^{q}},
\end{equation}
where the implicit constant depends only on $d$ and $q$. 
\end{lemma}
\begin{proof}[Proof of Lemma \ref{20/10/15/11:51}]
Rewrite $W$ as 
\begin{equation}\label{18/12/19/01:03}
W(x)
= 
\Big(\frac{|x|^{2}}{d(d-2)} \Big)^{\frac{-(d-2)}{2}}
+
Z(x),
\end{equation}
where 
\begin{equation}\label{18/12/19/10:15}
Z(x)
:=
\Big(1+\frac{|x|^{2}}{d(d-2)} \Big)^{\frac{-(d-2)}{2}}
-
\Big(\frac{|x|^{2}}{d(d-2)} \Big)^{\frac{-(d-2)}{2}}
.
\end{equation}
Then, it is easy to see that  the same proof as Lemma \ref{18/12/02/11:25} is  applicable.  
\end{proof}

%%%%%%%%%%%%%%%%%%%%%%%%%

\begin{lemma}\label{20/12/22/13:55}
Assume $d=3,4$.  Then,  there exists a constant $A_{1}>0$ depending only on $d$ such that 
the following holds for all $s>0$:   
\begin{equation}\label{20/12/22/13:58}
\big|
\langle 
(-\Delta +s)^{-1} W 
, 
V\Lambda W 
\rangle 
-A_{1}  \delta(s)^{-1} 
\big|
\lesssim 1
,
\end{equation}
where the implicit constant depends only on $d$. 
\end{lemma}
\begin{proof}[Proof of Lemma \ref{20/12/22/13:55}]
Observe from \eqref{20/12/22/14:18} that
\begin{equation}\label{20/12/22/14}
\Re\mathcal{F}[V\Lambda W](0)
=
-\frac{d+2}{d-2}\langle W^{\frac{4}{d-2}}, \Lambda W \rangle 
=
\frac{d-2}{2}\|W\|_{L^{\frac{d+2}{d-2}}}^{\frac{d+2}{d-2}}
.  
\end{equation} 
Then,  applying Lemma \ref{20/10/15/11:51} as $g=\Lambda W$ and $q=2^{*}$,   
we find that the claim of the lemma is true. 
\end{proof}

%%%%%%%%%%%%%%%%%%%%%%%%%

\begin{lemma}[cf. Lemma 2.5 of \cite{CG}]\label{18/12/19/01:00}
Assume $d=3,4$ and  $\frac{2}{d-2} < p <\frac{d+2}{d-2}$.  
 Then,  the following holds for all $s>0$ and $\max\{1, \frac{d}{(d-2)p} \} < r <\frac{d}{2}$:  
\begin{equation}\label{18/12/20/20:57}
\big|
\langle (-\Delta+s )^{-1} W^{p} , V\Lambda W \rangle
-
K_{p}
\big|
\lesssim 
s^{\frac{d}{2r}-1}
,
\end{equation}
where the implicit constant depends only on $d$, $p$ and $r$, and 
\begin{equation}\label{18/12/20/20:32}
K_{p}:=-\langle W^{p}, \Lambda W \rangle 
= 
\frac{4-(d-2)(p-1)}{2(p+1)}\|W\|_{L^{p+1}}^{p+1}
.
\end{equation}  
\end{lemma}
\begin{proof}[Proof of Lemma \ref{18/12/19/01:00}] 
Let $s>0$. 
 Observe from $V \Lambda W=\Delta W$ (see \eqref{20/9/13/9:52}) that 
\begin{equation}\label{18/12/20/20:25}
\begin{split}
\langle (-\Delta+s)^{-1} W^{p} , V\Lambda W \rangle
&=
\langle  W^{p}, (-\Delta+s)^{-1} \{ -(-\Delta +s) + s \}\Lambda W \rangle
\\[6pt]
&=
-\langle W^{p}, \Lambda W \rangle 
+ 
s \langle  W^{p}, (-\Delta+s)^{-1} \Lambda W \rangle
.
\end{split}
\end{equation}
Furthermore, by H\"older's inequality and Lemma \ref{18/11/05/10:29}, we see that
for any $\max\{ 1, \frac{d}{(d-2)p}\} < r <\frac{d}{2}$, 
\begin{equation}\label{18/12/25/10:28}
\begin{split}
s
\big| 
\langle  W^{p}, (-\Delta+s)^{-1} \Lambda W \rangle
\big|
&\le 
s  
\|W^{p} \|_{L^{r}} 
\|(-\Delta+s)^{-1}  \Lambda W \|_{L^{\frac{r}{r-1}}}
\\[6pt]
&\lesssim
s^{\frac{d-2}{2}-\frac{d(r-1)}{2r}}
\|\Lambda W \|_{L_{\rm weak}^{\frac{d}{d-2}}}
\lesssim 
s^{\frac{d}{2r}-1},
\end{split}
\end{equation}
where the implicit constants depend only on $d$, $p$ and $r$.  
 Putting \eqref{18/12/20/20:25} and \eqref{18/12/25/10:28} together, 
 and using \eqref{20/12/22/14:18}, we obtain the desired estimate \eqref{18/12/20/20:57}. 
\end{proof}

%%%%%%%%%%%%%%%%%%%%%%%%%%%%%%%%%%%%%%%%%%%%%%%%%%%%%%%%%%%%%%%%%%%%%%%%%%%%%%%%

\subsection{Proof of Proposition \ref{18/11/17/07:17}}
\label{18/11/11/10:47}

In this section, we give a proof of Proposition \ref{18/11/17/07:17}.  Assume $d= 3,4$, and let $\frac{d}{d-2}<q \le \infty$.  

We introduce operators $Q$ and $\Pi$ from $L^{q}(\mathbb{R}^{d})$ to itself as 
\begin{equation}\label{19/01/27/15:38}
Qf:=
\frac{\langle f, V \Lambda W \rangle}{\langle \Lambda W, V \Lambda W \rangle}\Lambda W,
\quad 
\Pi f:=
\frac{\langle f, V \Lambda W \rangle}{\|V \Lambda W \|_{L^{2}}^{2}}
V\Lambda W
.
\end{equation} 
Observe that $Q^{2}=Q$, $\Pi^{2}=\Pi$, and  
\begin{equation}\label{19/01/28/15:30}
\|Q\|_{L^{q}\to L^{q}}\lesssim 1,
\end{equation}
where the implicit constant depends only on $d$ and $q$. 
 Furthermore,  it follows from $(-\Delta )\Lambda W = -V\Lambda W$ that 
\begin{align}
\label{19/01/28/09:53}
&\{ 1+(-\Delta)^{-1}V \}Q =0,
\\[6pt]
\label{19/01/28/16:25}
&\Pi \{ 1+(-\Delta)^{-1}V\} =0
.
\end{align}

\begin{lemma}\label{19/01/31/14:17}
Assume $d= 3,4$,  and let $\frac{d}{d-2} < q \le \infty$.  Then, the following hold:
\begin{align*}
X_{q}&=(1-Q)L_{\rm rad}^{q}(\mathbb{R}^{d})
:=
\{ (1-Q)f\colon f \in L_{\rm rad}^{q}(\mathbb{R}^{d})\},
\\[6pt]
X_{q}&=(1-\Pi)L_{\rm rad}^{q}(\mathbb{R}^{d})
:=
\{ (1-\Pi) f \colon f \in L_{\rm rad}^{q}(\mathbb{R}^{d})\}.
\end{align*}
Furthermore, $(1-Q)|_{X_{q}}$ and $(1-\Pi)|_{X_{q}}$ are the identity maps, namely  
\begin{equation}\label{20/9/12/16:45}
(1-Q)f =(1 - \Pi ) f =f 
\quad 
\mbox{for all $f\in X_{q}$}
. 
\end{equation}
\end{lemma}
\begin{proof}[Proof of Lemma \ref{19/01/31/14:17}]
Observe that for any $f \in L^{q}(\mathbb{R}^{d})$,
\begin{equation}\label{19/01/27/16:25}
\langle (1-Q) f, V\Lambda W \rangle=0 .
\end{equation}
Hence, $(1-Q)L_{\rm rad}^{q}(\mathbb{R}^{d})\subset X_{r}$. On the other hand, if $f\in X_{q}$, then $\langle f , V\Lambda W \rangle=0$ and therefore 
 $f=Qf+(1-Q)f=(1-Q)f$. Thus, $X_{q} \subset (1-Q)L_{\rm rad}^{q}(\mathbb{R}^{d})$. Similarly, we can prove $(1-\Pi)L_{\rm rad}^{q}(\mathbb{R}^{d})=X_{q}$.
\end{proof}

Next, for $d\ge 3$ and $\frac{d}{d-2}<q \le \infty$, we introduce linear subspaces $\mathbf{X}_{q}$ and $\mathbf{Y}_{q}$ of $L^{q}(\mathbb{R}^{d})\times L^{q}(\mathbb{R}^{d})$ equipped with the norms $\|\cdot\|_{\mathbf{X}_{q}}$ and $\|\cdot \|_{\mathbf{Y}_{q}}$ as 
\begin{align}
\label{19/01/30/13:41}
\mathbf{X}_{q} &:=X_{q} \times Q L_{\rm rad}^{q}(\mathbb{R}^{d}),
\qquad 
\|(f_{1},f_{2})\|_{\mathbf{X}_{q}}
:=
\|f_{1}\|_{L^{q}}+\|f_{2}\|_{L^{q}}
,
\\[6pt]
\label{19/02/01/10:45}
\mathbf{Y}_{q}&:=X_{q} \times \Pi L_{\rm rad}^{q}(\mathbb{R}^{d}),
\qquad 
\|(f_{1},f_{2})\|_{\mathbf{Y}_{q}}
:=
\|f_{1}\|_{L^{q}}+\|f_{2}\|_{L^{q}}
.
\end{align}
Furthermore, for $\varepsilon >0$,  we define operators $B_{\varepsilon} \colon \mathbf{X}_{q} \to L_{\rm rad}^{q}(\mathbb{R}^{d})$ and $\mathbf{C} \colon  L_{\rm rad}^{q}(\mathbb{R}^{d}) \to 
 \mathbf{Y}_{q}$ by 
\begin{align}
\label{19/01/30/16:03}
B_{\varepsilon}(f_{1},f_{2})
&:=\varepsilon f_{1} + f_{2} ,
\qquad 
 \mbox{for all $(f_{1}, f_{2})\in \mathbf{X}_{q}$}, 
\\[6pt] 
\label{19/01/30/17:02}
\mathbf{C} f&:=((1-Q)f,~\Pi f)
\qquad 
\mbox{for all $f\in L_{\rm rad}^{q}(\mathbb{R}^{d})$} 
.
\end{align}

\begin{lemma}\label{19/01/28/16:00}
Assume $d= 3,4$. Let $\frac{d}{d-2} < q \le \infty$. 
 Then, the following hold:
\begin{enumerate}
\item 
For any $\varepsilon>0$, $B_{\varepsilon} \colon \mathbf{X}_{q} \to L_{\rm rad}^{q}(\mathbb{R}^{d})$ is surjective. 
\item 
$\mathbf{C} \colon  L_{\rm rad}^{q}(\mathbb{R}^{d}) \to  \mathbf{Y}_{q}$ is injective.
\end{enumerate}
\end{lemma}
\begin{proof}[Proof of Lemma \ref{19/01/28/16:00}] 
Let $g \in L_{\rm rad}^{q}(\mathbb{R}^{d})$, and define $f_{1}:=\varepsilon^{-1}(1-Q)g$ and $f_{2}:= Q g$. Then, Lemma \ref{19/01/31/14:17} shows that $(f_{1},f_{2}) \in \mathbf{X}_{q}$. Furthermore, it is obvious that $B_{\varepsilon}(f_{1},f_{2})=g$. Hence, $B_{\varepsilon}$ is 
 surjective. 

Let $f\in L_{\rm rad}^{q}(\mathbb{R}^{d})$. Suppose $\mathbf{C} f=0$, so that $(1-Q)f=0$ and $\Pi f=0$. In particular, we have $f=Qf =c \Lambda W$ for some $c\in \mathbb{R}$.  Furthermore, this together with  $\Pi f =0$ implies $c=0$. Thus, $f=0$ and $\mathbf{C}$ is injective.  
\end{proof}

\begin{lemma}\label{19/01/31/17:09}
Assume $d= 3,4$. Then, all of the following hold: 
\begin{enumerate}
\item 
For any $\frac{d}{d-2}< q < \infty$ and $s>0$, 
\begin{equation}\label{19/01/31/17:41}
s \| (-\Delta+s)^{-1}(-\Delta)^{-1}V Q\|_{L^{q}\to L^{q}}
\lesssim 
s^{\frac{d-2}{2}-\frac{d}{2q}},
\end{equation}
where the implicit constant depends only on $d$ and $q$.

\item 
For any $0< s \le 1$, 
\begin{equation}\label{21/1/11/10:31}
s \| (-\Delta+s)^{-1}(-\Delta)^{-1}V Q\|_{L^{\infty}\to L^{\infty}}
\lesssim 
\delta(s)^{-1} s,
\end{equation}
where the implicit constant depends only on $d$.

\item 
For any $\frac{d}{d-2}< q \le \infty$,  $\frac{d}{d-2}< r < \infty$ and $0< s\le 1$,   
\begin{equation}\label{19/01/31/17:42}
s \| \Pi (-\Delta+s)^{-1}(-\Delta)^{-1}V \|_{L_{\rm rad}^{r} \to \Pi L_{\rm rad}^{q}}
\lesssim 
\delta(s)^{-1} s,
\end{equation} 
where the implicit constant depends only on $d$, $q$ and $r$ (independent of $s$); If $d=3$, then \eqref{19/01/31/17:42} still holds for $r=\infty$. 
\end{enumerate}
\end{lemma}
\begin{proof}[Proof of Lemma \ref{19/01/31/17:09}]
We shall prove \eqref{19/01/31/17:41} and \eqref{21/1/11/10:31}. 
 Let $\frac{d}{d-2}< q \le \infty$, $f \in L^{q}(\mathbb{R}^{d})$ and $s>0$. 
 Observe from $(-\Delta)^{-1}V\Lambda W=-\Lambda W$ (see \eqref{20/8/15/15:47}) and  H\"older's inequality  that 
\begin{equation}\label{19/01/31/17:10}
\begin{split} 
&
s \| (-\Delta+s)^{-1}(-\Delta)^{-1}V Qf \|_{L^{q}}
=  
s \frac{|\langle f, V \Lambda W \rangle|}{|\langle \Lambda W, V \Lambda W \rangle|} \| (-\Delta+s)^{-1}(-\Delta)^{-1}V\Lambda W \|_{L^{q}}
\\[6pt]
&\lesssim 
s \|f\|_{L^{q}} 
\| (-\Delta+\alpha)^{-1}\Lambda W \|_{L^{q}}, 
\end{split}
\end{equation}
where the implicit constant depends only on $d$ and $q$. Assume $q\neq \infty$.
 Then, applying Lemma \ref{18/11/05/10:29} to the right-hand side of \eqref{19/01/31/17:10}, we see that 
\begin{equation}\label{21/1/11/10:57}
s \| (-\Delta+s)^{-1}(-\Delta)^{-1}V Qf \|_{L^{q}}
\lesssim
s^{\frac{d}{2}(\frac{d-2}{d}-\frac{1}{q})} \|f\|_{L^{q}}
\|\Lambda W \|_{L_{\rm weak}^{\frac{d}{d-2}}}
\lesssim
s^{\frac{d-2}{2}-\frac{d}{2q}} \|f\|_{L^{q}},
\end{equation}
where the implicit constants depend only on $d$ and $q$.
 Thus, \eqref{19/01/31/17:41} holds. 
 When $q=\infty$ and $0<s\le 1$, applying  Lemma \ref{21/1/10/15:37} to the right-hand side of \eqref{19/01/31/17:10}, we see that 
\begin{equation}\label{21/1/11/11:3}
s \| (-\Delta+s)^{-1}(-\Delta)^{-1}V Qf \|_{L^{\infty}}
\lesssim
s \delta(s)^{-1} \|f\|_{L^{\infty}},
\end{equation}
where the implicit constant depends only on $d$. 
 Thus, \eqref{21/1/11/10:31} holds. 

Next, we shall prove \eqref{19/01/31/17:42}. 
 Let $g \in L_{\rm rad}^{r}(\mathbb{R}^{d})$. 
 Observe from $|V(x)|\lesssim (1+|x|)^{-4}$ and $r<\infty$ that 
\begin{equation}\label{20/9/12/15:8}
| \mathcal{F}[Vg](0) |
\le 
\|Vg \|_{L^{1}} 
\lesssim \|g\|_{L^{r}} 
 ,
\end{equation}
where the implicit constant depends only on $d$ and $r$; Note here that when $d=3$, \eqref{20/9/12/15:8} still holds for $r=\infty$.
 Then, by $(-\Delta)^{-1}V \Lambda W = -\Lambda W$,  Lemma \ref{18/12/02/11:25} and \eqref{20/9/12/15:8}, we see that 
\begin{equation}\label{19/01/31/17:43}
\begin{split} 
&
s \| \Pi (-\Delta+s)^{-1}(-\Delta)^{-1}V g \|_{L^{q}}
=
s \| \Pi (-\Delta)^{-1} (-\Delta+s)^{-1}V g \|_{L^{q}}
\\[6pt]
&= 
s \frac{| \langle  (-\Delta)^{-1} (-\Delta+s)^{-1}V g,  V \Lambda W \rangle|}{
\|V \Lambda W \|_{L^{2}}^{2}} \| V\Lambda W \|_{L^{q}}
\\[6pt]
&= 
s \frac{| \langle (-\Delta+s)^{-1} V g, \Lambda W \rangle |}{
\|V \Lambda W \|_{L^{2}}^{2}} \| V\Lambda W \|_{L^{q}}
\lesssim
 s \delta(s)^{-1} \| g \|_{L^{r}}  
,
\end{split}
\end{equation}
where the implicit constant depends only on $d$, $q$ and $r$. 
 Thus, \eqref{19/01/31/17:42} holds. 
\end{proof}

Now, we are in a position to prove Proposition \ref{18/11/17/07:17}: 
\begin{proof}[Proof of Proposition \ref{18/11/17/07:17}]
Assume $d=3,4$. Let  $\frac{d}{d-2}< q < \infty$, and let $0< s < 1$ be a small number to be specified later. Furthermore, let  $\varepsilon >0$ be a small constant to be chosen later dependently only on $d$ and $q$.

Observe from Lemma \ref{18/11/05/10:29} that $1+(-\Delta +s)^{-1}V$ maps $L_{\rm rad}^{q}(\mathbb{R}^{d})$ to itself. Recall from Lemma \ref{19/01/28/16:00} that $B_{\varepsilon} \colon \mathbf{X}_{q} \to L_{\rm rad}^{q}(\mathbb{R}^{d})$ is surjective, and $\mathbf{C}\colon L_{\rm rad}^{q}(\mathbb{R}^{d}) \to \mathbf{Y}_{q}$ is injective. 
 Then, we define an operator $\mathbf{A}_{\varepsilon}(s)$ from $\mathbf{X}_{q}$ to $\mathbf{Y}_{q}$ as  
\begin{equation}\label{19/01/30/17:21}
\mathbf{A}_{\varepsilon}(s)
:= \mathbf{C} \{ 1+(-\Delta+s)^{-1}V \} B_{\varepsilon}
.
\end{equation}
By Lemma 3.12 of \cite{Jensen-Kato} (see Lemma \ref{21/1/12/9:48} in Appendix \ref{20/10/14/9:12}), we see that if the inverse of $\mathbf{A}_{\varepsilon}(s)$ exists, then  the operator $1+(-\Delta+s)^{-1}V \colon L_{\rm rad}^{q}(\mathbb{R}^{d}) \to L_{\rm rad}^{q}(\mathbb{R}^{d})$ has the inverse and   
\begin{equation}\label{19/01/30/17:18}
\{ 1+(-\Delta+s)^{-1}V\}^{-1} 
=
B_{\varepsilon} \{\mathbf{A}_{\varepsilon}(s)\}^{-1}\mathbf{C}.
\end{equation}
This fact is the main point of the proof.

Let $(f_{1},f_{2})\in \mathbf{X}_{q}$. 
 We may write  $\mathbf{A}_{\varepsilon}(s) (f_{1},f_{2})$ as follows
 (we do not care about the distinction between column and row vectors):
\begin{equation}\label{19/01/31/13:15}
\mathbf{A}_{\varepsilon}(s) (f_{1},f_{2})
=
(A_{11}f_{1}+ A_{12}f_{2},\ A_{21}f_{1}+A_{22}f_{2})
=
\left( \begin{array}{cc} A_{11}& A_{12} \\ A_{21} & A_{22} 
\end{array}\right)
\left( \begin{array}{c} f_{1} \\ f_{2}  
\end{array}\right),
\end{equation}
where 
\begin{align}
\label{A11}
A_{11}&:=\varepsilon (1-Q) \{1+(-\Delta+s)^{-1}V\}|_{X_{q}},
\\[6pt]
\label{A12} 
A_{12}&:=(1-Q)\{1+(-\Delta+s)^{-1}V\}Q,
\\[6pt]
\label{A21}
A_{21}&:=\varepsilon \Pi \{ 1+(-\Delta+s)^{-1}V\}|_{X_{q}},
\\[6pt]
\label{A22}  
A_{22}&:=\Pi \{ 1+(-\Delta+s )^{-1}V\} Q
.
\end{align}
Recall that $1+(-\Delta+s )^{-1}V \colon L_{\rm rad}^{q}(\mathbb{R}^{d}) \to L_{\rm rad}^{q}(\mathbb{R}^{d})$. 
 Then, observe from Lemma \ref{19/01/31/14:17} and \eqref{19/01/28/15:30} that 
\begin{equation}\label{21/1/12/9:10}
A_{11} \colon X_{q} \to X_{q},
\qquad 
A_{12} \colon QL_{\rm rad}^{q}(\mathbb{R}^{d})\to X_{q}.
\end{equation} 
Furthermore, it is obvious that 
\begin{equation}\label{21/1/12/9:16}
A_{21}\colon X_{q} \to \Pi L_{\rm{rad}}^{q}(\mathbb{R}^{d}),
\qquad 
A_{22} \colon Q L_{\rm{rad}}^{q}(\mathbb{R}^{d}) \to \Pi L_{\rm{rad}}^{q}(\mathbb{R}^{d})
.
\end{equation}

We divide the proof into several steps:
\\
\noindent 
{\bf Step 1.}~We consider the operator $A_{11}\colon X_{q}\to X_{q}$ (see \eqref{A11}). 

Observe from Lemma \ref{18/11/13/14:25} that $\{ 1+(-\Delta)^{-1}V\}|_{X_{q}} \colon X_{q} \to X_{q}$ has the inverse, say $K_{1}$, such that 
\begin{equation}\label{19/01/31/14:14}
\|K_{1} \|_{X_{q}\to X_{q}} \lesssim 1, 
\end{equation}
where the implicit constant depends only on $d$ and $q$. Furthermore, 
  \eqref{20/9/12/16:45} in Lemma \ref{19/01/31/14:17} shows that 
\begin{equation}\label{19/01/31/15:30}
(1-Q) \{ 1+(-\Delta)^{-1}V\}|_{X_{q}}
=
\{ 1+(-\Delta)^{-1}V\}|_{X_{q}}. 
\end{equation}
Then, by \eqref{19/01/27/18:15} and \eqref{19/01/31/15:30} we may rewrite 
 $A_{11}$ as 
\begin{equation}\label{19/01/31/13:52}
A_{11}
=
\varepsilon \{ 1+(-\Delta)^{-1}V\}|_{X_{q}}
+
\varepsilon S_{11}
=
\varepsilon \{1+(-\Delta)^{-1}V\}|_{X_{q}}
\{ 1+K_{1}S_{11} \},  
\end{equation}
where 
\begin{equation}\label{19/01/31/14:12}
S_{11}:= -s (1-Q) (-\Delta+s)^{-1} (-\Delta)^{-1}V |_{X_{q}}. 
\end{equation}

Now, fix $1< r_{0} <\frac{dq}{d+2q}$; The dependence of a constant on $r_{0}$ can be absorbed into that on $d$ and $q$. Note that $q>\frac{dr_{0}}{d-2r_{0}}$. Then, by \eqref{19/01/31/14:14}, \eqref{19/01/28/15:30},  Lemma \ref{18/11/05/10:29} and \eqref{20/8/16/11:20}, we see that  for any $f\in X_{q}$, 
\begin{equation}\label{19/01/31/15:48}
\begin{split}
\|K_{1} S_{11}f \|_{L^{q}}  
&\lesssim
s \|  (-\Delta+s)^{-1} (-\Delta)^{-1}V f\|_{L^{q}}
\lesssim 
s^{\frac{d}{2}(\frac{d-2r_{0}}{dr_{0}}-\frac{1}{q})} 
\| (-\Delta)^{-1}V f\|_{L^{\frac{dr_{0}}{d-2r_{0}}}}
\\[6pt]
&\lesssim 
s^{\frac{d}{2}(\frac{d-2r_{0}}{dr_{0}}-\frac{1}{q})} 
\|V f\|_{L^{r_{0}}}
\lesssim 
s^{\frac{d}{2}(\frac{d-2r_{0}}{dr_{0}}-\frac{1}{q})} 
\|f\|_{L^{q}}
,
\end{split}
\end{equation} 
where the implicit constants depend only on $d$ and $p$. 
 Observe from the computations in \eqref{19/01/31/15:48} and
 Lemma 3.17 that $S_{11} \colon X_{q} \to X_{q}$. 
 Hence, the Neumann series shows that if $s$ is  sufficiently small dependently only on $d$ and $q$, then the operator $1+K_{1}S_{11} \colon X_{q}\to X_{q}$ has the inverse, say $K_{2}$, which satisfies 
\begin{equation}\label{19/01/31/16:05}
\|K_{2} \|_{X_{q}\to X_{q}}\le (1-\|K_{1}S_{11}\|_{X_{q}\to X_{q}})^{-1} \lesssim 1,
\end{equation}
where the implicit constant depends only on $d$ and $q$.
We summarize: 
\begin{equation}\label{19/01/31/16:10}
A_{11}^{-1}=\varepsilon^{-1} K_{2}K_{1},
\quad 
\|A_{11}^{-1}\|_{X_{q}\to X_{q}} \lesssim \varepsilon^{-1},
\end{equation} 
where $K_{1}$ and $K_{2}$ are the inverses of $(1+(-\Delta)^{-1}V)|_{X_{q}} \colon X_{q}\to X_{q}$, 
 and $1+K_{1}S_{11} \colon X_{q}\to X_{q}$, respectively, and the implicit constant depends only on $d$ and $q$.

%%%%%%%%%%%%%%%%%%%%

\noindent 
{\bf Step 2.}~We consider the operator $A_{12} \colon QL_{\rm rad}^{q}(\mathbb{R}^{d})\to X_{q}$ (see \eqref{A12}). Our aim is to show that  
\begin{equation}
\label{19/01/31/16:57}
\| A_{12} \|_{QL_{\rm rad}^{q}\to X_{q}} 
\lesssim 
s^{\frac{d-2}{2}-\frac{d}{2q}}
, 
\end{equation}
where the implicit constant depends only on $d$ and $q$.  

Let $f \in QL_{\rm rad}^{q}(\mathbb{R}^{d})$. Then, by \eqref{19/01/27/18:15}, \eqref{19/01/28/09:53} and \eqref{19/01/28/15:30}, we see that 
\begin{equation}\label{19/01/31/16:28}
\begin{split} 
\|A_{12}f\|_{L^{q}}
&=
\|
(1-Q)\big\{ 1+(-\Delta)^{-1}V - s(-\Delta+s)^{-1}(-\Delta)^{-1}V \big\} Qf \|_{L^{q}}
\\[6pt]
&=
\|
(1-Q)\big\{ s(-\Delta+s)^{-1}(-\Delta)^{-1}V \big\} Qf \|_{L^{q}}
\\[6pt]
&\lesssim 
s \|
(-\Delta+s)^{-1} (-\Delta)^{-1}V Qf \|_{L^{q}}
,
\end{split}
\end{equation}
where the implicit constant depends only on $d$ and $q$. 
 Applying  \eqref{19/01/31/17:41} in Lemma \ref{19/01/31/17:09} to the right-hand side of \eqref{19/01/31/16:28}, we obtain \eqref{19/01/31/16:57}.

%%%%%%%%%%%%%%%%%%%%

\noindent 
{\bf Step 3.}~We consider the operator $A_{21}\colon X_{q} \to \Pi L_{\rm rad}^{q}(\mathbb{R}^{d})$ (see \eqref{A21}).  
 By \eqref{19/01/27/18:15},  \eqref{19/01/28/16:25}, and \eqref{19/01/31/17:42} in Lemma \ref{19/01/31/17:09} with $r=q$,  
  we see that 
\begin{equation}\label{19/02/02/09:18}
\| A_{21} \|_{L_{\rm rad}^{q} \to \Pi L_{\rm rad}^{q}} 
\lesssim 
\varepsilon \delta(s)^{-1} s
,
\end{equation}
where the implicit constant depends only on $d$ and $q$.

%%%%%%%%%%%%%%%%%%%%%%%

\noindent 
{\bf Step 4.}~We consider the operator $A_{22}\colon QL_{\rm rad}^{q}(\mathbb{R}^{d}) \to \Pi L_{\rm rad}^{q}(\mathbb{R}^{d})$ (see \eqref{A22}). 

Let $f \in QL_{\rm rad}^{q}(\mathbb{R}^{d})$. 
 Observe from \eqref{19/01/27/18:15}, \eqref{19/01/28/09:53} and $(-\Delta)^{-1}V\Lambda W=-\Lambda W$ (see \eqref{20/8/15/15:47}) that 
\begin{equation}\label{21/1/12/10:10}
\begin{split}
A_{22}f
&=
-s 
\Pi (-\Delta+s)^{-1}(-\Delta)^{-1} VQf
\\[6pt]
&= 
-s
\frac{\langle f, V\Lambda W \rangle}{ \langle \Lambda W, V \Lambda W \rangle
\|V\Lambda W\|_{L^{2}}^{2}} 
\langle  (-\Delta+s)^{-1}(-\Delta)^{-1} V \Lambda W , V\Lambda W \rangle V\Lambda W
\\[6pt]
&= 
s 
\frac{\langle f, V\Lambda W \rangle}{ \langle \Lambda W, V \Lambda W \rangle
\|V\Lambda W\|_{L^{2}}^{2}} 
\langle   \Lambda W , (-\Delta+s)^{-1} V \Lambda W \rangle V\Lambda W
\\[6pt]
&=
s \frac{\langle   \Lambda W , (-\Delta+s)^{-1} V \Lambda W \rangle}{ \langle \Lambda W, V \Lambda W \rangle
}  \Pi f
.
\end{split}
\end{equation} 
Furthermore, applying Lemma \ref{18/12/02/11:25} to the right-hand side of \eqref{21/1/12/10:10} as $g=\Lambda W$, we find that $A_{22}f$ can be written as follows:  
\begin{equation}\label{19/01/27/15:37}
A_{22}f 
= 
\mathscr{C}_{0} \delta(s)^{-1} s
\Pi f
+
\mathscr{C}_{1} s \Pi f 
,
\end{equation}
where $\mathscr{C}_{0}$ and $\mathscr{C}_{1}$ are some constants depending only on $d$. 
 Define the operator $G$ with domain $QL_{\rm rad}^{q}(\mathbb{R}^{d})$ as 
\begin{equation}\label{21/1/13/10:22}
G:= \mathscr{C}_{0} \delta(s)^{-1} s
\Pi |_{QL_{\rm rad}^{q}(\mathbb{R}^{d})}
. 
\end{equation}
Observe that $G \colon QL_{\rm rad}^{q}(\mathbb{R}^{d}) \to \Pi L_{\rm rad}^{q}(\mathbb{R}^{d})$ is bijective, and the inverse $G^{-1}\colon \Pi L_{\rm rad}^{q}(\mathbb{R}^{d}) \to QL_{\rm rad}^{q}(\mathbb{R}^{d})$  is given by 
\begin{equation}\label{19/01/27/17:21}
G^{-1} g  
=
\mathscr{C}_{0}^{-1} s^{-1} 
\delta(s) \frac{ \langle g, \Lambda W \rangle}{\langle \Lambda W, V\Lambda W \rangle} \Lambda W 
. 
\end{equation}
Hence, we may write $A_{22}\colon QL_{\rm rad}^{q}(\mathbb{R}^{d}) \to \Pi L_{\rm rad}^{q}(\mathbb{R}^{d})$ as 
\begin{equation}\label{19/01/29/17:52}
 A_{22}
=
G (1+ s \mathscr{C}_{1} G^{-1} \Pi ) 
.
\end{equation}
Observe that 
\begin{equation}\label{19/01/31/18:07}
\|s \mathscr{C}_{1} G^{-1} \Pi f \|_{L^{q}}
\lesssim 
\delta(s) |\langle \Pi f, \Lambda W \rangle| \|\Lambda W \|_{L^{q}}
\lesssim 
\delta(s) |\langle f, V\Lambda W \rangle| 
\lesssim 
\delta(s)
\|f\|_{L^{q}}
,
\end{equation}
where the implicit constants depend only on $d$ and $q$.   
Hence, the Neumann series shows that if $s$ is sufficiently small dependently only on $d$ and $q$,  
 then the operator $1+s \mathscr{C}_{1} G^{-1} \Pi \colon QL_{\rm rad}^{q}(\mathbb{R}^{d})\to  QL_{\rm rad}^{q}(\mathbb{R}^{d})$ has the inverse, say $L$, such that
\begin{equation}\label{19/01/31/18:17}
\|L \|_{Q L_{\rm rad}^{q} \to QL_{\rm rad}^{q}}
\lesssim 1,
\end{equation}
where the implicit constant depends only on $d$ and $q$.  Then, we see that
\begin{equation}\label{19/02/02/09:33}
A_{22}^{-1}=LG^{-1}, 
\qquad 
\|A_{22}^{-1}\|_{\Pi L_{\rm rad}^{q} \to Q L_{\rm rad}^{q}} 
\lesssim 
\delta(s) s^{-1},
\end{equation}
where the implicit constant depends only on $d$ and $q$.

%%%%%%%%%%%%%%%%%%%%%%%%%%%%%%%%%

\noindent 
{\bf Step 5.}~We shall finish the proof. To this end, we define  
\begin{equation}\label{19/02/02/09:39}
\widetilde{\mathbf{A}}_{\varepsilon}
:=
\left( \begin{array}{cc}  
1 & A_{11}^{-1}A_{12} 
\\ 
A_{22}^{-1}A_{21} & 1 
\end{array} \right),
\qquad 
\mathbf{I}
:=
\left( \begin{array}{cc}  
1 &  0 
\\ 
0 & 1 
\end{array} \right).
\end{equation}
Note that $\widetilde{\mathbf{A}}_{\varepsilon} \colon \mathbf{X}_{q}\to \mathbf{X}_{q}$. 
 Observe from \eqref{19/01/31/13:15} that  
\begin{equation}\label{19/02/02/09:45}
\mathbf{A}_{\varepsilon}(s)(f_{1},f_{2})
=
\left( \begin{array}{cc}  
A_{11} & 0  
\\ 
0& A_{22} 
\end{array} \right)
\widetilde{\mathbf{A}}_{\varepsilon}
\left( \begin{array}{c}
f_{1} \\ f_{2}
\end{array} \right)
\quad 
\mbox{for all $(f_{1},f_{2}) \in \mathbf{X}_{q}$}
.
\end{equation}
By \eqref{19/01/31/16:10}, \eqref{19/02/02/09:33}, \eqref{19/01/31/16:57} 
 and \eqref{19/02/02/09:18}, we see that if $s$ is sufficiently small depending only on $d$ and $q$,  
\begin{equation}\label{19/02/02/09:50}
\begin{split}
\| \mathbf{I}- \widetilde{\mathbf{A}}_{\varepsilon}
\|_{\mathbf{X}_{q} \to \mathbf{X}_{q}}
&\le 
\|A_{11}^{-1} A_{12}\|_{QL_{\rm rad}^{q} \to X_{q}}
+
\|A_{22}^{-1} A_{21}\|_{X_{q} \to QL_{\rm rad}^{q}}
\\[6pt]
&\lesssim
\varepsilon^{-1} 
\|A_{12}\|_{QL_{\rm rad}^{q} \to X_{q}}
+
\delta(s)s^{-1} 
\|A_{21}\|_{X_{q} \to \Pi L_{\rm rad}^{q}}
\\[6pt]
&\lesssim
\varepsilon^{-1} s^{\frac{d-2}{2}-\frac{d}{2q}}
+
\varepsilon 
,
\end{split} 
\end{equation}
where the implicit constants depend only on $d$ and $q$.
Thus, we find that there exist $s_{0}>0$ and $\varepsilon_{0}>0$, both depending only on $d$ and $q$, such that for any $0<s < s_{0}$,    
\begin{equation}\label{19/02/02/10:08}
\| \mathbf{I}- \widetilde{\mathbf{A}}_{\varepsilon_{0}}\|_{\mathbf{X}_{q} \to \mathbf{X}_{q}} \le \frac{1}{2}
,
\end{equation}
and $\widetilde{\mathbf{A}}_{\varepsilon_{0}}$ has the inverse $\mathbf{D}:=\widetilde{\mathbf{A}}_{\varepsilon_{0}}^{-1}$ satisfying 
\begin{equation}\label{19/02/02/10:12}
\|\mathbf{D}\|_{\mathbf{X}_{q} \to \mathbf{X}_{q}} \le 2.
\end{equation}
The dependence of a constant on $\varepsilon_{0}$ can be absorbed into that on $d$ and $q$. Furthermore, we may assume that $\varepsilon_{0}\le \delta(s_{0})s_{0}^{-1}$. Then, for any $0<s < s_{0}$, the inverse of $\mathbf{A}_{\varepsilon_{0}}(s)$ exists and is given by  
\begin{equation}\label{19/02/02/10:37}
\mathbf{A}_{\varepsilon_{0}}(s)^{-1}
=
\left(\begin{array}{cc} D_{11} & D_{12} \\ D_{21} & D_{22} \end{array} \right)
\left(\begin{array}{cc} A_{11}^{-1} & 0 \\ 0 & A_{22}^{-1} \end{array} \right)
,
\end{equation}
where $D_{jk}$ is the $(j,k)$-entry of $\mathbf{D}$. Furthermore, we see from \eqref{19/01/31/16:10}, \eqref{19/02/02/09:33} and 
\eqref{19/02/02/10:12} that for any $0<s < s_{0}$, 
\begin{equation}\label{19/02/02/10:55}
\|\mathbf{A}_{\varepsilon_{0}}(s)^{-1} \|_{\mathbf{Y} \to \mathbf{X}} 
\lesssim  
\|A_{11}^{-1} \|_{X_{q} \to X_{q}}
+
\|A_{22}^{-1} \|_{\Pi L_{\rm rad}^{q} \to Q L_{\rm rad}^{q}}
\lesssim
\delta(s)s^{-1},
\end{equation}
where the implicit constants depend only on $d$ and $q$. 
 Thus, Lemma 3.12 of \cite{Jensen-Kato} (see Lemma \ref{21/1/12/9:48} in Appendix \ref{20/10/14/9:12}) shows that for any $0< s < s_{0}$,  the inverse of $1+(-\Delta+s)^{-1}V$ exists as an operator from $L_{\rm rad}^{q}(\mathbb{R}^{d})$ to itself and 
\begin{equation}\label{19/02/02/11:07}
\{ 1+(-\Delta+s)^{-1}V\}^{-1} 
=
B_{\varepsilon_{0}} \mathbf{A}_{\varepsilon_{0}}(s)^{-1}\mathbf{C}.
\end{equation}
Observe that 
\begin{equation}\label{19/02/02/11:10}
\|B_{\varepsilon_{0}}\|_{\mathbf{X}_{q}\to L_{\rm rad}^{q}}\lesssim 1, 
\qquad 
\|\mathbf{C}\|_{L_{\rm rad}^{q} \to \mathbf{Y}}\lesssim 1
.
\end{equation}
Hence, we see from \eqref{19/02/02/10:55} through \eqref{19/02/02/11:10} that 
 for any $0<s <s_{0}$, 
\begin{equation}\label{19/02/02/11:20}
\|
\{ 1+(-\Delta+s)^{-1}V \}^{-1} 
\|_{L_{\rm rad}^{q}\to  L_{\rm rad}^{q}}
\lesssim 
\delta(s)s^{-1}
,
\end{equation}
which proves \eqref{19/02/02/11:47}. 

In remains to prove \eqref{18/11/11/15:50}. Assume $f\in X_{q}$. Then, $\Pi f =0$ and therefore 
\begin{equation}\label{19/02/02/11:21}
\begin{split}
\{ 1+(-\Delta+s)^{-1}V\}^{-1} 
f 
&= 
B_{\varepsilon_{0}} \left(\begin{array}{cc} D_{11} & D_{12} \\ D_{21} & D_{22} \end{array} \right)
\left(\begin{array}{cc} A_{11}^{-1} & 0 \\ 0 & A_{22}^{-1} \end{array} \right)
\mathbf{C}f
\\[6pt]
&=
D_{11}A_{11}^{-1}(1-Q)f 
+
D_{21}A_{11}^{-1}(1-Q)f.
\end{split} 
\end{equation}
Then, \eqref{18/11/11/15:50} follows from \eqref{19/02/02/11:21}, \eqref{19/01/31/16:10} with $\varepsilon=\varepsilon_{0}$, and \eqref{19/02/02/10:12}. 
 Thus, we have completed the proof.  
\end{proof}

%%%%%%%%%%%%%%%%%%%%%%%%%%%%%%%%%%%%%%%%%%%%%%%%%%%%%%%%

\section{Uniqueness}\label{18/12/17/07:07}

In this section, we give a  proof of  Theorem \ref{18/09/09/17:09}.  
 Our proof is based on the scaling argument  and the fixed-point argument developed in \cite{CG}. 
We will see that ``the uniqueness provided by the fixed-point argument'' is helpful in proving the uniqueness of ground states to \eqref{eq:1.1}. 

First, observe that if $\Phi$ is a solution to \eqref{eq:1.1} and $\lambda>0$, 
 then $u:=T_{\lambda}[\Phi]$ satisfies 
\begin{equation}\label{20/12/6/16:37} 
 -\Delta u  
+\omega  \lambda^{-2^{*}+2} u
-\lambda^{-2^{*}+p+1}  | u|^{p-1} u
-| u |^{\frac{4}{d-2}} u   =0 
.
\end{equation}
For $s>0$ and $t>0$,  we consider a  general form of \eqref{20/12/6/16:37}:    
\begin{equation}\label{20/8/1/1:1}
-\Delta u + s u - t |u|^{p-1}u - |u|^{\frac{4}{d-2}}u = 0
.
\end{equation}
When  $s=t=0$,  \eqref{20/8/1/1:1} agrees with the equation \eqref{eq:1.15}.  
 Hence, we may expect that  \eqref{20/8/1/1:1} admits a ground state asymptotically looks like $W$ as $s,t \to 0$.

We look for  a radial solution to \eqref{20/8/1/1:1} of the form $W+\eta$, which leads us to 
\begin{equation}\label{20/12/6/14:41}
(-\Delta +V+s)  \eta
= 
-s W + t W^{p} + N(\eta; t),
\end{equation}
where $N(\eta;t)$ is the function defined by \eqref{19/01/13/15:16}. 
By the decomposition $-\Delta + V +s =(-\Delta+s) \{ 1 + (-\Delta +s)^{-1}V \}$, and the substitution $s=\alpha(\tau)$,  
 we can rewrite \eqref{20/12/6/14:41}  as  
\begin{equation}\label{20/8/19/15:59}
\eta
= 
\{ 1 + (-\Delta +\alpha(\tau))^{-1}V \}^{-1}
(-\Delta+\alpha(\tau))^{-1} F(\eta;\alpha(\tau), t), 
\end{equation}
where $F(\eta; \alpha(\tau), t):=-\alpha(\tau) W + t W^{p} + N(\eta; t)$ (see \eqref{19/01/13/15:17}). 
The significance of the substitution $s=\alpha(\tau)$ is that  $\alpha(\tau)$ is written as $\tau \delta(\alpha(\tau))$ (see \eqref{20/11/2/9:40}); 
  We will use this fact later (see \eqref{20/11/16/10:42} below). 
 Note that  a solution to \eqref{20/8/1/1:1} which asymptotically looks like $W$ corresponds to 
 a solution to  \eqref{20/8/19/15:59} which vanishes as $\tau \to 0$  (in $L^{q}(\mathbb{R}^{d})$ for all $q>\frac{d}{d-2}$).   
 Furthermore, recall from Proposition \ref{18/11/17/07:17} (see also \eqref{20/12/6/11:56} in Remark \ref{20/9/13/10:11}) that
 for any $\frac{d}{d-2}<q <\infty$,  
\begin{equation}\label{20/12/6/15:5}  
\|\{ 1 + (-\Delta +\alpha(\tau))^{-1}V \}^{-1}\|_{L_{\rm rad}^{q}\to L_{\rm rad}^{q}} \lesssim \tau^{-1},
\end{equation}
where the implicit constant depends only on $d$ and $q$. 
 From these point of view, we need to  eliminate the singular behavior of  a solution to \eqref{20/8/19/15:59} arising from \eqref{20/12/6/15:5}. 
 To this end,  we require  the following condition, as well as \cite{CG}: 
\begin{equation}\label{19/01/13/15:03}
\langle (-\Delta+ \alpha(\tau) )^{-1} F(\eta; \alpha(\tau),t), V\Lambda W \rangle 
= 0
.
\end{equation}
As a result,  we consider the system of \eqref{20/8/19/15:59} and \eqref{19/01/13/15:03}. 
 We will find a solution by using Banach fixed-point theorem, as in  \cite{CG}.  
 To this end, using \eqref{19/01/13/15:17} and $\alpha(\tau)= \tau \delta(\alpha(\tau))$ (see \eqref{20/11/2/9:40}),
 we rewrite  \eqref{19/01/13/15:03} as 
\begin{equation}\label{20/11/16/10:42}
\tau
=
\frac{t \langle (-\Delta+ \alpha(\tau) )^{-1} W^{p}, V\Lambda W \rangle  
+\langle (-\Delta+ \alpha(\tau) )^{-1} N(\eta;t), V\Lambda W \rangle }
{\delta(\alpha(\tau)) \langle (-\Delta+ \alpha(\tau) )^{-1} W, V\Lambda W \rangle}
.
\end{equation}

Here, we introduce several symbols:  
\begin{notation}\label{20/11/16/11:8}
\begin{enumerate}
\item 
For $t>0$, $0< \tau <1$,  and $\eta \in H^{1}(\mathbb{R}^{d})$, 
 we define $\mathscr{X}(\tau)$, $\mathscr{W}_{p}(\tau)$ and $\mathscr{N}(t, \tau, \eta)$ as  
\begin{align}
\label{20/10/15/8:1}
\mathscr{X}(\tau)
&:=
\delta( \alpha(\tau) ) \langle 
(-\Delta + \alpha(\tau))^{-1} W,  V \Lambda W
\rangle
,
\\[6pt]
\label{20/10/15/8:2}
\mathscr{W}_{p}(\tau)
&:=
\langle 
(-\Delta + \alpha(\tau))^{-1} W^{p},  V \Lambda W
\rangle
,
\\[6pt]
\label{20/10/15/8:3}
\mathscr{N}(t, \tau, \eta)
&:=
\langle (-\Delta + \alpha(\tau) )^{-1} N(\eta; t)
 ,  V \Lambda W
\rangle
.
\end{align}

\item
For $t>0$, $0< \tau <1$ and $\eta \in H^{1}(\mathbb{R}^{d})$, 
 we define $\mathfrak{s}(t; \tau, \eta)$ as 
\begin{equation}\label{19/01/14/13:30}
\mathfrak{s}(t; \tau, \eta) 
:=
\frac{t \mathscr{W}_{p}(\tau) + \mathscr{N}(t,\tau,\eta)}{\mathscr{X}(\tau)}
.
\end{equation}

\item 
For $t>0$, $0< \tau <1$  and $\eta \in H^{1}(\mathbb{R}^{d})$, 
 we define $\mathfrak{g}(t; \tau, \eta)$ as  
\begin{equation}\label{20/10/25/11:45}
\mathfrak{g}(t; \tau, \eta)
:=
\{ 1+ (-\Delta + \alpha(\tau) )^{-1} V \}^{-1} 
(-\Delta +\alpha(\tau) )^{-1} 
F(\eta; \alpha(\tau), t)
.
\end{equation}

\item
We use $A_{1}$ and $K_{p}$  to denote the constants given in Lemma \ref{20/12/22/13:55} and Lemma \ref{18/12/19/01:00}, respectively.  

\item 
For $t>0$, we define an interval $I(t)$ as 
\begin{equation}\label{20/10/20/10:36}
I(t):=\bigm[ \frac{K_{p}}{2A_{1}}t,~ \frac{3K_{p}}{2A_{1}}t \bigm] 
. 
\end{equation}
In stead of $I(t)$, we may use any interval  which contains $\frac{K_{p}}{A_{1}}t$ and  whose length is comparable to $t$.

\item 
For  $\frac{d}{d-2}<q < \infty$  we define a number $\Theta_{q}$ as 
\begin{equation}\label{20/11/1/11:6}
\Theta_{q}
:=
\frac{d-2}{2} - \frac{d}{2q}
.  
\end{equation}
\item
 For $\frac{d}{d-2}<q< \infty$, $R>0$ and $0< t <1$, 
 we define $Y_{q}(R,t)$ as
\begin{equation}\label{19/01/13/15:35}
Y_{q}(R,t)
:=
\bigm\{ 
\eta \in L_{\rm rad}^{q}(\mathbb{R}^{d}) 
\colon 
\|\eta \|_{L^{q}} 
\le R \alpha(t)^{\Theta_{q}}
\bigm\},
\end{equation} 
Note here that when $d=4$,   $\alpha$ is defined only on $(0,1)$ (see Lemma \ref{20/12/30/8:13}). 
\end{enumerate}
\end{notation}
\begin{remark}\label{20/11/24/11}
Observe from $s=\alpha(\beta(s))$ that   
\begin{align}
\label{20/11/24/1}
\delta(s) \langle 
(-\Delta +s )^{-1} W,  V \Lambda W
\rangle
&=
\mathscr{X}( \beta(s))
,
\\[6pt]
\label{20/11/24/2}
\langle 
(-\Delta + s)^{-1} W^{p},  V \Lambda W
\rangle
&=
\mathscr{W}_{p}(\beta(s))
,
\\[6pt]
\label{20/11/24/3}
\langle (-\Delta + s )^{-1} N(\eta; t)
 ,  V \Lambda W
\rangle
&=
\mathscr{N}(t, \beta(s), \eta)
.
\end{align}
\end{remark}

\begin{remark}\label{20/12/7/11:22}
\begin{enumerate}
\item
We can verify that for any  $\frac{d}{d-2} <q <\infty$,  $R>0$ and $0< t <1$,  
 the set $Y_{q}(R,t)$ is complete as a subspace of $L^{q}(\mathbb{R}^{d})$ with the induced metric (see Section \ref{19/03/02/16:22}). 

\item
Since $\alpha(t)$ is strictly increasing on $(0,1)$ and $\lim_{t\to 0}\alpha(t)=0$ (see Section \ref{20/8/18/11:39}), 
 we find that: for any $\theta>0$ and $R>0$, there exists $T(\theta, R)>0$ depending only on $d$, $\theta$ and $R$ such that  
\begin{equation}\label{19/02/19/12:01}
(1+R)^{\frac{d+2}{d-2}} \alpha(t)^{\theta}
\le 1
\quad 
\mbox{for all $0<t \le T(\theta,R)$}
.
\end{equation}
\end{enumerate}
\end{remark}

Now, we use the symbols \eqref{19/01/14/13:30} and \eqref{20/10/25/11:45} to rewrite  the system of 
 the equations \eqref{20/8/19/15:59}  and  \eqref{19/01/13/15:03} as follows:   
\begin{equation}\label{19/01/13/14:55}
(\tau, \eta)
=
\big(\mathfrak{s}(t; \tau, \eta),  \mathfrak{g}(t; \tau, \eta) \big)
.
\end{equation}

We can find a solution to \eqref{19/01/13/14:55}: 
\begin{proposition}\label{19/01/21/08:01}
Assume $d=3,4$ and $\frac{4}{d-2}-1 <p<\frac{d+2}{d-2}$.  
 Let $\max\{ 2^{*}, \frac{2^{*}}{p-1}\}<q <\infty$ and $R>0$.
 Then, there exists $0< T_{1}(q,R) <1$ depending only on $d$,  $p$, $q$ and $R$  with the following property:
 for any $0< t < T_{1}(q, R)$,  
 the product metric space $I(t) \times Y_{q}(R,t)$ admits one and only one solution $(\tau_{t}, \eta_{t})$ to 
  \eqref{19/01/13/14:55}.     
 Furthermore, $\tau_{t}$ is continuous and strictly increasing with respect to $t$ on $(0,T_{1}(q,R))$.   
 \end{proposition}

Observe that if $(\tau_{t}, \eta_{t})$ is a solution to \eqref{19/01/13/14:55},  then $W+\eta_{t}$ is one to the following equation:
\begin{equation}\label{19/02/25/12:07}
-\Delta u  + \alpha(\tau_{t})  u  
-
t |u|^{p-1} u
-
|u|^{\frac{4}{d-2}}u 
=0
.
\end{equation}
We denote the action and Nehari functional associated with \eqref{19/02/25/12:07} by 
\begin{align}
\label{20/11/25/7:1}
\widetilde{\mathcal{S}}_{t}(u)
&:=
\frac{ \alpha(\tau_{t})}{2}\|u\|_{L^{2}}^{2}
+
\frac{1}{2}\|\nabla u \|_{L^{2}}^{2}
-
\frac{t}{p+1}\|u\|_{L^{p+1}}^{p+1}
-
\frac{1}{2^{*}}\|u\|_{L^{2^{*}}}^{2^{*}},
\\[6pt]
\label{20/11/25/7:2}
\widetilde{\mathcal{N}}_{t}(u)
&:=
\alpha(\tau_{t})\|u\|_{L^{2}}^{2}
+
\|\nabla u \|_{L^{2}}^{2}
-
t \|u\|_{L^{p+1}}^{p+1}
-
\|u\|_{L^{2^{*}}}^{2^{*}}
.
\end{align}
Furthermore,  we define  $\widetilde{\mathcal{G}}_{t}$ as the set of all positive radial ground states to \eqref{19/02/25/12:07}, 
 namely,   $\widetilde{\mathcal{G}}_{t}$ denotes the set of positive radial minimizers for the following variational problem:
\begin{equation}\label{19/02/11/17:03}
\inf\big\{ 
\widetilde{\mathcal{S}}_{t}(u) 
 \colon u \in H^{1}(\mathbb{R}^{d})\setminus \{0\},
\  
\widetilde{\mathcal{N}}_{t}(u)=0 
\big\}
.
\end{equation}

The following proposition is a key to proving the uniqueness  of ground states to \eqref{eq:1.1}: 
\begin{proposition}[cf. Lemma 3.11 of \cite{CG}]
\label{19/02/11/16:56}
Assume $d=3,4$ and $\frac{4}{d-2}-1<p<\frac{d+2}{d-2}$.  
 Let $\max\{2^{*}, \frac{2^{*}}{p+3-2^{*}}\} < q <\infty$. 
 Furthermore, for $R>0$, let $T_{1}(q,R)$ denote the number given in Proposition \ref{19/01/21/08:01}. 
 Then,  there exist $R_{*}>0$ and $0< T_{*} <T_{1}(q,R_{*})$,  both depending only on $d$, $p$ and $q$, 
 with the following property:  
 Let $0<t < T_{*}$,  and let $(\tau_{t}, \eta_{t})$  be a unique  solution to \eqref{19/01/13/14:55} in $I(t) \times Y_{q}(R_{*},t)$ (see Proposition \ref{19/01/21/08:01}).
 Then,  $W+ \eta_{t}$ is a unique positive radial ground state to \eqref{19/02/25/12:07}; 
 in other words,  $\widetilde{\mathcal{G}}_{t}=\{ W+\eta_{t}\}$.    
\end{proposition}
We give a proof of Proposition \ref{19/02/11/16:56} in Section \ref{19/02/17/09:57}.

Now, we are in a position to prove  Theorem \ref{18/09/09/17:09}: 
\begin{proof}[Proof of Theorem \ref{18/09/09/17:09}]
Let $d=3,4$ and $\frac{4}{d-2}-1<p<\frac{d+2}{d-2}$. 
 Fix $\max\{2^{*}, \frac{2^{*}}{p+3-2^{*}}\} < q <\infty$, so that the dependence of a constant on $q$ can be absorbed into that on $d$ and $p$. 
 Then, by Proposition \ref{19/01/21/08:01} and Proposition \ref{19/02/11/16:56}, 
 we see that  there exist  $R_{*}>0$ and $0< T_{*}<1$  depending only on $d$ and $p$ such that for any $0<t \le T_{*}$,    
 $I(t) \times Y_{q}(R_{*},t)$  admits a unique  solution  $(\tau_{t}, \eta_{t})$  to \eqref{19/01/13/14:55}, and  $\widetilde{\mathcal{G}}_{t}=\{W+\eta_{t}\}$.

We shall show that there exists $\omega_{*}>0$ with the following property: for any $\omega>\omega_{*}$, there exists $0<t(\omega) <T_{*}$ such that 
\begin{equation}\label{20/9/16/12:3}
\alpha( \tau_{t(\omega)} ) = \omega t(\omega)^{\frac{2^{*}-2}{2^{*}-(p+1)}} 
.
\end{equation}
Observe from $\beta(s)=\delta(s)^{-1}s$ (see \eqref{19/01/15/08:25}) that for any $0< t<T_{*}$, 
\begin{equation}\label{20/8/20/11:24}
\beta( \omega t^{\frac{2^{*}-2}{2^{*}-(p+1)}})
=
\left\{ 
\begin{array}{ccc}
 \sqrt{\omega } t^{\frac{2}{5-p}} 
&\mbox{if}& d=3,
\\[6pt]
\log{( 1+\omega^{-1} t^{\frac{-2}{3-p}})}
 \omega t^{\frac{2}{3-p}} 
&\mbox{if}& d=4.
\end{array}
\right. 
\end{equation}
Then, by $\tau_{t} \in I(t):= \big[\frac{K_{p}}{2A_{1}}t,  \frac{3K_{p}}{2A_{1}}t \big]$ and \eqref{20/8/20/11:24}, we find (see Figure \ref{fig1}) that  there exists $\omega_{*}>0$ with the following property:
 for any $\omega >\omega_{*}$,  there exists $0<t(\omega)< T_{*}$ such that 
\begin{equation}\label{20/8/20/10:52}
\tau_{t(\omega)} = \beta( \omega t(\omega)^{\frac{2^{*}-2}{2^{*}-(p+1)}})
.
\end{equation}

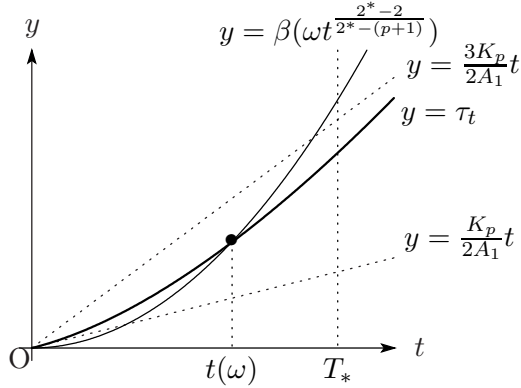
\begin{figure}[htb]\label{fig1}
\begin{center}
 \input{figure-1.tex}
\end{center}
\end{figure}
Since $\alpha$ is the inverse function of $\beta$, \eqref{20/8/20/10:52} implies  \eqref{20/9/16/12:3}.

%%%%%%%

Now, we shall finish the proof of Theorem \ref{18/09/09/17:09}.  
  Let $\omega_{*}$ be a frequency found in the above, so that for any $\omega>\omega_{*}$, there exists $0< t(\omega) < T_{*}$ such that   
 \eqref{20/9/16/12:3} holds.  Suppose for contradiction that there exists $\omega >\omega_{*}$ for which the equation 
\eqref{eq:1.1} admits two distinct positive ground states, say $\Phi_{\omega}$ and $\Psi_{\omega}$.  
Then,  we define $\lambda(\omega) >0$ by  
\begin{equation}\label{20/9/16/12:4}
\lambda(\omega)^{-\{2^{*}-(p+1)\}} =t(\omega)
.
\end{equation} 
Furthermore, we define $\widetilde{\Phi}_{\omega}$ and $\widetilde{\Psi}_{\omega}$ as  
\begin{equation}\label{20/8/19/12:8}
\widetilde{\Phi}_{\omega} :=  
T_{\lambda(\omega)} \big[ \Phi_{\omega}  \big],
\qquad 
\widetilde{\Psi}_{\omega}:= 
T_{\lambda(\omega)}\big[ \Psi_{\omega} \big] 
,
\end{equation}
where  $T_{\lambda(\omega)}$ is the scaling operator defined by \eqref{20/9/25/7:1}.
 We shall derive a contradiction by showing that  
\begin{equation}\label{20/11/25/11:50}
\widetilde{\Phi}_{\omega}, \widetilde{\Psi}_{\omega} \in 
\widetilde{\mathcal{G}}_{t(\omega)}.
\end{equation} 
Indeed,  \eqref{20/11/25/11:50} contradicts Proposition \ref{19/02/11/16:56}; Thus, the claim of Theorem \ref{18/09/09/17:09} is true.  

Let us prove \eqref{20/11/25/11:50}.  
 Observe from $\mathcal{N}_{\omega}(\Phi_{\omega})=0$,  a computation involving the scaling, and \eqref{20/9/16/12:3}  
 that 
\begin{equation}\label{20/9/16/14:50}
\widetilde{ \mathcal{N} }_{t(\omega)} 
( \widetilde{\Phi}_{\omega} )
=0
,
\end{equation}
where $\mathcal{N}_{\omega}$ and $\widetilde{ \mathcal{N} }_{t(\omega)}$ are the functionals defined by \eqref{20/8/5/10:4} and \eqref{20/11/25/7:2}, respectively. Furthermore,  \eqref{20/9/16/14:50} implies 
\begin{equation}\label{20/9/16/14:57}
\widetilde{\mathcal{S}}_{t(\omega)}
(\widetilde{\Phi}_{\omega})
\ge 
\inf\big\{ 
\widetilde{\mathcal{S}}_{t(\omega)}(u) 
 \colon u \in H^{1}(\mathbb{R}^{d})\setminus \{0\},
\  
\widetilde{\mathcal{N}}_{t(\omega)}(u)=0 
\big\}
.
\end{equation}
If $\widetilde{\Phi}_{\omega}$ were not  in $\widetilde{\mathcal{G}}_{t(\omega)}$, then  the equality failed in \eqref{20/9/16/14:57} and  there existed a function $u \in H^{1}(\mathbb{R}^{d})\setminus \{0\}$ such that 
\begin{align}
\label{20/9/16/14:59}
\widetilde{\mathcal{S}}_{t(\omega)}(u)
&<
\widetilde{\mathcal{S}}_{t(\omega)}(\widetilde{\Phi}_{\omega})
,
\\[6pt]
\label{20/9/16/15:1}
\widetilde{\mathcal{N}}_{t(\omega)}(u)
&=0 
.
\end{align}
Then, define $U_{\omega}$ by $U_{\omega}(x):= T_{\lambda(\omega)^{-1}}[u](x)=\lambda(\omega) u (\lambda(\omega)^{\frac{2}{d-2}}x)$. 
By \eqref{20/9/16/12:3}, \eqref{20/9/16/14:59}, \eqref{20/9/16/15:1} and a computation involving the scaling, we see that 
\begin{equation}\label{20/9/16/15:14}
\mathcal{S}_{\omega}(U_{\omega})
<
\mathcal{S}_{\omega}(\Phi_{\omega}), 
\qquad 
\mathcal{N}_{\omega}(U_{\omega})=0.
\end{equation} 
However, \eqref{20/9/16/15:14} contradicts that $\Phi$ is a ground state to \eqref{eq:1.1}. 
 Thus, it must hold that $\widetilde{\Phi}_{\omega} \in \widetilde{\mathcal{G}}_{t(\omega)}$. Similarly, 
   we can verify that $\widetilde{\Psi}_{\omega}\in \widetilde{\mathcal{G}}_{t(\omega)}$.  
 Thus,  we have proved \eqref{20/11/25/11:50}, which completes the proof.
\end{proof}

%%%%%%%%%%%%%%%%%%%%%%%%%%%%%%%%%%%%%%%%%%%%%%%%%%%%%%%%%%%%%%%%%%%%%%%%%%%%%%%%

\subsection{Basic estimates}\label{19/01/13/14:51}

In this section,  we prepare  several estimates  for the proof of Proposition \ref{19/01/21/08:01}.   
 In order to describe  them,  we use the following symbols: 
\begin{notation}\label{20/12/7/11:39}
\begin{enumerate}
\item
For $d=3,4$, $p>1$ and $q>1$,  define a number $\nu_{q}$ as    
\begin{equation}\label{20/10/26/8:27}
\nu_{q}
:=
\left\{ \begin{array}{cl}
\displaystyle{\frac{d}{2q}} 
& 
\mbox{if $d=3,4$ and $\frac{d}{d-2} \le p<\frac{d+2}{d-2}$}, 
\\[12pt]
\displaystyle{1 - \frac{(d-2)(p-1) }{4}}  
& 
\mbox{if $d=3,4$ and $1<p<\frac{d}{d-2}$}
. 
\end{array} \right. 
\end{equation}
Note that 
\begin{equation}\label{21/1/6/8:15}
\frac{(d-2)p-2}{2}>\frac{d-2}{2} - \nu_{q}.
\end{equation}
Furthermore,  if $p> \frac{4}{d-2}-1$ and $q>\max\{2^{*}, \frac{2^{*}}{p-1}\}$,  then 
\begin{equation}
\label{20/11/1/8:12}
\Theta_{q} -\nu_{q} >0, 
\qquad
\Theta_{q} +\frac{(d-2)(p-1)}{2}-1>0
.
\end{equation}

\item 
For $d=3,4$,  $p> \frac{4}{d-2}-1$ and $q>\max\{2^{*}, \frac{2^{*}}{p-1}\}$, define $\theta_{q}>0$ as 
\begin{equation}\label{21/1/6/10:11}
\theta_{q}:=\frac{1}{2}\min \Big\{  
\Theta_{q} -\nu_{q},~
\Theta_{q} +\frac{(d-2)(p-1)}{2}-1
\Big\} 
.
\end{equation}
\end{enumerate}
\end{notation}

\begin{remark}\label{20/12/7/11:47}
Assume $d=3,4$ and $p> \frac{4}{d-2}-1$,  and  let $q>\max\{ 2^{*}, \frac{2^{*}}{p-1}\}$.
 Then,  we see  from \eqref{20/10/15/16:46}  that  
 if $t>0$ is sufficiently small depending only on $d$, $p$ and  $q$, then 
\begin{align}
\label{21/2/6/15:27} 
t
&\le 
\alpha(t)^{\Theta_{q}}, 
\\[6pt]
\label{20/11/14/13:25}
t^{-1}  \alpha(t)^{2\Theta_{q}} 
&\le  
\alpha(t)^{\Theta_{q} -\nu_{q}}
.
\end{align}
\end{remark}

\begin{lemma}\label{19/02/24/14:00}
Assume $d=3,4$.   Let  $2^{*}< q < \infty$. 
 Then, for any $R>0$,  $s>0$,  $0< t_{1}\le t_{2}<1$, 
 $\eta_{1} \in Y_{q}(R,t_{1})$  and $\eta_{2}\in Y_{q}(R,t_{2})$,  
 the following holds: 
\begin{equation}\label{20/9/15/12:1}
\| (-\Delta+s)^{-1} D(\eta_{1},\eta_{2}) 
\|_{L^{q}}
\lesssim 
\Bigm\{
R\alpha(t_{2})^{\Theta_{q}}
+
s^{\frac{2^{*}}{q}-1} 
R^{\frac{4}{d-2}} \alpha(t_{2})^{\frac{4}{d-2}\Theta_{q}}  
\Bigm\}
\|\eta_{1}-\eta_{2}\|_{L^{q}}
,
\end{equation}
where the implicit constant depends only on $d$ and $q$.
\end{lemma}
\begin{proof}[Proof of Lemma \ref{19/02/24/14:00}]
Note that $\frac{(6-d)dq}{(d-2)(2q-d)} >\frac{d}{d-2}$. Hence,  we see from  \eqref{20/9/13/9:55}  that    
\begin{equation}\label{20/9/10/10:41}
W \in L^{\frac{(6-d) dq}{(d-2)(2q-d)}}(\mathbb{R}^{d}).
\end{equation}
Furthermore, observe from  $0<t_{1}\le t_{2}$ that 
\begin{equation}\label{20/10/13/11:31}
\max_{j=1,2}\|\eta_{j}\|_{L^{q}}\le R \alpha(t_{2})^{\Theta_{q}}
.
\end{equation}
Then, by \eqref{20/9/15/14:54}, \eqref{19/02/19/11:43}, 
 Lemma \ref{18/11/23/17:17},  Lemma \ref{18/11/05/10:29}, 
 H\"older's inequality, 
 \eqref{20/9/10/10:41} and \eqref{20/10/13/11:31}, 
 we see that 
\begin{equation}\label{19/02/24/14:12}
\begin{split}
&
\| (-\Delta +s)^{-1} D(\eta_{1},\eta_{2}) \|_{L^{q}}
\\[6pt]
&
\lesssim
\| (-\Delta +s)^{-1} 
\big\{
W^{\frac{6-d}{d-2}} (|\eta_{1}|+|\eta_{2} |)+ (|\eta_{1}|+|\eta_{2}|)^{\frac{4}{d-2}}
\big\} 
|\eta_{1}-\eta_{2} |
\|_{L^{q}}
\\[6pt]
&
\lesssim
\| W^{\frac{6-d}{d-2}}
(|\eta_{1}|+|\eta_{2} |) 
|\eta_{1}-\eta_{2} |
\|_{L^{\frac{dq}{d+2q}}}
+
s^{\frac{2^{*}}{q}-1}
\|
(|\eta_{1}|+|\eta_{2} |)^{\frac{4}{d-2}} 
|\eta_{1}-\eta_{2} |
\|_{L^{\frac{(d-2)q}{d+2}}}
\\[6pt]
&\lesssim 
\| W \|_{ L^{\frac{(6-d)dq}{(d-2)(2q-d)}}}^{\frac{6-d}{d-2}}
\| |\eta_{1}| + |\eta_{2}| \|_{L^{q}}
\|\eta_{1}-\eta_{2}\|_{L^{q}}
+
s^{\frac{2^{*}}{q}-1} 
\|  |\eta_{1}| + |\eta_{2}| \|_{L^{q}}^{\frac{4}{d-2}}
\|\eta_{1}-\eta_{2}\|_{L^{q}}
\\[6pt]
&\lesssim  
\{
R\alpha(t_{2})^{\Theta_{q}}
+
s^{\frac{2^{*}}{q}-1} 
R^{\frac{4}{d-2}} \alpha(t_{2})^{\frac{4}{d-2}\Theta_{q}}  
\}
\|\eta_{1}-\eta_{2}\|_{L^{q}}
,
\end{split}
\end{equation}
where the implicit constants depend only on $d$ and $q$. This proves the lemma. 
\end{proof}

\begin{lemma}\label{19/02/23/18:11}
Assume $d=3,4$ and $1<p <\frac{d+2}{d-2}$.  
 Let  $2^{*}\le q <\infty$. Then, for any $R>0$,  $s>0$,  $0< t_{1}\le t_{2}<1$, 
  $\eta_{1}\in Y_{q}(R,t_{1})$ and  $\eta_{2}\in Y_{q}(R,t_{2})$,  
 the following holds:   
\begin{equation}\label{19/02/23/06:02}
\| (-\Delta +s)^{-1}  E(\eta_{1},\eta_{2}) \|_{L^{q}}
\lesssim 
\Bigm\{ 
s^{-\nu_{q}}
+
s^{\frac{d(p-1)}{2q}-1} R^{p-1}
\alpha(t_{2})^{(p-1)\Theta_{q}}
\Bigm\} 
\| \eta_{1}-\eta_{2} \|_{L^{q}}
,
\end{equation}
where  the implicit constant depends only on $d$, $p$ and $q$. 
\end{lemma}
\begin{proof}[Proof of Lemma \ref{19/02/23/18:11}]
First, assume that $1<p<\frac{d}{d-2}$. 
  Let $0<\varepsilon_{0}<\frac{(d-2)(p-1)}{2}$, and define $q_{0}$ as 
\begin{equation}\label{20/10/13/10:37}
q_{0}:=\frac{dq}{d+(d-2)(p-1)q-2q\varepsilon_{0}}
.
\end{equation}
Observe that  $1< q_{0} < q$ and  $\frac{d}{d-2}< \frac{(p-1)q_{0}q}{q-q_{0}}$.
Furthermore,  by  \eqref{20/9/13/9:55}, we see that  
\begin{equation}\label{20/10/13/10:43}
W \in L^{\frac{(p-1)q_{0}q}{q-q_{0}}}(\mathbb{R}^{d}).
\end{equation}  
Then, by \eqref{20/9/15/14:54}, \eqref{19/02/19/17:02},  Lemma \ref{18/11/05/10:29}, 
H\"older's inequality, \eqref{20/10/13/10:43} and $\alpha(t_{1}) \le \alpha(t_{2})$, 
we see that 
\begin{equation}\label{20/8/9/14:26}
\begin{split}
&
\| (-\Delta +s)^{-1} E(\eta_{1},\eta_{2}) \|_{L^{q}}
\\[6pt]
&\lesssim 
s^{\frac{d}{2}(\frac{1}{q_{0}}-\frac{1}{q})-1}
\| 
W^{p-1}
|\eta_{1}-\eta_{2} |
\|_{L^{q_{0}}}
+
s^{\frac{d(p-1)}{2q}-1}
\| 
\{|\eta_{1}|+|\eta_{2}| \}^{p-1}
|\eta_{1}-\eta_{2} |
\|_{L^{\frac{q}{p}}}
\\[6pt]
&\lesssim 
s^{\frac{(d-2)(p-1)}{2}-\varepsilon_{0}-1} 
\| W \|_{L^{\frac{(p-1)q_{0}q}{q-q_{0}}}}^{p-1}
\| \eta_{1}-\eta_{2} \|_{L^{q}}
+
s^{\frac{d(p-1)}{2q}-1} 
\| |\eta_{1}| + |\eta_{2}| \|_{L^{q}}^{p-1}
\| \eta_{1}-\eta_{2} \|_{L^{q}}
\\[6pt]
&\lesssim 
s^{-1}
\Bigm\{ 
s^{\frac{(d-2)(p-1)}{2}-\varepsilon_{0}} 
+
s^{\frac{d(p-1)}{2q}}
R^{p-1}
\alpha(t_{2})^{(p-1)\Theta_{q}}
\Bigm\} 
\| \eta_{1}-\eta_{2} \|_{L^{q}}
,
\end{split}
\end{equation}
where the implicit constants depend only on $d$, $p$, $q$ and $\varepsilon_{0}$. 
Note that we may choose $\varepsilon_{0}=\frac{(d-2)(p-1)}{4}$ in \eqref{20/8/9/14:26}. 

Next,  assume that $\frac{d}{d-2}\le p < \frac{d+2}{d-2}$. 
 Observe from the assumptions about $p$ and $q$ that 
\begin{equation}\label{20/9/15/17:38}
\frac{(p-1)dq}{2q-d}>\frac{d}{d-2} \ge \frac{d}{2},
\qquad 
\frac{d(p-1)}{2q}<1,
\qquad
p< q
.
\end{equation}
Then, a computation similar to \eqref{20/8/9/14:26} shows that 
\begin{equation}\label{20/9/15/16:47}
\begin{split}
&
\| (-\Delta +s)^{-1} E(\eta_{1},\eta_{2}) \|_{L^{q}}
\\[6pt]
&\lesssim 
s^{\frac{d}{2}(\frac{2}{d}-\frac{1}{q})-1}
\| 
W^{p-1}
|\eta_{1}-\eta_{2} |
\|_{L^{\frac{d}{2}}}
+
s^{\frac{d(p-1)}{2q}-1}
\| 
\{|\eta_{1}|+|\eta_{2}| \}^{p-1}
|\eta_{1}-\eta_{2} |
\|_{L^{\frac{q}{p}}}
\\[6pt]
&\lesssim 
s^{-\frac{d}{2q}}
\| W \|_{L^{\frac{(p-1)dq}{2q-d}}}^{p-1}
\| \eta_{1}-\eta_{2} \|_{L^{q}}
+
s^{\frac{d(p-1)}{2q}-1}
\| |\eta_{1}|+ |\eta_{2}| \|_{L^{q}}^{p-1}
\| \eta_{1}-\eta_{2} \|_{L^{q}}
\\[6pt]
&\lesssim
s^{-1}
\Bigm\{ 
s^{1-\frac{d}{2q}}
+
s^{\frac{d(p-1)}{2q}} R^{p-1}
\alpha(t_{2})^{(p-1)\Theta_{q}}
\Bigm\} 
\| \eta_{1}-\eta_{2} \|_{L^{q}}
,
\end{split}
\end{equation}
where the implicit constants depend only on $d$, $p$ and $q$. 

Putting \eqref{20/8/9/14:26} with $\varepsilon_{0}=\frac{(d-2)(p-1)}{4}$, and \eqref{20/9/15/16:47},  
 we find that  \eqref{19/02/23/06:02} holds.
\end{proof}

The following lemma follows immediately from  Lemma \ref{19/02/24/14:00}, Lemma \ref{19/02/23/18:11}, 
  \eqref{20/11/8/11:39} and \eqref{20/11/8/11:35}:  
\begin{lemma}\label{19/02/10/16:33}
Assume $d=3,4$ and $\frac{4}{d-2}-1<p<\frac{d+2}{d-2}$.   
 Let $2^{*}<q<\infty$. 
\begin{enumerate}
\item  
For any $R>0$,  $s>0$, $t>0$, $0< t_{0}<1$  and $\eta \in Y_{q}(R,t_{0})$, the following holds: 
\begin{equation}\label{20/9/16/9:51}
\begin{split}
\| (-\Delta +s)^{-1}  N(\eta;t) \|_{L^{q}}
&\lesssim 
R^{2} 
\alpha(t_{0})^{2 \Theta_{q}} 
+
s^{\frac{2^{*}}{q}-1} 
R^{\frac{d+2}{d-2}}
\alpha(t_{0})^{ \frac{d+2}{d-2} \Theta_{q}} 
\\[6pt]
&\quad +
t \Bigm\{
s^{-\nu_{q}}
R \alpha(t_{0})^{\Theta_{q}}
+
s^{\frac{d(p-1)}{2q}-1} R^{p} 
\alpha(t_{0})^{ p \Theta_{q}}
\Bigm\},
\end{split}
\end{equation}
where the implicit constant depends only on $d$, $p$ and $q$.

\item
For any $R>0$,  $s>0$, $t>0$,  $0<t_{1}\le t_{2}<1$,  $\eta_{1}\in Y_{q}(R,t_{1})$ and  $\eta_{2} \in Y_{q}(R,t_{2})$, 
 the following holds:  
\begin{equation}\label{20/9/16/6:4}
\begin{split}
&\| (-\Delta +s)^{-1}  \{ N(\eta_{1};t) - N(\eta_{2};t) \} \|_{L^{q}}
\\[6pt]
&\lesssim 
\big\{ 
R \alpha(t_{2})^{\Theta_{q}} 
+
s^{\frac{2^{*}}{q}-1} 
R^{\frac{4}{d-2}}\alpha(t_{2})^{\frac{4}{d-2}\Theta_{q}}
\} 
\| \eta_{1}-\eta_{2} \|_{L^{q}}
\\[6pt]
&\quad +
t 
\big\{
s^{- \nu_{q}}
+
s^{\frac{d(p-1)}{2q} -1} 
R^{p-1}\alpha(t_{2})^{(p-1)\Theta_{q}}
\big\} 
\| \eta_{1}-\eta_{2} \|_{L^{q}}
,
\end{split}
\end{equation}
where the implicit constant depends only on $d$, $p$ and $q$.  
\end{enumerate} 
\end{lemma}

\begin{lemma}\label{19/01/03/16:41}
Assume $d=3,4$. Then, there exists $0< T <1$ depending only on $d$ such that 
 if $0< t_{1}\le t_{2}\le T$, then 
\begin{align}
\label{19/01/03/12:15}
|\alpha(t_{2})- \alpha(t_{1})|
&\le 
2 \delta(\alpha(t_{2})) 
|t_{2}
- 
t_{1}|,
\\[6pt]
\label{19/01/03/11:39}
\big| 
\delta(\alpha(t_{2}))-\delta(\alpha(t_{1}))
\big|
&\le 
(1+o_{t_{1}}(1)) t_{1}^{-1} 
\delta(\alpha(t_{1})) \delta(\alpha(t_{2}))^{d-3} 
|t_{2}- t_{1}|
. 
\end{align}
\end{lemma}
\begin{proof}[Proof of Lemma \ref{19/01/03/16:41}]
We may  assume that $d=4$, as the case $d=3$ ($\alpha(t)=t^{2}$) is almost trivial.   
 Thus, $\delta(s)=\frac{1}{\log{(1+s^{-1})}}$.  Note that $\frac{d\delta}{ds}(s)=\{ (s+s^{2})\log^{2}{(1+s^{-1})}\}^{-1}$
 and  $(s+s^{2})\log^{2}(1+s^{-1})$ is strictly increasing on  (for instance) the interval $(0, e^{-2})$.

 Let $0<t_{1}\le t_{2}<1$.  Then,  by  the fundamental theorem of calculus and  \eqref{20/11/2/9:40}, 
 we see that  if $t_{2}\le \beta(e^{-2})$ (hence $\alpha(t_{2}) \le e^{-2}$), then 
\begin{equation}\label{19/01/03/11:47}
\begin{split}
&
\delta(\alpha(t_{2}))
-
\delta(\alpha(t_{1}))
=
\int_{\alpha(t_{1})}^{\alpha(t_{2})}
\frac{1}{(s+s^{2})\log^{2}{(1+s^{-1})}}\, ds 
\\[6pt]
&\le
\frac{\alpha(t_{2})-\alpha(t_{1})}{(\alpha(t_{1})+\alpha(t_{1})^{2})\log^{2}{(1+\alpha(t_{1})^{-1})}}
\\[6pt]
&\le 
\alpha(t_{1})^{-1} 
\delta(\alpha(t_{1}))^{2}
\{\alpha(t_{2})-\alpha(t_{1})\}
= 
t_{1}^{-1}\delta(\alpha(t_{1}))
\{\alpha(t_{2})-\alpha(t_{1})\}.
\end{split}
\end{equation}
Furthermore, by \eqref{20/11/2/9:40}, we see  that 
\begin{equation}\label{19/01/16/18:50}
\alpha(t_{2}) - \alpha(t_{1}) 
=
\delta(\alpha(t_{2})) (t_{2}-t_{1})  
+
\{ \delta(\alpha(t_{2}))  - \delta(\alpha(t_{1})) \}
t_{1}
.
\end{equation}
Plugging \eqref{19/01/03/11:47} into \eqref{19/01/16/18:50}, 
 we see that 
\begin{equation}\label{20/11/17/18:4}
\alpha(t_{2}) - \alpha(t_{1}) 
\le 
\delta(\alpha(t_{2})) (t_{2}-t_{1})  
+
\delta(\alpha(t_{1})) 
\{ \alpha(t_{2}) - \alpha(t_{1}) \}
,
\end{equation}
which implies that 
\begin{equation}\label{20/11/13/9:8}
\alpha(t_{2}) - \alpha(t_{1}) 
\le 
(1+o_{t_{1}}(1))
\delta(\alpha(t_{2})) (t_{2}-t_{1})  
. 
\end{equation}
If $t_{1}$ is sufficiently small depending only on $d$, then  \eqref{20/11/13/9:8} yields  \eqref{19/01/03/12:15}. 
 Furthermore, plugging \eqref{19/01/03/12:15} into \eqref{19/01/03/11:47}, 
 we obtain \eqref{19/01/03/11:39}.
\end{proof}

%%%%%%%%%%%%%%%%%%%%%%%%%%%%%%%%%%%%%%%%%%%%%%%%%%%%%%%%%%%%%%%%%%%%%%%%%%%%%%%%

Now, observe from Lemma \ref{20/12/22/13:55},  Lemma \ref{18/12/19/01:00} and $\lim_{s\to 0}\delta(s)=0$ that  if $s>0$ is sufficiently small depending only on $d$ and $p$,
 then  
\begin{align}
\label{19/01/18/12:32}
|
\delta (s) 
\langle   
(-\Delta +s)^{-1}  W,  V \Lambda W 
\rangle 
|
&\sim A_{1} \sim 1,
\\[6pt]
\label{19/01/18/12:33}
| 
\langle 
(-\Delta +s)^{-1} W^{p}
,  V \Lambda W 
\rangle
|
&\sim K_{p} \sim 1
, 
\end{align}
where the implicit constants depend only on $d$ and $p$.

\begin{lemma}\label{19/01/03/10:53}
Assume $d=3, 4$. Then, the following holds for all $0<s_{1}\le s_{2} <1$:  
\begin{equation}\label{19/01/03/10:58}
\begin{split}
&\big|
\delta(s_{1}) 
\langle (-\Delta + s_{1})^{-1} W,  V \Lambda W \rangle
-
\delta(s_{2})
\langle  (-\Delta +s_{2})^{-1} W,  V  \Lambda W \rangle
\big|
\\[6pt]
&\lesssim 
\log{(1+s_{2}^{-1})} |s_{1}-s_{2}| 
+
|\delta(s_{1})- \delta(s_{2})|,
\end{split}
\end{equation}
where the implicit constant depends only on $d$. 
\end{lemma}
\begin{proof}[Proof of Lemma \ref{19/01/03/10:53}]
Since $-\Delta W= W^{\frac{d+2}{d-2}}$, we may write the Fourier transform of $W$ as   
\begin{equation}\label{19/02/15/12:15}
\mathcal{F}[W](\xi)
=
|\xi|^{-2} \mathcal{F}[W^{\frac{d+2}{d-2}}](\xi)
.
\end{equation}
Furthermore,  using the spherical coordinate system, we can verify that  
\begin{equation}\label{19/01/03/15:37}
\Big|
\int_{|\xi|\le 1} \!
\frac{\delta(s_{1})}{(|\xi|^{2}+s_{1})|\xi|^{2}}
\,d\xi 
-
\int_{|\xi|\le 1}
\!
\frac{\delta(s_{2})}{(|\xi|^{2}+s_{2})|\xi|^{2}}\,d\xi
\Big|~
\left\{ \begin{array}{ll}
 \lesssim |\delta(s_{1})-\delta(s_{2})| &\mbox{if $d=3$},
\\[6pt]
=0 &\mbox{if $d=4$}. 
\end{array} \right.
\end{equation}

Next, we define a function $g$ as 
\begin{equation}\label{19/02/15/12:07}
g
:=
\mathcal{F}[V \Lambda W]\mathcal{F}[W^{\frac{d+2}{d-2}}]
=
\mathcal{F}[V\Lambda W*W^{\frac{d+2}{d-2}}]
.
\end{equation}
Observe from $|V\Lambda W(x)| \sim (1+|x|)^{-(d+2)}$ and $|W^{\frac{d+2}{d-2}}(x)|\ \sim (1+|x|)^{-(d+2)}$ that 
\begin{equation}\label{19/02/16/15:55}
|V\Lambda W*W^{\frac{d+2}{d-2}}(x)|\lesssim (1+|x|)^{-(d+2)}
.
\end{equation}
Then, by  \eqref{19/02/16/15:55}  and \eqref{19/02/16/15:55},  
 we see that 
\begin{equation}\label{19/02/15/12:31}
\begin{split}
|g(\xi)-g(0)-\xi\cdot \nabla g(0)|
&=
\Big| 
\int_{\mathbb{R}^{d}}
(e^{-ix\cdot \xi}-1+i\xi \cdot x) \mathcal{F}^{-1}[g](x)\,dx 
\Big|
\\[6pt]
&\lesssim 
\int_{\mathbb{R}^{d}}
\min\{|x||\xi|, |x|^{2}|\xi|^{2}\} \big|V\Lambda W * W^{\frac{d+2}{d-2}}(x)\big|\,dx
\\[6pt]
&\lesssim 
\int_{\mathbb{R}^{d}}
\min\{|x||\xi|, |x|^{2}|\xi|^{2}\} 
(1+|x|)^{-(d+2)} \,dx
.
\end{split} 
\end{equation}

Now, we shall derive the desired estimate \eqref{19/01/03/10:58}. We see from Parseval's identity, $(-\Delta+s_{j})=\mathcal{F}[(|\xi|^{2}+s_{j})]\mathcal{F}$, \eqref{19/02/15/12:15}  and \eqref{19/01/03/15:37} that  
\begin{equation}\label{19/01/03/14:11}
\begin{split}
&
\big|
\delta(s_{1}) \langle 
R(s_{1}) W,  V \Lambda W
\rangle
-
\delta(s_{2})
\langle R(s_{2}) W
,  V  \Lambda W \rangle 
\big|
\\[6pt]
&=
\big|
\langle \big\{ \delta(s_{1}) R(s_{1}) -\delta(s_{2})R(s_{2}) \big\} 
V \Lambda W, W \rangle
\big|
\\[6pt]
&\lesssim  
\Big|
\int_{\mathbb{R}^{d}}
\Bigm\{ 
\frac{\delta(s_{1})}{(|\xi|^{2}+s_{1})|\xi|^{2}}
-
\frac{\delta(s_{2})}{(|\xi|^{2}+s_{2})|\xi|^{2}}
\Bigm\}
g(\xi)\,d\xi  
\Bigm|
\\[6pt]
&\lesssim  
\delta(s_{1}) 
\Big| \int_{|\xi|\le 1} 
\Bigm\{
\frac{
g(\xi)-g(0) 
}{(|\xi|^{2}+s_{1})|\xi|^{2}}
-
\frac{g(\xi)-g(0) 
}{(|\xi|^{2}+s_{2})|\xi|^{2}}
\Bigm\}
d\xi 
\Bigm|
+ 
|\delta(s_{1})-\delta(s_{2})| |g(0) |
\\[6pt]
&\quad +
| \delta(s_{1})-\delta(s_{2}) |
\Big|
\int_{|\xi|\le 1} 
\frac{ 
g(\xi)-g(0)
}{
(|\xi|^{2}+s_{2})|\xi|^{2}}  
\,d\xi 
\Bigm|
\\[6pt]
&\quad +
\int_{1\le|\xi|} 
\Bigm|
\frac{\delta(s_{1})}{(|\xi|^{2}+s_{1})|\xi|^{2}}
-
\frac{\delta(s_{2})}{(|\xi|^{2}+s_{2})|\xi|^{2}}
\Bigm|
| g(\xi) |
\,d\xi 
.
\end{split}
\end{equation}
We consider the first term on the right-hand side of \eqref{19/01/03/14:11}. 
 We see from \eqref{20/12/12/1}, \eqref{19/02/15/12:31}, and 
 Lemma \ref{19/02/16/10:53} in the appended section \ref{20/10/14/9:12} that 
\begin{equation}\label{19/01/03/17:11}
\begin{split}
&
\delta(s_{1}) 
\Big| \int_{|\xi|\le 1} 
\Bigm\{
\frac{
g(\xi)-g(0) 
}{(|\xi|^{2}+s_{1})|\xi|^{2}}
-
\frac{
g(\xi)-g(0) 
}{(|\xi|^{2}+s_{2})|\xi|^{2}}
\Bigm\}
d\xi 
\Bigm|
\\[6pt]
&=
\delta(s_{1}) |s_{1}-s_{2}|
\Big| \int_{|\xi|\le 1} 
\frac{ g(\xi)-g(0)
-\xi\cdot \nabla g(0)
}{(|\xi|^{2}+s_{1})(|\xi|^{2}+s_{2})|\xi|^{2}}\,d\xi 
\Bigm|
\\[6pt]
&\lesssim 
\delta(s_{1})
|s_{1}- s_{2}|
\int_{\mathbb{R}^{d}}
\int_{|\xi|\le 1} 
\frac{\min\{ |x||\xi|, \, |x|^{2}|\xi|^{2}\}}{(|\xi|^{2}+s_{1})
(|\xi|^{2}+s_{2})|\xi|^{2}}
(1+|x|)^{-(d+2)} \,d\xi dx 
\\[6pt]
&\lesssim 
|s_{1}-s_{2}|\log{(1+s_{2}^{-2})}
.
\end{split}
\end{equation}
Next, we consider the second term on the right-hand side of \eqref{19/01/03/14:11}. We see from \eqref{19/02/16/15:55} that 
\begin{equation}\label{19/02/09/14:15}
\begin{split}
| \delta(s_{1})-\delta(s_{2}) |
|g(0)|
&\le 
| \delta(s_{1})-\delta(s_{2}) |
\int_{\mathbb{R}^{d}} 
\big| V\Lambda W*W^{\frac{d+2}{d-2}}(x) \big| \,dx
\\[6pt]
&\lesssim
| \delta(s_{1})-\delta(s_{2}) |
\int_{\mathbb{R}^{d}} (1+|x|)^{-(d+2)}\,dx 
\lesssim 
| \delta(s_{1})-\delta(s_{2}) |
.
\end{split} 
\end{equation}
We move on to the third term on the right-hand side of \eqref{19/01/03/14:11}. 
 We see from \eqref{20/12/12/1} and \eqref{19/02/15/12:31} that   
\begin{equation}\label{19/01/03/14:06}
\begin{split}
&
| \delta(s_{1})-\delta(s_{2}) |
\Big|
\int_{|\xi|\le 1} 
\frac{ 
g(\xi)-g(0)
}{
(|\xi|^{2}+s_{2})|\xi|^{2}}  
\,d\xi 
\Bigm|
\\[6pt]
&=
| \delta(s_{1})-\delta(s_{2}) |
\Big|
\int_{|\xi|\le 1} 
\frac{ 
g(\xi)-g(0)-\xi\cdot (\nabla g)(0)
}{
(|\xi|^{2}+s_{2})|\xi|^{2}}  
\,d\xi 
\Bigm|
\\[6pt]
&\lesssim | \delta(s_{1})-\delta(s_{2}) |
\int_{|\xi|\le 1}  \int_{\mathbb{R}^{d}}
\frac{(|x||\xi|)^{\frac{3}{2}}}{(|\xi|^{2}+s_{2})|\xi|^{2}}(1+|x|)^{-(d+2)}
\,dxd\xi
\lesssim | \delta(s_{1})-\delta(s_{2}) |.
\end{split} 
\end{equation}
It remains to estimate the last term on the right-hand side of \eqref{19/01/03/14:11}. Note that $g$ is radial. Furthermore, it follows from Remark \ref{18/12/14/16:37} and \eqref{19/02/16/15:55} that  
\begin{equation}\label{19/02/16/15:45}
 |\xi| |g (\xi) | \lesssim 
\| g \|_{L^{2}}^{\frac{4-d}{2}} \|\nabla g\|_{L^{2}}^{\frac{d-2}{2}}
=
\|V\Lambda W * W^{\frac{d+2}{d-2}}\|_{L^{2}}^{\frac{4-d}{2}}
\||x| V\Lambda W * W^{\frac{d+2}{d-2}}\|_{L^{2}}^{\frac{d-2}{2}}
\lesssim 1
.
\end{equation}
Hence, we see from \eqref{19/02/16/15:45} that 
\begin{equation}\label{19/01/03/14:05}
\begin{split}
&
\int_{1\le|\xi|} 
\Bigm|
\frac{\delta(s_{1})}{(|\xi|^{2}+s_{1})|\xi|^{2}}
-
\frac{\delta(s_{2})}{(|\xi|^{2}+s_{2})|\xi|^{2}}
\Bigm|
|g(\xi)|
\,d\xi 
\\[6pt]
&\le
\delta(s_{1})
\int_{1\le|\xi|} 
\Bigm|
\frac{1}{(|\xi|^{2}+s_{1})|\xi|^{2}}
-
\frac{1}{(|\xi|^{2}+s_{2})|\xi|^{2}}
\Bigm|
| g(\xi) |
\,d\xi 
\\[6pt]
&\quad 
+
|\delta(s_{1})-\delta(s_{2})|
\int_{1\le|\xi|} 
\frac{ |g(\xi)|}{(|\xi|^{2}+s_{2})|\xi|^{2}}
\,d\xi 
\\[6pt]
&\le
\delta(s_{1}) | s_{1}-s_{2}|
\int_{1\le|\xi|} 
\frac{1}{|\xi|^{6}}
\,d\xi 
+
|\delta(s_{1})-\delta(s_{2})|
\int_{1\le|\xi|} 
\frac{1}{|\xi|^{5}}
\,d\xi 
\\[6pt]
&\lesssim 
\delta(s_{1}) | s_{1}-s_{2}|
+
|\delta(s_{1})-\delta(s_{2})|
.
\end{split} 
\end{equation}

Putting the estimates \eqref{19/01/03/14:11}, \eqref{19/01/03/17:11}, \eqref{19/02/09/14:15}, \eqref{19/01/03/14:06} and \eqref{19/01/03/14:05} together, we obtain the desired estimate \eqref{19/01/03/10:58}.  
\end{proof}

\begin{lemma}\label{19/02/09/10:47}
Assume $d=3,4$ and  $\frac{2}{d-2}< p < \frac{d+2}{d-2}$. Let $\max\{\frac{d}{(d-2)p}, \, \frac{d}{4} \} < r <\frac{d}{2}$. 
Then, the following holds for all $s_{1},s_{2}>0$:
\begin{equation}\label{19/01/03/20:27}
\big|
\langle (-\Delta+ s_{1})^{-1} W^{p}
,  V \Lambda W \rangle 
-
\langle (-\Delta +s_{2})^{-1} W^{p}
,  V \Lambda W \rangle 
\big| 
\lesssim 
|s_{1}- s_{2}|
s_{2}^{\frac{d}{2r}-2}
,
\end{equation}
where the implicit constant depends only on $d$, $p$ and $r$.
\end{lemma}
\begin{proof}[Proof of Lemma \ref{19/02/09/10:47}]
Observe from \eqref{20/9/13/9:52} that for any $s>0$, 
\begin{equation}\label{20/10/11/16:4}
V\Lambda W = -( -\Delta +s ) \Lambda W +s \Lambda W.   
\end{equation}
Then, by \eqref{20/10/11/16:4},  Lemma \ref{20/8/16/15:45},  H\"older's inequality and Lemma \ref{18/11/05/10:29},
 we see  that  
\begin{equation}\label{19/01/03/21:10}
\begin{split}
&  
\big|
\langle (-\Delta + s_{1})^{-1} W^{p},  V \Lambda W \rangle
-
\langle (-\Delta + s_{2})^{-1} W^{p},  V \Lambda W \rangle 
\big|
\\[6pt]
&= 
\big|
s_{1} \langle  W^{p},  (-\Delta +s_{1})^{-1} \Lambda W \rangle
-
s_{2} \langle W^{p},  (-\Delta +s_{2})^{-1} \Lambda W \rangle 
\big|
\\[6pt]
&\le 
s_{1} 
\big|
\langle  W^{p},  \big\{ (-\Delta + s_{1})^{-1} -(-\Delta +s_{2})^{-1}\big\} \Lambda W \rangle
\big|
\\[6pt]
&\quad +
|s_{1}-s_{2}|
\big| \langle W^{p},  (-\Delta +s_{2})^{-1} \Lambda W \rangle \big|
\\[6pt]
&\le 
s_{1} |s_{1}-s_{2}| 
\| (-\Delta + s_{2})^{-1}  W^{p}  \|_{L^{\frac{d}{2}}}
\| (-\Delta +s_{1})^{-1} \Lambda W  \|_{L^{\frac{d}{d-2}}}
\\[6pt]
&\quad +
|s_{1}-s_{2}| 
\|W^{p}\|_{L^{r}}\|(-\Delta + s_{2})^{-1} \Lambda W \|_{L^{\frac{r}{r-1}}}
\\[6pt]
&\lesssim
s_{1}|s_{1}- s_{2}|
s_{2}^{\frac{d}{2r}-2}\| W^{p}  \|_{L^{r}}
s_{1}^{-1} \| \Lambda W  \|_{L_{\rm weak}^{\frac{d}{d-2}}}
+
|s_{1}- s_{2}|\| W^{p}  \|_{L^{r}}
s_{2}^{\frac{d}{2r}-2}
\| \Lambda W  \|_{L_{\rm weak}^{\frac{d}{d-2}}}
\\[6pt]
&\lesssim 
|s_{1}- s_{2}|
s_{2}^{\frac{d}{2r}-2}
,
\end{split}
\end{equation}
where the implicit constant depends only on $d$, $p$ and $r$. This proves 
 the lemma.  
\end{proof} 

\begin{lemma}\label{19/02/09/16:47}
Assume $d=3,4$ and $\frac{4}{d-2}-1 <p <\frac{d+2}{d-2}$.  
 Let $2^{*}<q < \infty$. 
 Then, for any $R>0$,  $s_{1}, s_{2} >0$, $t>0$,  
 $0< t_{1} \le t_{2}<1$,  
 $\eta_{1} \in Y_{q}(R,t_{1})$ 
  and $\eta_{2}\in Y_{q}(R,t_{2})$,  
  the following holds: 
\begin{equation}\label{19/02/09/16:50}
\begin{split}
&\big|
\langle (-\Delta + s_{1})^{-1} N(\eta_{1};t)
 ,  V \Lambda W
\rangle
-
\langle (-\Delta + s_{2})^{-1} N(\eta_{2};t)
 ,  V \Lambda W
\rangle
\big|
\\[6pt] 
&\lesssim  
s_{1}^{\frac{d}{2q}-1}
\Bigm\{ 
R^{2} \alpha(t_{1})^{ 2\Theta_{q}} 
+ 
s_{2}^{\frac{2^{*}}{q}-1}  
R^{\frac{d+2}{d-2}}
\alpha(t_{1})^{\frac{d+2}{d-2}\Theta_{q}}
\Bigm\} 
|s_{1}-s_{2}| 
\\[6pt]
& \quad   +
t s_{1}^{\frac{d}{2q}-1} 
\Bigm\{ 
s_{2}^{-\nu_{q}} 
R \alpha(t_{1})^{\Theta_{q}}
+
s_{2}^{\frac{d(p-1)}{2q}-1} 
R^{p} 
\alpha(t_{1})^{p\Theta_{q}}
\Bigm\} 
|s_{1}-s_{2}| 
\\[6pt]
&\quad +
\Bigm\{
R \alpha(t_{2})^{\Theta_{q}}
+
s_{2}^{\frac{2^{*}}{q}-1}
R^{\frac{4}{d-2}}
\alpha(t_{2})^{\frac{4}{d-2}\Theta_{q}}
\Bigm\} 
\|\eta_{1} -\eta_{2} \|_{L^{q}}
\\[6pt]
&\quad +
t 
\big\{ 
s_{2}^{-\nu_{q}}  
+
s_{2}^{\frac{d(p-1)}{2q}-1} 
R^{p-1}
\alpha(t_{2})^{(p-1)\Theta_{q}}
\big\}
\|\eta_{1}-\eta_{2}\|_{L^{q}}
,
\end{split}
\end{equation}
where the implicit constant depends only on $d$, $p$ and $q$. 
\end{lemma}
\begin{proof}[Proof of Lemma \ref{19/02/09/16:47}]
Observe from an elementary computation and Lemma \ref{20/8/16/15:45}  that 
\begin{equation}\label{20/10/14/10:23}
\begin{split}
&
\langle (-\Delta + s_{1})^{-1} N(\eta_{1};t)
 ,  V \Lambda W
\rangle
-
\langle (-\Delta + s_{2})^{-1} N(\eta_{2};t)
 ,  V \Lambda W
\rangle
\\[6pt] 
&=
-(s_{1}-s_{2})
\langle (-\Delta + s_{1})^{-1}  (-\Delta + s_{2})^{-1} N(\eta_{1};t)
 ,  V \Lambda W
\rangle
\\[6pt]
&\quad +
\langle (-\Delta + s_{2})^{-1} \{ N(\eta_{1};t) - N(\eta_{2};t)\}
 ,  V \Lambda W
\rangle
.
\end{split}
\end{equation}
Then, by H\"older's inequality and  Lemma \ref{18/11/05/10:29},  
we see that  
\begin{equation}\label{19/02/19/15:21}
\begin{split}
&
\big|
\langle (-\Delta + s_{1})^{-1} N(\eta_{1};t)
 ,  V \Lambda W
\rangle
-
\langle (-\Delta + s_{2})^{-1} N(\eta_{2};t)
 ,  V \Lambda W
\rangle
\big|
\\[6pt] 
&\le 
|s_{1}-s_{2}| 
\|(-\Delta + s_{1})^{-1} V\Lambda W\|_{L^{\frac{q}{q-1}}}
\| (-\Delta + s_{2})^{-1} N(\eta_{1},t) \|_{L^{q}}
\\[6pt]
&\quad +
 \| V \Lambda W\|_{L^{\frac{q}{q-1}}}
\| (-\Delta + s_{2})^{-1} \{ N(\eta_{1};t) - N(\eta_{2};t)\|_{L^{q}}
\\[6pt]
&\lesssim
|s_{1}-s_{2}| s_{1}^{\frac{d}{2q}-1}\| V \Lambda W \|_{L^{1}}
\| (-\Delta + s_{2})^{-1} N(\eta_{1};t) \|_{L^{q}}
\\[6pt]
&\quad +
\| (-\Delta + s_{2})^{-1} \{ N(\eta_{1};t) - N(\eta_{2};t)\|_{L^{q}}
,
\end{split}
\end{equation}
where the implicit constant depends only on $d$, $p$ and $q$. 
Furthermore, applying Lemma \ref{19/02/10/16:33} to the right-hand side of \eqref{19/02/19/15:21}, 
we find that the claim of the lemma is true.  
\end{proof}

%%%%%%%%%%%%%%%%%%%%%%

\begin{lemma}\label{20/10/20/10:23}
Assume $d=3,4$ and $\frac{4}{d-2}-1<p<\frac{d+2}{d-2}$.  
 Let $\max\{2^{*}, \frac{2^{*}}{p-1}\} < q <\infty$ and $R>0$.    
   Then, there exists $0< T(q,R)<1$ depending only on $d$, $p$, $q$ and $R$ such that  
  for any $0<t <T(q,R)$,  $(\tau, \eta) \in I(t) \times Y_{q}(R,t)$  and $\max\{ 1, \frac{d}{(d-2)p}\}<r <\frac{d}{2}$,  
 the following estimate holds:    
\begin{equation}\label{20/10/15/17:12}
\Big| 
\mathfrak{s}(t;\tau, \eta) 
-\frac{K_{p}}{A_{1}}  t 
\Big|
\lesssim 
t 
\Big\{ \alpha(t)^{ \frac{d}{2r}-1 }
+
\delta(\alpha(t)) 
+
\alpha(t)^{\theta_{q}}
\Big\}
, 
\end{equation} 
where the implicit constant depends only on $d$, $p$, $q$ and $r$. 
\end{lemma}
\begin{remark}\label{20/11/15/16:23} 
Fix $\max\{1,  \frac{d}{(d-2)p} \} < r <\frac{d}{2}$, so that  the dependence of a constant on $r$ can be absorbed into that on $d$  and $p$. 
Then, Lemma \ref{20/10/20/10:23} shows that if $0< t<1$ is sufficiently small depending only on $d$, $p$, $q$ and $R$, then 
\begin{equation}\label{20/11/15/15:31}
\mathfrak{s}(t; \tau, \eta) 
\in I(t).
\end{equation} 
\end{remark}
\begin{proof}[Proof of Lemma \ref{20/10/20/10:23}] 
Let $\max\{ 1, \frac{d}{(d-2)p}\}<r <\frac{d}{2}$.  
 By \eqref{19/01/18/12:32},  H\"older's inequality,  
  Lemma \ref{18/12/19/01:00} and  Lemma \ref{20/12/22/13:55}  
 we see that 
 for any sufficiently small $t>0$ depending only on $d$ and $p$,   
 and any $(\tau, \eta) \in I(t) \times Y_{q}(R,t)$,  
\begin{equation}\label{19/01/14/13:27}
\begin{split}
&\Big| 
\mathfrak{s}(t; \tau, \eta) 
-\frac{K_{p}}{A_{1}}  t 
\Big|
=
\Big| 
\frac{t \mathscr{W}_{p}(\tau) + \mathscr{N}(t,\tau,\eta)}{\mathscr{X}(\tau)}
-\frac{K_{p}}{A_{1}}  t 
\Big|
\\[6pt]
&\le 
\Big| 
\frac{t A_{1} \mathscr{W}_{p}(\tau) - K_{p}t \mathscr{X}(\tau)
}{
A_{1} \mathscr{X}(\tau)  }
 \Bigm|
+
\Big|
\frac{\mathscr{N}(t,\tau, \eta)}{ \mathscr{X}(\tau)  }
\Big|
\\[6pt]
&\lesssim 
t  \big|  A_{1} \{ \mathscr{W}_{p}(\tau) - K_{p}\}  
-
K_{p}  \{ \mathscr{X}(\tau) -A_{1} \}  \big|
+
| \mathscr{N}(t,\tau,\eta) |
\\[6pt]
&\lesssim 
t | \mathscr{W}_{p}(\tau) - K_{p} |  
+
t  | \mathscr{X}(\tau) -A_{1} | 
+
\| (-\Delta+\alpha(\tau))^{-1} N(\eta;t) \|_{L^{q}}
\| V\Lambda W \|_{L^{\frac{q}{q-1}}}
\\[6pt]
&\lesssim
t \alpha(\tau)^{\frac{d}{2r}-1}
+
t \delta(\alpha(\tau)) 
+ 
\| (-\Delta+\alpha(\tau))^{-1} N(\eta; t) \|_{L^{q}}
,
\end{split} 
\end{equation}
where the implicit constants depend only on $d$, $p$, $q$ and $r$.  
 Note here that $\tau \in I(t)$ implies $\alpha(\tau) \sim \alpha(t)$. 
 Hence, applying Lemma \ref{19/02/10/16:33} to the right-hand side of \eqref{19/01/14/13:27} as $s=\alpha(\tau)$,   
  and  using  \eqref{20/11/14/13:25}, $\Theta_{q}<1$  (hence $\frac{2^{*}}{q}-1+\frac{d+2}{d-2}\Theta_{q} > 2 \Theta_{q}$), and \eqref{19/02/19/12:01} with $\theta=\theta_{q}$, 
 we see that if $t>0$ is sufficiently small depending only on $d$, $p$, $q$ and $R$, then 
\begin{equation}\label{20/10/15/15:42}
\begin{split}
&\Big| 
\frac{t \mathscr{W}_{p}(\tau) + \mathscr{N}(t,\tau,\eta)}{\mathscr{X}(\tau)}
-
\frac{K_{p}}{A_{1}}  t 
\Big|
\\[6pt]
&\lesssim 
t \alpha(\tau)^{ \frac{d}{2r}-1 }
+
t \delta(\alpha(\tau)) 
+ 
R^{2} 
\alpha(t)^{2\Theta_{q}} 
+
\alpha(\tau)^{\frac{2^{*}}{q}-1}
R^{\frac{d+2}{d-2}}
\alpha(t)^{\frac{d+2}{d-2}\Theta_{q}} 
\\[6pt]
&\quad +
t \Big\{
\alpha(\tau)^{-\nu_{q}} R \alpha(t)^{\Theta_{q}}
+
\alpha(\tau)^{\frac{d(p -1)}{2q}-1} R^{p} \alpha(t)^{p\Theta_{q}}
\Big\}
\\[6pt]
&\lesssim 
t \alpha(t)^{ \frac{d}{2r}-1 }
+
t \delta(\alpha(t)) 
+
t 
\alpha(t)^{\frac{1}{2}\{ \Theta_{q}-\nu_{q}\}}
+
t \alpha(t)^{ \frac{1}{2}\{ \Theta_{q}+ \frac{(d-2)(p -1)}{2}-1\}}
,
\end{split}
\end{equation}  
where the implicit constants depend only on $d$, $p$, $q$ and $r$.
Thus, we have proved the lemma.    
\end{proof}

%%%%%%%%%%%%%%%%%%%%%%%%%%%%%%%%%%%%%%%%%%

\begin{lemma}\label{20/10/27/11:21}
Assume $d=3,4$ and $\frac{4}{d-2}-1<p<\frac{d+2}{d-2}$.
 Let $\max\{2^{*}, \frac{2^{*}}{p-1} \}< q <\infty$ and $R>0$. 
 Then, there exists $0< T(q,R)<1$ depending only on $d$, $p$, $q$ and $R$ 
 with the following property:
Let $0< t_{1} \le t_{2} < T(q,R)$,  and assume that $t_{2} \le 2t_{1}$.
Furthermore,  let  $(\tau_{1}, \eta_{1}) \in I(t_{1}) \times Y_{q}(R,t_{1})$ 
 and   $(\tau_{2}, \eta_{2}) \in I(t_{2}) \times Y_{q}(R,t_{2})$.  
 Then, the following holds for all $\max\{1, \frac{d}{(d-2)p}\} < r <\frac{d}{2}$:  
\begin{equation}\label{20/10/27/11:22}
\begin{split}
&
\Big| \mathfrak{s}(t_{2}; \tau_{2}, \eta_{2})
-
\mathfrak{s}(t_{1}; \tau_{1},  \eta_{1})
-
\frac{K_{p}}{A_{1}}(t_{2}-t_{1})
\Big|
\\[6pt]
&\lesssim 
\Big\{ 
\alpha(t_{1})^{\frac{d}{2r}-1}
+
\delta(\alpha(t_{1}))
+
\alpha( t_{1})^{\theta_{q}}   
\Big\}
|t_{1}-t_{2}|
\\[6pt]
&\quad +
\Big\{  
\delta\big( \alpha( t_{1} ) \big)^{d-2}  
+
\alpha( t_{1})^{\frac{d}{2r}-1}
+
\alpha(t_{1})^{\frac{d}{2q}}
\Big\} 
| \tau_{1} - \tau_{2} |
+
t_{1} \alpha(t_{1})^{-\Theta_{q} +\theta_{q}}
\|\eta_{1} -\eta_{2}\|_{L^{q}}
, 
\end{split} 
\end{equation} 
where the implicit constant depends only on $d$, $p$, $q$ and $r$. 
\end{lemma}

\begin{remark}\label{20/11/6/11:50}
Fix $\max\{1,\frac{d}{(d-2)p}\} < r <\frac{d}{2}$, so that  the dependence of a constant on $r$ can be 
 absorbed into that on $d$ and $p$.  Then, we may write \eqref{20/10/27/11:22} as   
\begin{equation}\label{20/11/6/11:53} 
\begin{split}
&
\Big| 
 \mathfrak{s}(t_{2}; \tau_{2}, \eta_{2})
-
\mathfrak{s}(t_{1}; \tau_{1},  \eta_{1})
-
\frac{K_{p}}{A_{1}}(t_{2}-t_{1})
\Big|
\\[6pt]
&\lesssim  
o_{t_{1}}(1)
\big\{ 
|t_{1}-t_{2}| 
+
|\tau_{1}- \tau_{2} |
\big\} 
+
t_{1} \alpha(t_{1})^{-\Theta_{q}+\theta_{q}}
\|\eta_{1} -\eta_{2}\|_{L^{q}}
,
\end{split} 
\end{equation} 
where the implicit constant depends only on $d$, $p$ and $q$. 
\end{remark}

\begin{proof}[Proof of Lemma \ref{20/10/27/11:21}] 
First, observe  from the assumptions about $t_{1},t_{2}, \tau_{1},\tau_{2}$  that 
  if $t_{1}$ and $t_{2}$ are  sufficiently small depending only on $d$ and $p$, then  
\begin{equation}\label{20/10/28/11:17}
\tau_{1} \sim t_{1} \sim  \tau_{2} \sim t_{2},
\qquad 
\alpha(\tau_{1}) \sim  \alpha(t_{1}) \sim  \alpha(\tau_{2}) \sim \alpha(t_{2}), 
\end{equation} 
where the implicit constants depend only on $d$ and $p$.

Now, observe from \eqref{19/01/14/13:30} that  
\begin{equation}\label{20/10/20/11:39}
\begin{split}
&
\Big| 
 \mathfrak{s}(t_{2}; \tau_{2}, \eta_{2})
-
\mathfrak{s}(t_{1}; \tau_{1},  \eta_{1})
-
\frac{K_{p}}{A_{1}}(t_{2}-t_{1})
\Big|
\\[6pt]
&\le   
|t_{2}-t_{1}|
\Big|
\frac{
\mathscr{W}_{p}( \tau_{2} ) 
}{ 
\mathscr{X}(\tau_{2} ) }
-\frac{K_{p}}{A_{1}}
\Big|
+
\Big| 
\frac{
\mathscr{N}(t_{2},  \tau_{2},  \eta_{2})
-
\mathscr{N}(t_{1},  \tau_{2},  \eta_{2})
 }{  \mathscr{X}(\tau_{2}) }
\Big| 
\\[6pt]
&\quad 
+
t_{1} 
\Big|
\frac{
\mathscr{W}_{p}( \tau_{2} ) 
}{ 
\mathscr{X}(\tau_{2}) }
-
\frac{
\mathscr{W}_{p}( \tau_{1} ) 
}{ 
\mathscr{X}(\tau_{1}) }
\Big|
+
\Big|
 \frac{
\mathscr{N}(t_{1},  \tau_{2},  \eta_{2})
 }{  \mathscr{X}(\tau_{2}) }
 -
\frac{
\mathscr{N}(t_{1},  \tau_{1},  \eta_{1})
 }{  \mathscr{X}(\tau_{1}) }
 \Big|
 .
\end{split} 
\end{equation}
We shall estimate each term on the right-hand side of \eqref{20/10/20/11:39} separately:
\\
\noindent 
{\bf Estimate for 1st term.}~We shall derive an estimate for the first term on the right-hand side of \eqref{20/10/20/11:39}.
Let $\max\{1, \frac{d}{(d-2)p}\} < r <\frac{d}{2}$.
 Then,  by \eqref{19/01/18/12:32},   
 Lemma \ref{20/12/22/13:55}, Lemma \ref{18/12/19/01:00} and \eqref{20/10/28/11:17},  
 we see that   for any 
\begin{equation}\label{20/11/6/11:8}
\begin{split}
&
\Big|
\frac{
 \mathscr{W}_{p}( \tau_{2} ) 
}{ 
\mathscr{X}(\tau_{2}) }
-
\frac{K_{p}}{A_{1}}
\Big|
=
\Big|
\frac{
 A_{1}\mathscr{W}_{p}( \tau_{2} ) 
 -
 K_{p}\mathscr{X}(\tau_{2})
}{ 
\mathscr{X}(\tau_{2})A_{1} }
\Big|
\lesssim  
\big|
 A_{1} \mathscr{W}_{p}( \tau_{2} ) 
 -
 K_{p}\mathscr{X}(\tau_{2})
\big|
\\[6pt]
&\le  
A_{1}
\big|
\langle (-\Delta+\alpha( \tau_{2} ))^{-1}W^{p}, V\Lambda W \rangle  
 -
K_{p} 
\big|
\\[6pt]
&\quad +
K_{p} \delta(\alpha(\tau_{2}))
\big|
 A_{1} \delta(\alpha(\tau_{2}))^{-1} 
 -
\langle (-\Delta+\alpha(\tau_{2}) )^{-1}W, V\Lambda W  \rangle 
\big|
\\[6pt]
&\lesssim 
\alpha(t_{1})^{\frac{d}{2r}-1}
+
\delta(\alpha(t_{1}))
,
\end{split} 
\end{equation}
where the implicit constants depend only on $d$,  $p$ and $r$. 
 Thus, we find that  if $t_{1}$ and $t_{2}$ are  sufficiently small depending only on $d$, $p$ and $q$, then 
\begin{equation}\label{20/10/21/9:29}
|t_{2}-t_{1}|
\Big|
\frac{
 \mathscr{W}_{p}( \tau_{2} ) 
}{ 
\mathscr{X}(\tau_{2}) }
-
\frac{K_{p}}{A_{1}}
\Big|
\lesssim 
\big\{ 
\alpha(t_{1})^{\frac{d}{2r}-1}
+
\delta(\alpha(t_{1}))
\big\}
|t_{1}-t_{2}| 
,
\end{equation}
where the implicit constants depend only on $d$,  $p$ and $r$.  

\noindent 
{\bf Estimate for 2nd term.}~ Consider the  second term on the right-hand side of \eqref{20/10/20/11:39}.  
 
We see from \eqref{20/11/8/11:35}  that   
\begin{equation}\label{20/10/21/11:7}
\mathscr{N}(t_{1},  \tau_{2},  \eta_{2})
-
\mathscr{N}(t_{2},  \tau_{2},  \eta_{2})
=
(t_{1} -t_{2} ) 
\langle 
(-\Delta+ \alpha( \tau_{2}) )^{-1} 
E(\eta_{2},0),
V\Lambda W 
\rangle 
.
\end{equation}
Then,  by \eqref{19/01/18/12:32}, \eqref{20/10/21/11:7},  H\"older's inequality,  
 Lemma \ref{19/02/23/18:11} with $s=\alpha(\tau_{2})$,    \eqref{20/10/28/11:17}
 and \eqref{19/02/19/12:01} with $\theta=\theta_{q}$,   
 we see that 
   if $t_{1}$ and $t_{2}$ are sufficiently small depending only on $d$, $p$, $q$ and $R$, then  
\begin{equation}\label{20/10/30/12:7}
\begin{split} 
&\Big| 
\frac{
\mathscr{N}(t_{2},  \tau_{2},  \eta_{2})
-
\mathscr{N}(t_{1},  \tau_{2},  \eta_{2})
 }{  \mathscr{X}(\tau_{2}) }
\Big| 
\lesssim 
| \mathscr{N}(t_{1},  \tau_{2},  \eta_{2})
-
\mathscr{N}(t_{2},  \tau_{2},  \eta_{2}) |
\\[6pt]
&\lesssim  
|t_{1}-t_{2}|
\| (-\Delta+ \alpha( \tau_{2}) )^{-1} E(\eta_{2},0)   \|_{L^{q}}
\|  V\Lambda W \|_{L^{\frac{q}{q-1}}}
\\[6pt]
&\lesssim 
|t_{1}-t_{2}|
\Big\{
\alpha(\tau_{2})^{-\nu_{q}}
+
\alpha(\tau_{2})^{\frac{d(p-1)}{2q}-1} 
R^{p-1} 
\alpha(t_{2})^{(p-1)\Theta_{q}} 
\Big\}
R \alpha(t_{2})^{\Theta_{q}}
\\[6pt]
&\lesssim 
\Big\{
\alpha( t_{1})^{\frac{1}{2}\{ \Theta_{q}-\nu_{q} \}}
+
\alpha (t_{1})^{\frac{1}{2} \{ \Theta_{q}+ \frac{(d-2)(p-1)}{2}-1  \}}   
\Big\}
|t_{1}-t_{2}|
,
\end{split} 
\end{equation}
where the implicit constants depend only on $d$, $p$ and $q$.

\noindent 
{\bf Estimate for 3rd term.}~ Consider the third term on the right-hand side of \eqref{20/10/20/11:39}.  
 Let $\max\{1, \frac{d}{(d-2)p}\}< r < \frac{d}{2}$.   
 Then, by \eqref{19/01/18/12:32},  \eqref{19/01/18/12:33},  
 Lemma \ref{19/01/03/10:53},  Lemma \ref{19/02/09/10:47},  \eqref{20/10/28/11:17}, 
 Lemma \ref{19/01/03/16:41},
 we see that  
\begin{equation}\label{20/11/1/9:39}
\begin{split} 
&
t_{1} 
\Big|
\frac{
\mathscr{W}_{p}( \tau_{2} ) 
}{ 
\mathscr{X}(\tau_{2}) }
-
\frac{
\mathscr{W}_{p}( \tau_{1} ) 
}{ 
\mathscr{X}(\tau_{1}) }
\Big|
=
t_{1} 
\Big|
\frac{
\{ \mathscr{X}(\tau_{1}) - \mathscr{X}(\tau_{2}) \}
\mathscr{W}_{p}( \tau_{2} ) 
-
\mathscr{X}(\tau_{2}) 
\{ \mathscr{W}_{p}( \tau_{1} ) - \mathscr{W}_{p}( \tau_{2} )\} 
}{ 
\mathscr{X}(\tau_{1}) \mathscr{X}(\tau_{2}) }
\Big|
\\[6pt]
&\lesssim  
t_{1} 
\big|
\mathscr{X}(\tau_{1})
-
\mathscr{X}(\tau_{2})
\big| 
+
t_{1}
\big|
\mathscr{W}_{p}( \tau_{1} ) 
-
\mathscr{W}_{p}( \tau_{2} ) 
\big|
\\[6pt]
&\lesssim 
t_{1} 
\Big\{
\log{( 1+\alpha(t_{1})^{-1} )} 
+
\alpha( t_{1} )^{\frac{d}{2r}-2}
\Big\}
| \alpha(\tau_{1}) - \alpha(\tau_{2}) |
+
t_{1} 
| \delta( \alpha(\tau_{1}) )   - \delta( \alpha(\tau_{2})) |
\\[6pt]
&\lesssim 
t_{1} 
\Big\{
\log{( 1+\alpha(t_{1})^{-1} )} 
+
\alpha( t_{1} )^{\frac{d}{2r}-2}
\Big\}
\delta(\alpha(t_{1})) |\tau_{1}-\tau_{2}|
+
\delta(\alpha(t_{1}))^{d-2}
|\tau_{1} -\tau_{2}| 
\\[6pt]
&\lesssim 
\Big\{  
\delta\big( \alpha( t_{1} ) \big)^{d-2}  
+
\alpha( t_{1})^{\frac{d}{2r}-1}
\Big\} 
| \tau_{1} - \tau_{2} |
,
\end{split} 
\end{equation}
where the implicit constant depends only on $d$, $p$ and $r$.

\noindent 
{\bf Estimate for the last term.}~ We shall derive an estimate for the last term on the right-hand side of \eqref{20/10/20/11:39}, and 
 finish the proof of the lemma. 
 
 Observe from an elementary computation and \eqref{19/01/18/12:32} that 
\begin{equation}\label{20/10/21/10:45}
\begin{split} 
&
\Big|
 \frac{
\mathscr{N}(t_{1},  \tau_{2},  \eta_{2})
 }{  \mathscr{X}(\tau_{2}) }
 -
\frac{
\mathscr{N}(t_{1},  \tau_{1},  \eta_{1})
 }{  \mathscr{X}(\tau_{1}) }
 \Big|
 \\[6pt]
 &\le 
 \Big|
 \frac{
\{ 
\mathscr{X}(\tau_{1})
-
\mathscr{X}(\tau_{2})
\}
\mathscr{N}(t_{1},  \tau_{2},  \eta_{2})
 }{ \mathscr{X}(\tau_{1}) \mathscr{X}(\tau_{2}) }
\Big|
+ 
 \Big|
 \frac{
\mathscr{X}(\tau_{2})
\{ 
\mathscr{N}(t_{1},  \tau_{2},  \eta_{2})
-
\mathscr{N}(t_{1},  \tau_{1},  \eta_{1})
\}
 }{ \mathscr{X}(\tau_{1}) \mathscr{X}(\tau_{2}) }
\Big|
\\[6pt]
&\lesssim
\big|
\mathscr{X}(\tau_{2}) 
-
\mathscr{X}(\tau_{1})
\big|
\big| 
\mathscr{N}(t_{1},  \tau_{2},  \eta_{2})
\big| 
+ 
\big|  
\mathscr{N}(t_{1},  \tau_{2},  \eta_{2})
-
\mathscr{N}(t_{1},  \tau_{1},  \eta_{1})
\big| 
,
\end{split} 
\end{equation}
where the implicit constant depends only on $d$ and $p$. 

Consider the first term on the right-hand side of \eqref{20/10/21/10:45}.  
 By Lemma \ref{19/01/03/10:53},   \eqref{20/10/28/11:17},  Lemma \ref{19/01/03/16:41}, 
  we see that  
\begin{equation}\label{20/10/29/10:14}
\begin{split}
\big|
\mathscr{X}(\tau_{2}) 
-
\mathscr{X}(\tau_{1})
\big|
&\lesssim 
\log{\big( 1+ \alpha(t_{1})^{-1}  \big)}
|\alpha(\tau_{1}) - \alpha(\tau_{2})| 
+
\big|
\delta\big( \alpha(\tau_{1})\big)
- 
\delta\big( \alpha(\tau_{2}) \big)   
\big|
\\[6pt]
&\lesssim 
\log{\big( 1+ \alpha(t_{1})^{-1}  \big)} 
\delta\big( \alpha(t_{1}) \big)
| \tau_{1} -  \tau_{2} | 
+
t_{1}^{-1} \delta\big( \alpha(t_{1}) \big)^{d-2}
| \tau_{1} -  \tau_{2}| 
\\[6pt]
&\lesssim 
\Big\{
1 +
t_{1}^{-1} \delta\big( \alpha(t_{1}) \big)^{d-2}
\Big\} 
| \tau_{1} -  \tau_{2}| 
, 
\end{split}
\end{equation} 
where the implicit constants depend only on $d$ and $p$. 
 Furthermore,   by H\"older's inequality, 
 Lemma \ref{19/02/10/16:33},  \eqref{20/10/28/11:17}, $\Theta_{q} \le 1$,  
 \eqref{20/11/14/13:25},  and \eqref{19/02/19/12:01} with $\theta=\theta_{q}$,   
 we see that 
\begin{equation}\label{20/10/29/10:25}
\begin{split}
&
\big| 
\mathscr{N}(t_{1},  \tau_{2},  \eta_{2})
\big| 
\le 
\|   (-\Delta +\alpha(\tau_{2}))^{-1} N(\eta_{2}; t_{1})  \|_{L^{q}} 
\|V\Lambda W \|_{L^{\frac{q}{q-1}}}
\\[6pt]
&\lesssim 
R^{2} \alpha(t_{1})^{2 \Theta_{q}}
+
R^{\frac{d+2}{d-2}}
\alpha(t_{1})^{\Theta_{q}+1}
+
t_{1}
\Big\{
R \alpha(t_{1})^{\Theta_{q}-\nu_{q}} 
+
R^{p}\alpha(t_{1})^{\Theta_{q}+\frac{(d-2)(p-1)}{2}-1}
\Big\}
\\[6pt]
&\lesssim 
(R^{2} +R^{\frac{d+2}{d-2}} + R ) t_{1} \alpha(t_{1})^{\Theta_{q}-\nu_{q}}
+
t_{1}
R^{p}\alpha(t_{1})^{\Theta_{q}+\frac{(d-2)(p-1)}{2}-1}
\\[6pt]
&\lesssim 
t_{1} 
\Big\{ 
\alpha(t_{1})^{\frac{1}{2}\{ \Theta_{q}-\nu_{q}\}}
+
\alpha(t_{1})^{\frac{1}{2}\{ \Theta_{q}+\frac{(d-2)(p-1)}{2}-1\}}
\Big\}
,
\end{split} 
\end{equation}
where the implicit constants depend only on $d$, $p$ and $q$. 
Then, we find from  \eqref{20/10/29/10:14},  \eqref{20/10/29/10:25} and $t_{1}\le  \delta(\alpha(t_{1}))^{d-2}$   
 that   
 \begin{equation}\label{20/11/1/10:33}
\begin{split}
&
\big|
\mathscr{X}(\tau_{2}) 
-
\mathscr{X}(\tau_{1})
\big|
\big| 
\mathscr{N}(t_{1},  \tau_{2},  \eta_{2})
\big| 
\\[6pt]
&\lesssim 
\delta\big( \alpha(t_{1}) \big)^{d-2}
\Big\{
\alpha(t_{1})^{\frac{1}{2}\{ \Theta_{q}-\nu_{q}\}} 
+
\alpha(t_{1})^{\frac{1}{2}\{ \Theta_{q}+\frac{(d-2)(p-1)}{2} -1\}}
\Big\}
| \tau_{1} -  \tau_{2} |
\\[6pt]
&\lesssim 
\delta\big( \alpha(t_{1}) \big)^{d-2}
| \tau_{1} -  \tau_{2} |
 , 
\end{split}
\end{equation} 
where the implicit constant depends only on $d$, $p$ and $q$. 

%%%%%%%%%

It remains to estimate the the second term on the right-hand side of \eqref{20/10/21/10:45}. 
 By Lemma \ref{19/02/09/16:47},  \eqref{19/01/03/12:15} in Lemma \ref{19/01/03/16:41}, 
 \eqref{20/10/28/11:17}  and $t_{1}^{-1} \alpha(t_{1}) = \delta(\alpha(t_{1}))$ (see \eqref{20/11/2/9:40}),  
 \eqref{20/11/14/13:25}, $\Theta_{q}<1$ and \eqref{21/1/6/8:15},   we see that   
\begin{equation}\label{20/10/25/10:17}
\begin{split}
&
\big| 
\mathscr{N}(t_{1},  \tau_{2},  \eta_{2})
-
\mathscr{N}(t_{1},  \tau_{1},  \eta_{1})
\big| 
\\[6pt]
&\lesssim 
\Big\{
R^{2}\alpha(t_{1})^{\Theta_{q}+\frac{d-2}{2}-1}
+
R^{\frac{d+2}{d-2}}
\alpha(t_{1})^{\frac{d-2}{2}}
\Big\}
t_{1}^{-1} \alpha(t_{1})
| \tau_{1}- \tau_{2} |
\\[6pt]
&\quad + 
\Big\{
R \alpha(t_{1})^{\frac{d-2}{2}-\nu_{q}}
+
R^{p}
\alpha(t_{1})^{\frac{(d-2)p -2}{2}}
\Big\}
| \tau_{1}- \tau_{2} |
\\[6pt]
&\quad +
\Big\{ 
R \alpha (t_{1})^{\Theta_{q}}
+
R^{\frac{4}{d-2}} 
\alpha(t_{1})  
\Big\}
\|\eta_{1} -\eta_{2}\|_{L^{q}}
\\[6pt]
&\quad +
t_{1}
\Big\{
 \alpha (t_{1})^{-\nu_{q}}
+
R^{p-1} 
\alpha(t_{1})^{\frac{(d-2)(p-1)}{2}-1} 
\Big\}
\|\eta_{1} -\eta_{2}\|_{L^{q}}
\\[6pt]
&\lesssim 
(R^{2}+R +R^{\frac{d+2}{d-2}} + R^{p}) \alpha(t_{1})^{\frac{d-2}{2}-\nu_{q}}
| \tau_{1}- \tau_{2} |
\\[6pt]
&\quad +
t_{1}
\Big\{
(R+R^{\frac{4}{d-2}}+1)
\alpha (t_{1})^{-\nu_{q}}
+
R^{p-1} 
\alpha(t_{1})^{\frac{(d-2)(p-1)}{2}-1} 
\Big\}
\|\eta_{1} -\eta_{2}\|_{L^{q}}
,
\end{split}
\end{equation}
where the implicit constants depend only on $d$, $p$ and $q$.  
 Furthermore, by  \eqref{20/10/25/10:17},
 and \eqref{19/02/19/12:01} with $\theta=\theta_{q}$, 
 we see that  
\begin{equation}\label{20/10/30/11:59}
\begin{split}
&
\big| 
\mathscr{N}(t_{1},  \tau_{2},  \eta_{2})
-
\mathscr{N}(t_{1},  \tau_{1},  \eta_{1})
\big| 
\\[6pt]
&\lesssim 
\alpha(t_{1})^{\frac{d}{2q}+ \frac{1}{2}\{ \Theta_{q}-\nu_{q}\}}
| \tau_{1}- \tau_{2} |
\\[6pt]
&\quad +
t_{1}\alpha(t_{1})^{-\Theta_{q}}
\Big\{ 
 \alpha (t_{1})^{\frac{1}{2}\{\Theta_{q}-\nu_{q}\}}
+
\alpha(t_{1})^{\frac{1}{2}\{ \Theta_{q}+\frac{(d-2)(p-1)}{2}-1\}} 
\Big\}
\|\eta_{1} -\eta_{2}\|_{L^{q}}
,
\end{split}  
\end{equation} 
where the implicit constant depends only on $d$, $p$ and $q$. 
 Then,  plugging  \eqref{20/11/1/10:33} and \eqref{20/10/30/11:59}  into \eqref{20/10/21/10:45}, 
  we see that  if $t_{1}$ and $t_{2}$ are sufficiently small depending only on $d$, $p$,  $q$ and $R$,
 then    
\begin{equation}\label{20/11/14/15:3}
\begin{split} 
&
\Big|
 \frac{
\mathscr{N}(t_{1},  \tau_{2},  \eta_{2})
 }{  \mathscr{X}(\tau_{2}) }
 -
\frac{
\mathscr{N}(t_{1},  \tau_{1},  \eta_{1})
 }{  \mathscr{X}(\tau_{1}) }
 \Big|
 \\[6pt]
 &\lesssim
\Big\{
\delta\big( \alpha(t_{1}) \big)^{d-2}
+
\alpha(t_{1})^{\frac{d}{2q}}
\Big\}
| \tau_{1} -  \tau_{2} |
\\[6pt]
&\quad +
t_{1}\alpha(t_{1})^{-\Theta_{q}}
\Big\{ 
 \alpha (t_{1})^{\frac{1}{2}\{\Theta_{q}-\nu_{q}\}}
+
\alpha(t_{1})^{\frac{1}{2}\{ \Theta_{q}+\frac{(d-2)(p-1)}{2}-1\}} 
\Big\}
\|\eta_{1} -\eta_{2}\|_{L^{q}}
,
\end{split} 
\end{equation}
where the implicit constant depends only on $d$, $p$ and $q$. 

Now,  putting  \eqref{20/10/20/11:39},    \eqref{20/10/21/9:29},   \eqref{20/10/30/12:7}, 
 \eqref{20/11/1/9:39} and \eqref{20/11/14/15:3} together,   
 we  find that  \eqref{20/10/27/11:22} holds. Thus, we have completed the proof.   
\end{proof}

%%%%%%%%%%%%%%%%%%%%%%%%%%%%%%%%%%%%%%%%%%

%%%%%%%%%%%%%%%%%%%%%%%%%%%%%%%%%%

\subsection{Proof of Proposition \ref{19/01/21/08:01}} \label{19/02/17/10:06}
In this section, we give a proof of Proposition \ref{19/01/21/08:01}. 
  The following lemma  plays an essential role to 
 prove Proposition \ref{19/01/21/08:01}:  
\begin{lemma}\label{20/10/30/11:47}
Assume $d=3,4$ and $\frac{4}{d-2}-1 <p<\frac{d+2}{d-2}$.   
Let $\max\{2^{*}, \frac{2^{*}}{p-1} \}< q <\infty$ and $R>0$. 
 Then,  there exists $0< T(q,R) <1$ depending only on $d$, $p$, $q$ and $R$ 
 such that the following estimates hold:
\begin{enumerate}
\item
Let $0< t < T(q,R)$ and  $\eta \in   Y_{q}(R,t)$. 
Furthermore, let $\tau \in I(t)$ with $\tau=\mathfrak{s}(t; \tau,\eta)$. 
 Then,    
\begin{equation}\label{20/11/7/10:23}
\| \mathfrak{g}(t; \tau, \eta)\|_{L^{q}} 
\lesssim 
\alpha(t)^{\Theta_{q}} 
,
\end{equation} 
where the implicit constant depends only on $d$, $p$ and $q$.  

\item
Let $0< t_{1} \le t_{2} < T(q,R)$ with $t_{2} \le 2t_{1}$,  
 and let $\eta_{1}  \in  Y_{q}(R,t_{1})$ and $\eta_{2} \in  Y_{q}(R,t_{2})$.
 Furthermore, let $\tau_{1} \in I(t_{1})$  with   
$\tau_{1}=\mathfrak{s}(t_{1}; \tau_{1},\eta_{1})$, 
and let $\tau_{2}\in I(t_{2})$ with $\tau_{2}=\mathfrak{s}(t_{2};\tau_{2},\eta_{2})$.
 Then, 
\begin{equation}\label{20/10/30/11:48}
\begin{split}
&
\| 
\mathfrak{g}(t_{2};  \tau_{2}, \eta_{2}) 
-
\mathfrak{g}(t_{1};  \tau_{1}, \eta_{1}) \|_{L^{q}}
\\[6pt]
&\lesssim 
o_{t_{1}}(1) t_{1}^{-1} 
\alpha(t_{1})^{\Theta_{q}}
|t_{2}-t_{1}| 
+
|t_{2}-t_{1}|
+ 
\alpha(t_{1})^{\theta_{q}} 
\|\eta_{2} -\eta_{1} \|_{L^{q}}
,
\end{split}
\end{equation}
where the implicit constant depends only on $d$, $p$ and $q$. 
\end{enumerate} 
\end{lemma}
\begin{proof}[Proof of Lemma \ref{20/10/30/11:47}]
By \eqref{19/01/13/15:17}, $\alpha(\mathfrak{s}(t; \tau, \eta))= \mathfrak{s}(t; \tau, \eta) \delta( \alpha(\mathfrak{s}(t; \tau, \eta)) )$ (see \eqref{20/11/2/9:40}),   
 and \eqref{20/10/15/8:1} through  \eqref{20/10/15/8:3}  that   
\begin{equation}\label{21/2/5/10:9}
\begin{split}
&
\langle (-\Delta +\alpha( \mathfrak{s}(t; \tau, \eta) ) )^{-1}F( \eta;  \alpha( \mathfrak{s}(t; \tau, \eta) ),t), V\Lambda W \rangle 
\\[6pt]
&=
-\mathfrak{s}(t; \tau,\eta) \mathscr{X}(\mathfrak{s}(t;\tau,\eta) )
+
t \mathscr{W}_{p}(\mathfrak{s}(t; \tau,\eta) )
+
\mathscr{N}(\mathfrak{s}(t; \tau,\eta) )
,
\end{split}
\end{equation}
which together with \eqref{19/01/14/13:30}  implies that 
\begin{equation}\label{20/11/8/13:32}
\langle (-\Delta +\alpha(\tau) )^{-1}F( \eta;  \alpha(\tau), t), V\Lambda W \rangle 
=0,
\quad 
\mbox{provided $\tau = \mathfrak{s}(t; \tau,\eta)$.}
\end{equation}

Now, we shall prove \eqref{20/11/7/10:23}:
\\ 
\noindent 
{\bf Proof of  \eqref{20/11/7/10:23}}.~Let $0< t <1$ be a sufficiently small number to be specified in the middle of the proof,  dependently on $d$, $p$, $q$ and $R$.  Furthermore, let  $\eta  \in Y_{q}(R,t)$, and let $\tau \in I(t)$  with $\tau = \mathfrak{s}(t; \tau, \eta)$.  
 Then,  \eqref{18/11/11/15:50} in Proposition \ref{18/11/17/07:17} together with  \eqref{20/11/8/13:32} shows that 
\begin{equation}\label{21/2/5/10:31}
\| \mathfrak{g}(t; \tau,\eta) \|_{L^{q}}
\lesssim  
\| (-\Delta+ \alpha(\tau))^{-1} 
F(\eta;\alpha(\tau),t)  \|_{L^{q}}
, 
\end{equation}
where  the implicit constant depends only on $d$, $p$ and $q$.  
 Furthermore, by Lemma \ref{18/11/05/10:29},  
 Lemma \ref{18/11/23/17:17}, 
 \eqref{20/9/16/9:51} in Lemma \ref{19/02/10/16:33},  $\tau \in I(t)$ (hence $\tau \sim t$),  and  $\Theta_{q}<1$  (hence $\frac{2^{*}}{q}-1+\frac{d+2}{d-2}\Theta_{q} > 2 \Theta_{q}$),   
  we see that 
\begin{equation}\label{20/11/7/1:4}
\begin{split}
&
\| (-\Delta+ \alpha(\tau)  )^{-1}  F(\eta; \alpha(\tau), t) \|_{L^{q}}
\\[6pt]
&\le 
\alpha(\tau)  
\| 
(-\Delta+ \alpha(\tau)    )^{-1}  W 
\|_{L^{q}}
+
t \|  (-\Delta+ \alpha(\tau)    )^{-1}  W^{p} \|_{L^{q}}
\\[6pt]
&\quad + 
\|  (-\Delta+ \alpha(\tau)    )^{-1}
N(\eta;t)\|_{L^{q}}
\\[6pt]
&\lesssim  
\alpha(\tau)  ^{\Theta_{q}}
\| W  \|_{L_{\rm weak}^{\frac{d}{d-2}}}
+
t \| W^{p} \|_{L^{\frac{dq}{d+2q}}}
\\[6pt]
&\quad +  
R^{2} \alpha(t)^{2\Theta_{q}} 
+ 
\alpha(\tau)^{\frac{2^{*}}{q}-1} 
R^{\frac{d+2}{d-2}} 
\alpha(t)^{\frac{d+2}{d-2}\Theta_{q}}
\\[6pt]
&\quad + 
t \Big\{ 
\alpha(\tau)^{-\nu_{q}} 
R\alpha(t)^{\Theta_{q}} 
+ 
\alpha(\tau)^{\frac{d(p-1)}{2q}-1} 
R^{p}
\alpha(t)^{p \Theta_{q}}
\Big\} 
\\[6pt]
&\lesssim 
\alpha(t)^{\Theta_{q}}
+
(R^{2} + R^{\frac{d+2}{d-2}} ) \alpha(t)^{2\Theta_{q}} 
+
t R \alpha(t)^{\Theta_{q}-\nu_{q}} 
+ 
t R^{p}   \alpha(t)^{\Theta_{q} + \frac{(d-2)(p-1)}{2}-1}
,
\end{split} 
\end{equation}
where the implicit constants depend only on $d$, $p$ and $q$. 
 Plugging  \eqref{20/11/7/1:4} into \eqref{21/2/5/10:31},  and then using \eqref{19/02/19/12:01} with $\theta=\theta_{q}$, and \eqref{21/2/6/15:27}, 
  we see that  if $t$ is sufficiently small depending only on $d$, $p$, $q$ and $R$, then 
\begin{equation}\label{20/11/7/13:2}
\begin{split}
\| \mathfrak{g}(t; \tau,\eta) \|_{L^{q}}
&\lesssim
\alpha(t)^{\Theta_{q}} 
+ 
\alpha(t)^{2 \Theta_{q}-\theta_{q}} 
+
t \alpha(t)^{ \theta_{q}} 
\lesssim 
\alpha(t)^{\Theta_{q}} 
,
\end{split} 
\end{equation}
where the implicit constants depend only on $d$, $p$ and $q$.  
 Thus, we find that  \eqref{20/11/7/10:23} holds.

%%%%%%%%%%%%%

Next, we shall prove \eqref{20/10/30/11:48}:
\\
\noindent 
{\bf Proof of   \eqref{20/10/30/11:48}.}~Let $0< t_{1} \le  t_{2}<1$ with $t_{2}\le 2t_{1}$, and let $\eta_{1}  \in  Y_{q}(R,t_{1})$ and $\eta_{2} \in Y_{q}(R,t_{2})$. 
 We will specify $t_{1}$ and $t_{2}$ in the middle of the proof, dependently on $d$, $p$, $q$ and $R$. 
 Furthermore, let $\tau_{1} \in I(t_{1})$ with $\tau_{1}=\mathfrak{s}(t_{1}; \tau_{1},\eta_{1})$,  and let $\tau_{2}\in I(t_{2})$ with  $\tau_{2}=\mathfrak{s}(t_{2}; \tau_{2},\eta_{2})$.  
 Observe from $t_{1}\le t_{2}\le 2t_{1}$, $\tau_{1}\in I(t_{1})$ and  $\tau_{2}\in I(t_{2})$ that     
\begin{equation}\label{20/11/2/9:39}
t_{1} \sim \tau_{1} \sim t_{2} \sim \tau_{2},
\qquad 
\alpha(t_{1}) 
\sim 
\alpha(\tau_{1}) 
\sim  
\alpha(t_{2})
\sim 
\alpha(\tau_{2}) 
,
\end{equation} 
where the implicit constants depend only on $d$ and $p$.  
By \eqref{20/11/8/13:32}, we see that   
\begin{equation}\label{20/10/25/14:15}
\langle 
(-\Delta+ \alpha( \tau_{1})  )^{-1} F(\eta_{1};\alpha(\tau_{1}),t_{1})
-
(-\Delta+\alpha( \tau_{2}))^{-1} 
F(\eta_{2};\alpha(\tau_{2}),t_{2}),~
V\Lambda W \rangle =0 
.
\end{equation} 
Furthermore, Lemma \ref{20/10/27/11:21} (see also Remark \ref{20/11/6/11:50}) together with  $\tau_{j}=\mathfrak{s}(t_{j}; \tau_{j},\eta_{j})$ shows that     if $t_{1}$ and $t_{2}$ are sufficiently small depending only on  $d$, $p$, $q$ and $R$, then  
\begin{equation}\label{20/11/3/9:26}
|\tau_{2} - \tau_{1}|
\lesssim 
o_{t_{1}}(1) |t_{2}-t_{1}|
+
t_{1}
\alpha (t_{1})^{-\Theta_{q}+\theta_{q}}
\|\eta_{2} -\eta_{1}\|_{L^{q}}
,
\end{equation} 
where the implicit constant depends only on $d$, $p$ and $q$.

For notational convenience, we put  
\begin{equation}\label{20/11/5/9:15}
\mathscr{A}_{j}:= 1+(-\Delta+\alpha( \tau_{j}))^{-1} V 
\quad 
\mbox{for $j=1,2$.}
\end{equation}
Observe from  
 $\mathscr{A}_{2}-\mathscr{A}_{1}= \{ (-\Delta+\alpha( \tau_{2}))^{-1}- (-\Delta+\alpha( \tau_{1}))^{-1} \} V$ that 
\begin{equation}\label{20/10/25/11:55}
\begin{split}
&
\| \mathfrak{g}(t_{2}; \tau_{2}, \eta_{2})  -\mathfrak{g}(t_{1}; \tau_{1}, \eta_{1}) \|_{L^{q}}
\\[6pt]
&\le 
\| \mathscr{A}_{1}^{-1}
\big\{
 (-\Delta+ \alpha(\tau_{1}))^{-1}F(\eta_{1};\alpha(\tau_{1}), t_{1})
-
(-\Delta+ \alpha(\tau_{2}))^{-1}F(\eta_{2};\alpha(\tau_{2}), t_{2}) 
\big\}
\|_{L^{q}}
\\[6pt]
&\quad +
\| 
\mathscr{A}_{1}^{-1} 
\big\{
(-\Delta+ \alpha(\tau_{2}))^{-1} 
-
(-\Delta+ \alpha(\tau_{1}))^{-1} 
\big\} 
V \mathfrak{g}(t_{2};\tau_{2},\eta_{2})
 \|_{L^{q}}
 .
\end{split}
\end{equation}
Then,  the desired estimate \eqref{20/10/30/11:48} follows from  the following estimates: 
\begin{equation}\label{20/11/5/9:5}
\begin{split}
&
\|\mathscr{A}_{1}^{-1}
\big\{
 (-\Delta+ \alpha(\tau_{1}))^{-1}F(\eta_{1};\alpha(\tau_{1}), t_{1})
-
(-\Delta+ \alpha(\tau_{2}))^{-1}F(\eta_{2};\alpha(\tau_{2}), t_{2}) 
\big\}
\|_{L^{q}}
\\[6pt]
&\lesssim
o_{t_{1}}(1)  t_{1}^{-1} 
\alpha(t_{1})^{\Theta_{q}}
|t_{1}-t_{2}| 
+
|t_{1}-t_{2}| 
+ 
\alpha(t_{1})^{\theta_{q}} 
\|\eta_{1} -\eta_{2} \|_{L^{q}}
,
\end{split} 
\end{equation}
where the implicit constant depends only on $d$, $p$ and $q$; and 
\begin{equation}\label{20/11/11/11:25}
\begin{split} 
&
\| 
\mathscr{A}_{1}^{-1} 
\big\{
(-\Delta+ \alpha(\tau_{2}))^{-1} 
-
(-\Delta+ \alpha(\tau_{1}))^{-1} 
\big\} 
V \mathfrak{g}(t_{2};\tau_{2},\eta_{2})
 \|_{L^{q}}
\\[6pt]
&\lesssim 
o_{t_{1}}(1)t_{1}^{-1}  \alpha(t_{1})^{\Theta_{q}} |t_{2}-t_{1}|
+
\alpha (t_{1})^{\theta_{q}}
\|\eta_{2} -\eta_{1}\|_{L^{q}} 
, 
\end{split}
\end{equation}
where the implicit constant depends only on $d$, $p$ and $q$. 

It remains to prove \eqref{20/11/5/9:5} and \eqref{20/11/11/11:25}: 

%%%%%%%%%%%%%%%%

\noindent 
{\bf  Proof of \eqref{20/11/5/9:5}.}~By \eqref{18/11/11/15:50} in Proposition \ref{18/11/17/07:17}, and \eqref{20/10/25/14:15}, we see that 
\begin{equation}\label{21/2/5/16:21}
\begin{split}
&
\|\mathscr{A}_{1}^{-1}
\big\{
 (-\Delta+ \alpha(\tau_{1}))^{-1}F(\eta_{1};\alpha(\tau_{1}), t_{1})
-
(-\Delta+ \alpha(\tau_{2}))^{-1}F(\eta_{2};\alpha(\tau_{2}), t_{2}) 
\big\}
\|_{L^{q}}
\\[6pt]
&\lesssim
\| 
(-\Delta+ \alpha(\tau_{1}))^{-1}F(\eta_{1};\alpha(\tau_{1}), t_{1})
-
(-\Delta+ \alpha(\tau_{2}))^{-1}F(\eta_{2};\alpha(\tau_{2}), t_{2})
\|_{L^{q}}
,
\end{split} 
\end{equation}
where the implicit constant depends only on $d$, $p$ and $q$. Furthermore, we 
 estimate the right-hand side of \eqref{21/2/5/16:21} as follows:   
\begin{equation}\label{20/11/2/9:6}
\begin{split}
&
\| 
(-\Delta+ \alpha(\tau_{1}))^{-1}F(\eta_{1};\alpha(\tau_{1}), t_{1})
-
(-\Delta+ \alpha(\tau_{2}))^{-1}F(\eta_{2};\alpha(\tau_{2}), t_{2})
\|_{L^{q}}
\\[6pt]
&\le 
\| 
(-\Delta+ \alpha( \tau_{1})  )^{-1} 
\big\{ 
F(\eta_{1};\alpha(\tau_{1}), t_{1})
-
F(\eta_{2};\alpha(\tau_{2}), t_{2})
\big\} \|_{L^{q}}
\\[6pt]
&\quad +
\| 
\big\{ (-\Delta+ \alpha(\tau_{1})  )^{-1} 
-
(-\Delta+\alpha(\tau_{2}))^{-1} 
\big\} 
F(\eta_{2};\alpha(\tau_{2}), t_{2})
\|_{L^{q}}
.
\end{split} 
\end{equation}

 Consider the first term on the right-hand side of \eqref{20/11/2/9:6}.  
 Observe that  $\frac{p dq}{d+2q}>\frac{d}{d-2}$ for $p>\frac{4}{d-2}-1$ and $q>\frac{2^{*}}{p-1}$. 
 Hence, we see that 
\begin{equation}\label{21/2/6/12}
\| W^{p} \|_{L^{\frac{dq}{d+2q}}} \lesssim 1, 
\end{equation} 
where the implicit constants depend only on $d$, $p$ and $q$.
 Furthermore, observe from \eqref{20/11/8/11:39} that 
\begin{equation}\label{21/2/6/11:48}
\begin{split}
N(\eta_{1}; t_{1}) - N(\eta_{2}; t_{2})
&=
 N(\eta_{1}; t_{1}) - N(\eta_{2}; t_{1})
 +
 N(\eta_{2}; t_{1}) - N(\eta_{2}; t_{2})
\\[6pt]
&=
N(\eta_{1}; t_{1}) - N(\eta_{2}; t_{1})
+
(t_{1}-t_{2}) E(\eta_{2}, 0). 
\end{split}
\end{equation}
Then,  by \eqref{19/01/13/15:17}, \eqref{19/01/03/12:15} in Lemma \ref{19/01/03/16:41},   Lemma \ref{18/11/05/10:29}, $\tau_{1} \sim \tau_{2} \sim t_{1}$ (see \eqref{20/11/2/9:39}),  Lemma \ref{18/11/23/17:17},  \eqref{21/2/6/12} 
 and \eqref{21/2/6/11:48},  
 we see that 
\begin{equation}\label{20/10/25/13:51}
\begin{split}
&
\| 
(-\Delta+ \alpha( \tau_{1})  )^{-1} 
\big\{ 
F(\eta_{1};\alpha(\tau_{1}), t_{1})
-
F(\eta_{2};\alpha(\tau_{2}), t_{2})
\big\} \|_{L^{q}}
\\[6pt]
&\le 
| \alpha(\tau_{1} ) 
-
\alpha(\tau_{2} )  |
\| 
(-\Delta+ \alpha(\tau_{1})  )^{-1}  W 
\|_{L^{q}}
+
|t_{1} - t_{2}| 
\| 
(-\Delta+ \alpha(\tau_{1}))^{-1}  W^{p} \|_{L^{q}}
\\[6pt]
&\quad + 
\|  (-\Delta+ \alpha(\tau_{1})  )^{-1}
\big\{ 
N(\eta_{1};t_{1})
-
N(\eta_{2};t_{2}) 
\} \|_{L^{q}}
\\[6pt]
&\lesssim
\delta(\alpha(t_{1}))
|\tau_{1}-\tau_{2}|
\alpha(t_{1})^{\Theta_{q}-1}
\| W  \|_{L_{\rm weak}^{\frac{d}{d-2}}}  
+
|t_{1}-t_{2}| 
\\[6pt]
&\quad +  
\| 
(-\Delta+ \alpha(\tau_{1})  )^{-1}
\{ N(\eta_{1}; t_{1}) - N(\eta_{2}; t_{1}) \} \|_{L^{q}}
\\[6pt]
&\quad +  
|t_{1}-t_{2}|\| 
(-\Delta+ \alpha(\tau_{1})  )^{-1} E(\eta_{2},0) \|_{L^{q}}
, 
\end{split} 
\end{equation}
where the implicit constants depend only on $d$, $p$ and $q$. 
Recall from \eqref{20/11/2/9:40} that
\begin{equation}\label{20/11/4/9:39}
\delta(\alpha(t_{1}))
=
t_{1}^{-1} \alpha(t_{1})
.
\end{equation} 
Then, by \eqref{20/11/3/9:26} and \eqref{20/11/4/9:39},  
 we see that  the first term on the right-hand side of \eqref{20/10/25/13:51} is estimated as follows: 
\begin{equation}\label{21/2/5/17:19}
\begin{split}
&
\delta(\alpha(t_{1}))
|\tau_{1}-\tau_{2}|
\alpha(t_{1})^{\Theta_{q}-1}
\| W  \|_{L_{\rm weak}^{\frac{d}{d-2}}}
\\[6pt]
&\lesssim 
o_{t_{1}}(1) \delta(\alpha(t_{1}))|t_{1}-t_{2}|\alpha(t_{1})^{\Theta_{q}-1}
+
\delta(\alpha(t_{1}))t_{1}
\alpha (t_{1})^{\theta_{q}-1}
\|\eta_{2} -\eta_{1}\|_{L^{q}}
\\[6pt]
&\lesssim 
o_{t_{1}}(1)t_{1}^{-1} \alpha(t_{1})^{\Theta_{q}}  |t_{1}-t_{2}|
+
\alpha (t_{1})^{\theta_{q}}
\|\eta_{2} -\eta_{1}\|_{L^{q}}
,
\end{split} 
\end{equation}
where the implicit constants depend only on $d$, $p$ and $q$.   
 Applying Lemma \ref{19/02/10/16:33} to the third term on the right-hand side of \eqref{20/10/25/13:51},
  and using \eqref{20/11/2/9:39},  \eqref{21/2/6/15:27},  \eqref{19/02/19/12:01} with $\theta=\theta_{q}$,  
 and  and the definition of $\theta_{q}$ (see \eqref{21/1/6/10:11}),  
 we see that  
\begin{equation}\label{21/2/5/17:10}
\begin{split}
&
\|(-\Delta+ \alpha(\tau_{1})  )^{-1}
\big\{ 
N(\eta_{1};t_{1})
-
N(\eta_{2};t_{1}) 
\} \|_{L^{q}}
\\[6pt]
&\lesssim 
\big\{  R\alpha(t_{1})^{\Theta_{q}} 
+ 
\alpha(t_{1})^{\frac{2^{*}}{q}-1}
R^{\frac{4}{d-2}} \alpha(t_{1})^{\frac{4}{d-2}\Theta_{q}}
\big\}
\|\eta_{1} -\eta_{2} \|_{L^{q}}
\\[6pt]
&\quad +
t_{1} \big\{
\alpha(t_{1})^{-\nu_{q}} 
+ 
\alpha(t_{1})^{\frac{d(p-1)}{2q}-1}
R^{p-1} \alpha(t_{1})^{(p-1)\Theta_{q} }
\big\}
\|\eta_{1} -\eta_{2} \|_{L^{q}}
\\[6pt] 
&\lesssim 
\big\{  \alpha(t_{1})^{\Theta_{q}-\theta_{q}} 
+ 
\alpha(t_{1})^{1-\theta_{q}}
\big\}
\|\eta_{1} -\eta_{2} \|_{L^{q}}
\\[6pt]
&\quad + 
\big\{
\alpha(t_{1})^{\Theta_{q}-\nu_{q}} 
+ 
\alpha(t_{1})^{\Theta_{q}+\frac{(d-2)(p-1)}{2}-1 -\theta_{q}}
\big\}
\|\eta_{1} -\eta_{2} \|_{L^{q}}
\\[6pt]
&\lesssim 
\alpha(t_{1})^{\theta_{q}}
\|\eta_{1} -\eta_{2} \|_{L^{q}}
, 
\end{split}
\end{equation}
where the implicit constants depend only on $d$, $p$ and $q$. 
 Furthermore,  using Lemma \ref{19/02/23/18:11}, \eqref{20/11/2/9:39},  \eqref{21/2/6/15:27},  \eqref{19/02/19/12:01} with $\theta=\theta_{q}$,  
  and the definition of $\theta_{q}$ (see \eqref{21/1/6/10:11}), 
 we see that the last term on the right-hand side of \eqref{20/10/25/13:51} 
 is estimated as follows:    
\begin{equation}\label{20/11/2/10:15}
\begin{split}
&
|t_{1}-t_{2}|\| 
(-\Delta+ \alpha(\tau_{1})  )^{-1} E(\eta_{2},0) \|_{L^{q}}
\\[6pt]
&\lesssim  
|t_{1}-t_{2}| 
\big\{
\alpha(t_{1})^{-\nu_{q}} 
+ 
\alpha(t_{1})^{\frac{d(p-1)}{2q}-1} 
R^{p-1} 
\alpha(t_{1})^{(p-1)\Theta_{q}}
\big\}
\|\eta_{2} \|_{L^{q}}
\\[6pt]
&\lesssim 
t_{1}^{-1}\alpha(t_{1})^{\Theta_{q}}
|t_{1}-t_{2}|
\Big\{ 
\alpha(t_{1})^{-\nu_{q}}
+
R^{p-1} \alpha(t_{1})^{\frac{(d-2)(p-1)}{2}-1}
\Big\}
R \alpha(t_{1})^{\Theta_{q}}
\\[6pt]
&\lesssim 
t_{1}^{-1} |t_{1}-t_{2}|
\Big\{ 
\alpha(t_{1})^{\Theta_{q}-\nu_{q}-\theta_{q}}
+
\alpha(t_{1})^{\Theta_{q}+\frac{(d-2)(p-1)}{2}-1-\theta_{q}}
\Big\}
\alpha(t_{1})^{\Theta_{q}}
\\[6pt]
&\lesssim 
\alpha(t_{1})^{\theta_{q}}
t_{1}^{-1}\alpha(t_{1})^{\Theta_{q}}  |t_{1}-t_{2}|
=
o_{t_{1}}(1) t_{1}^{-1}\alpha(t_{1})^{\Theta_{q}}  |t_{1}-t_{2}|
,
\end{split} 
\end{equation}
where the implicit constants depend only on $d$, $p$ and $q$.  
  Plugging \eqref{21/2/5/17:19},  \eqref{21/2/5/17:10} and \eqref{20/11/2/10:15}   into  \eqref{20/10/25/13:51},  
 we find that
 \begin{equation}\label{21/2/6/16:29}
\begin{split}
&
\| 
(-\Delta+ \alpha( \tau_{1})  )^{-1} 
\big\{ 
F(\eta_{1};\alpha(\tau_{1}), t_{1})
-
F(\eta_{2};\alpha(\tau_{2}), t_{2})
\big\} \|_{L^{q}}
\\[6pt]
&\lesssim 
o_{t_{1}}(1) t_{1}^{-1}\alpha(t_{1})^{\Theta_{q}}  |t_{1}-t_{2}|
+
|t_{1}-t_{2}|
+
\alpha(t_{1})^{\theta_{q}}
\|\eta_{1} -\eta_{2} \|_{L^{q}}
,
\end{split}
\end{equation} 
where the implicit constants depend only on $d$, $p$ and $q$.

%%%%%%%%%%%%%%%%%%%%%%%%%%%%%%%%%%%%%%

Move on to the second term on the right-hand side of \eqref{20/11/2/9:6}. 
 By Lemma  \ref{20/8/16/15:45},  
 \eqref{19/01/03/12:15} in  Lemma \ref{19/01/03/16:41}, 
 Lemma \ref{18/11/05/10:29},
 \eqref{20/11/2/9:39}
 and $\delta(\alpha(t_{1}))=t_{1}^{-1}\alpha(t_{1})$ (see \eqref{20/11/2/9:40}),    
we see that 
\begin{equation}\label{20/11/3/15:47}
\begin{split}
&
\| 
\big\{ (-\Delta+ \alpha(\tau_{1})  )^{-1} 
-
(-\Delta+\alpha(\tau_{2}))^{-1} 
\big\} 
F(\eta_{2};\alpha(\tau_{2}), t_{2})
\|_{L^{q}}
\\[6pt]
&=
\big| 
\alpha( \tau_{1} ) - \alpha(\tau_{2}) 
\big|
\| 
 (-\Delta+ \alpha( \tau_{1})  )^{-1} 
 (-\Delta+ \alpha( \tau_{2})  )^{-1} 
F(\eta_{2};\alpha(\tau_{2}), t_{2})
\|_{L^{q}}
\\[6pt]
&\lesssim
\delta( \alpha(t_{1})) 
| \tau_{1} - \tau_{2} |
\alpha(t_{1})^{-1}
\| 
 (-\Delta+ \alpha(\tau_{1})  )^{-1} 
F(\eta_{2};\alpha(\tau_{2}), t_{2})
\|_{L^{q}}
\\[6pt]
&=
t_{1}^{-1}
| \tau_{1} - \tau_{2} |
\| 
 (-\Delta+ \alpha(\tau_{1})  )^{-1} 
F(\eta_{2};\alpha(\tau_{2}), t_{2})
 \|_{L^{q}}
,
\end{split}
\end{equation}
where the implicit constants depend only on $d$, $p$ and $q$.  
 Furthermore,  by \eqref{19/01/13/15:17},    
 Lemma \ref{18/11/05/10:29},  Lemma \ref{18/11/23/17:17},
 Lemma \ref{19/02/10/16:33},   \eqref{21/2/6/12}, 
  \eqref{20/11/2/9:39},  and \eqref{19/02/19/12:01} with $\theta=\theta_{q}$, 
 we see  that 
\begin{equation}\label{20/11/3/17:32}
\begin{split} 
&
\| 
 (-\Delta+ \alpha( \tau_{1})  )^{-1} 
F(\eta_{2};\alpha(\tau_{2}), t_{2})
 \|_{L^{q}}
\\[6pt]
&\lesssim 
\alpha(\tau_{2}) 
\| 
 (-\Delta+ \alpha( \tau_{1})  )^{-1} 
W \|_{L^{q}}
+
t_{2} 
\| 
 (-\Delta+ \alpha( \tau_{1})  )^{-1} 
W^{p} \|_{L^{q}}
\\[6pt]
&\quad +
\| 
 (-\Delta+ \alpha( \tau_{1})  )^{-1} 
N(\eta_{2}; t_{2}) \|_{L^{q}}
\\[6pt]
&\lesssim
\alpha(\tau_{2})
\alpha(\tau_{1})^{\Theta_{q}-1} 
\|W\|_{L_{\rm weak}^{\frac{d}{d-2}}}
+
t_{2} \|W^{p}\|_{L^{\frac{dq}{d+2q}}}
\\[6pt]
&\quad + 
R^{2}\alpha(t_{2})^{2\Theta_{q}} 
+
\alpha(\tau_{1})^{\frac{2^{*}}{q}-1} 
R^{\frac{d+2}{d-2}}
\alpha(t_{2})^{\frac{d+2}{d-2}\Theta_{q}}  
\\[6pt]
&\quad +
t_{2}
\big\{  
\alpha(\tau_{1})^{-\nu_{q}}
R \alpha(t_{2})^{\Theta_{q}}
+
\alpha(\tau_{1})^{\frac{d(p-1)}{2q}-1}
R^{p}  
\alpha(t_{2})^{p \Theta_{q}} 
\big\} 
\\[6pt]
&\lesssim
\alpha(t_{1})^{\Theta_{q}} 
+
t_{1} 
+ 
R^{2} \alpha(t_{1})^{2\Theta_{q}}
+
R^{\frac{d+2}{d-2}}\alpha(t_{1})^{\Theta_{q}+1}
\\[6pt]
&\quad +
t_{1}\{ R  \alpha(t_{1})^{\Theta_{q}-\nu_{q}}
+
R^{p}  \alpha(t_{1})^{\Theta_{q}+\frac{(d-2)(p-1)}{2}-1} 
\}
\\[6pt]
&\lesssim
\alpha(t_{1})^{\Theta_{q}} 
+
t_{1} 
+ 
\alpha(t_{1})^{2\Theta_{q}-\theta_{q}}
+
\alpha(t_{1})^{\Theta_{q}+1-\theta_{q} }
\\[6pt]
&\quad +
t_{1}\{ \alpha(t_{1})^{\Theta_{q}-\nu_{q}-\theta_{q}}
+
 \alpha(t_{1})^{\Theta_{q}+\frac{(d-2)(p-1)}{2}-1-\theta_{q}} 
\}
\\[6pt]
&\lesssim 
\alpha(t_{1})^{\Theta_{q}} + t_{1} +t_{1}\alpha(t_{1})^{\theta_{q}}
,
\end{split}
\end{equation}
where the implicit constants depend only on $d$, $p$ and $q$.
 Plugging \eqref{20/11/3/9:26} and \eqref{20/11/3/17:32} into \eqref{20/11/3/15:47}, 
  and using $t_{1} \le \alpha(t_{1})^{\Theta_{q}}$ (see \eqref{21/2/6/15:27}), 
 we see that 
\begin{equation}\label{20/11/4/10:37}
\begin{split}
&
\| 
\big\{ (-\Delta+ \alpha(\tau_{1})  )^{-1} 
-
(-\Delta+\alpha(\tau_{2}))^{-1} 
\big\} 
F(\eta_{2};\alpha(\tau_{2}), t_{2})
\|_{L^{q}}
\\[6pt]
&\lesssim
t_{1}^{-1}
\big\{  
o_{t_{1}}(1) |t_{2}-t_{1}|
+
t_{1}
\alpha (t_{1})^{-\Theta_{q}+\theta_{q}}
\|\eta_{2} -\eta_{1}\|_{L^{q}}
\big\}
\big\{ 
\alpha(t_{1})^{\Theta_{q}} 
+
t_{1} 
+ 
t_{1} \alpha(t_{1})^{\theta_{q}}
\big\}
\\[6pt]
&\lesssim
o_{t_{1}}(1)  t_{1}^{-1} 
\alpha(t_{1})^{\Theta_{q}}
|t_{1}-t_{2}| 
+
|t_{1}-t_{2}| 
+ 
\alpha(t_{1})^{\theta_{q}} 
\|\eta_{1} -\eta_{2} \|_{L^{q}}
,
\end{split}
\end{equation}
where the implicit constants depend only on $d$, $p$ and $q$.

Now,  putting  \eqref{21/2/5/16:21}, \eqref{20/11/2/9:6}, \eqref{21/2/6/16:29} and \eqref{20/11/4/10:37} together,    
 we find that \eqref{20/11/5/9:5} holds.

%%%%%%%%%%%%%%%%

\noindent 
{\bf Proof of  \eqref{20/11/11/11:25}.}~For notational convenience, we put 
\begin{equation}\label{21/2/7/12:8}
\mathfrak{g}_{2}:=
\mathfrak{g}(t_{2};\tau_{2},\eta_{2})
. 
\end{equation}
 In order to treat the singularity comes from $\mathscr{A}_{1}^{-1}$, we  will use 
 the operator $\Pi$ introduced in Section \ref{18/11/11/10:47}. 
  Then, by Lemma \ref{20/8/16/15:45},  
 \eqref{19/01/03/12:15} in Lemma \ref{19/01/03/16:41}  
 and \eqref{20/11/2/9:39},  we see that    
\begin{equation}\label{20/10/26/11:43}
\begin{split} 
&
\| 
\mathscr{A}_{1}^{-1} 
\big\{
(-\Delta+ \alpha(\tau_{2}))^{-1} 
-
(-\Delta+ \alpha(\tau_{1}))^{-1} 
\big\} 
V \mathfrak{g}(t_{2};\tau_{2},\eta_{2})
 \|_{L^{q}}
\\[6pt]
&=
| \alpha( \tau_{1}) - \alpha( \tau_{1})  |
\| \mathscr{A}_{1}^{-1} 
(-\Delta+ \alpha( \tau_{2}))^{-1} 
(-\Delta+ \alpha( \tau_{1}))^{-1} 
V \mathfrak{g}_{2}
\|_{L^{q}}
 \\[6pt]
&\lesssim 
\delta( \alpha(t_{1}) ) 
| \tau_{1} - \tau_{2} |
\|
\mathscr{A}_{1}^{-1} 
(-\Delta+ \alpha(\tau_{2}))^{-1} 
(-\Delta+ \alpha(\tau_{1}))^{-1} 
V \mathfrak{g}_{2} 
\|_{L^{q}},
\\[6pt]
&\le     
\delta( \alpha(t_{1}) ) 
| \tau_{1} - \tau_{2} | 
\| 
\mathscr{A}_{1}^{-1}
\Pi  
(-\Delta+ \alpha( \tau_{2}))^{-1} 
(-\Delta+ \alpha( \tau_{1}))^{-1} 
V \mathfrak{g}_{2} 
\|_{L^{q}}
\\[6pt]
&\quad +
\delta( \alpha(t_{1}) ) 
|\tau_{1} -\tau_{2} |
\|\mathscr{A}_{1}^{-1} 
(1-\Pi) 
(-\Delta+ \alpha( \tau_{2}))^{-1} 
(-\Delta+ \alpha( \tau_{1}))^{-1} 
V \mathfrak{g}_{2} \|_{L^{q}}
,
\end{split}
\end{equation}
where the implicit constants depend only on $d$ and $q$. 

%%%%%%%%%%%

Consider the first term on the right-hand side of \eqref{20/10/26/11:43}.  
 By \eqref{19/02/02/11:47} in Proposition \ref{18/11/17/07:17}, and \eqref{20/11/2/9:39}, 
 we see that  
\begin{equation}\label{20/11/7/11:19}
\begin{split}
&
\delta( \alpha(t_{1}) ) 
| \tau_{1} - \tau_{2} | 
\| 
\mathscr{A}_{1}^{-1}
\Pi  
(-\Delta+ \alpha( \tau_{2}))^{-1} 
(-\Delta+ \alpha( \tau_{1}))^{-1} 
V \mathfrak{g}_{2}
\|_{L^{q}}
\\[6pt]
&\lesssim 
\delta(\alpha(t_{1}))^{2} 
\alpha(t_{1})^{-1}  
|\tau_{1} - \tau_{2} | 
\| \Pi  (-\Delta+ \alpha( \tau_{2}))^{-1} 
(-\Delta+ \alpha( \tau_{1}))^{-1} 
V \mathfrak{g}_{2}\|_{L^{q}}
,
\end{split}
\end{equation} 
where the implicit constants depend only on $d$ and $q$. 
Observe from the definition of $\Pi$ (see \eqref{19/01/27/15:38}) 
and $V\Lambda W =\Delta \Lambda W =- (-\Delta + \alpha( \tau_{2}) )\Lambda W +  \alpha(\tau_{2}) \Lambda W$
 that 
\begin{equation}\label{20/11/10/9:57}
\begin{split}
&
\| 
\Pi  
(-\Delta+ \alpha( \tau_{2}))^{-1} 
(-\Delta+ \alpha( \tau_{1}))^{-1} 
V \mathfrak{g}_{2}
\|_{L^{q}}
\\[6pt]
&\lesssim 
\big|
\langle (-\Delta+ \alpha(\tau_{2}))^{-1} 
(-\Delta+ \alpha(\tau_{1}))^{-1} 
V \mathfrak{g}_{2},\,  
V \Lambda W \rangle 
\big|
\\[6pt]
&\le   
|\langle (-\Delta +\alpha( \tau_{1}) )^{-1} V \mathfrak{g}_{2}, \,  
\Lambda W \rangle|
\\[6pt]
&\quad + 
|\langle (-\Delta +\alpha(\tau_{1}))^{-1} V  \mathfrak{g}_{2},\,   
\alpha(\tau_{2}) (-\Delta +\alpha(\tau_{2}))^{-1} \Lambda W \rangle|
,
\end{split}
\end{equation} 
where the implicit constants depend only on $d$ and $q$. 
 Applying Lemma \ref{18/12/02/11:25} to the first term on the right-hand side of \eqref{20/11/10/9:57},  
and using $V(x) \lesssim (1+|x|)^{-4}$(see \eqref{20/9/13/9:58}) and \eqref{20/11/2/9:39}, 
 we see that 
\begin{equation}\label{20/11/10/10:6}
|\langle (-\Delta +\alpha( \tau_{1}) )^{-1} V \mathfrak{g}_{2}, \, 
\Lambda W \rangle |
\lesssim  
\delta(  \alpha( \tau_{1}) )^{-1} \| V \mathfrak{g}_{2} \|_{L^{1}}
+
\| \mathfrak{g}_{2}  \|_{L^{q}}
\lesssim 
\delta(  \alpha( t_{1}) )^{-1} 
\| \mathfrak{g}_{2}  \|_{L^{q}}, 
\end{equation}
where the implicit constants depend only on $d$ and $q$.   On the other hand, 
 by H\"older's inequality, Lemma \ref{18/11/05/10:29},  \eqref{20/11/2/9:39} 
 and $V(x) \lesssim (1+|x|)^{-4}$,   
 the second term on the right-hand side of \eqref{20/11/10/9:57} is estimated as   
\begin{equation}\label{20/11/10/10:39}
\begin{split}
&
|\langle (-\Delta +\alpha(\tau_{1}))^{-1} V  \mathfrak{g}_{2},  \, 
\alpha(\tau_{2}) (-\Delta +\alpha(\tau_{2}))^{-1} \Lambda W \rangle|
\\[6pt]
&\lesssim 
\| (-\Delta +\alpha(\tau_{1})  )^{-1} V  \mathfrak{g}_{2}  
\|_{L^{\frac{dq}{d+2q}}} 
\alpha( \tau_{2})
\| (-\Delta +\alpha(\tau_{2}))^{-1} \Lambda W 
\|_{L^{\frac{dq}{(d-2)q-d}}}
\\[6pt]
&\lesssim 
\alpha( \tau_{1})^{\frac{d-2}{2}-\frac{d}{2q}-1}
\| V  \mathfrak{g}_{2} \|_{L^{1}} 
\alpha( \tau_{2}) ^{\frac{d}{2q}}
\| \Lambda W \|_{L_{\rm weak}^{\frac{d}{d-2}}}
\lesssim
\alpha(t_{1})^{\frac{d-2}{2}-1}
\| \mathfrak{g}_{2}  \|_{L^{q}} 
,
\end{split}
\end{equation} 
where the implicit constants depend only on $d$ and $q$. 
Putting   \eqref{20/11/7/11:19},  \eqref{20/11/10/9:57},  \eqref{20/11/10/10:6}  and  \eqref{20/11/10/10:39} together, 
  and using $\delta(\alpha(t_{1}))=t_{1}^{-1}\alpha(t_{1})$ (see \eqref{20/11/2/9:40}) and \eqref{19/01/27/16:58}, 
 we see that 
\begin{equation}\label{20/11/5/11:10}
\begin{split}
&
\delta( \alpha(t_{1}) ) 
| \tau_{1} - \tau_{2} | 
\| 
\mathscr{A}_{1}^{-1}
\Pi  
(-\Delta+ \alpha( \tau_{2}))^{-1} 
(-\Delta+ \alpha( \tau_{1}))^{-1} 
V \mathfrak{g}_{2}
\|_{L^{q}}
\\[6pt]
&\lesssim 
\delta(\alpha(t_{1}))^{2} 
\alpha(t_{1})^{-1}   
|\tau_{1} -  \tau_{2} | 
\big\{
\delta(\alpha(t_{1}))^{-1} 
+
\alpha(t_{1})^{\frac{d-2}{2}-1} 
\big\}
\|   \mathfrak{g}_{2}\|_{L^{q}}
\\[6pt]
&=
|\tau_{1} - \tau_{2} |  
\big\{
t_{1}^{-1} 
+
\delta(\alpha(t_{1})) t_{1}^{-1} \alpha(t_{1})^{\frac{d-2}{2}-1} 
\big\}
\|   \mathfrak{g}_{2}\|_{L^{q}}
\lesssim
| \tau_{1} -  \tau_{2} |  
t_{1}^{-1} 
\|   \mathfrak{g}_{2}  \|_{L^{q}}
,
\end{split}
\end{equation}
where the implicit constants depend only on $d$ and $q$. 

%%%%%%%%%%%%%%%%%

Move on to the second term on the right-hand side of \eqref{20/10/26/11:43}. 
Observe from the definition of $\Pi$ (see \eqref{19/01/27/15:38}) that 
 \begin{equation}\label{20/11/10/11:17}
 \|1-\Pi\|_{L^{q}\to L^{q}}\lesssim 1, 
\end{equation}
where the implicit constant depends only on $d$ and $q$.  
 Note  that the functions appear in what follows are radial, and recall from Lemma \ref{19/01/31/14:17} that $X_{q}=(1-\Pi) L_{\rm rad}^{q}(\mathbb{R}^{d})$.
Then, by \eqref{18/11/11/15:50} in Proposition \ref{18/11/17/07:17},  
 \eqref{20/11/10/11:17},  Lemma \ref{18/11/23/17:17},   
 Lemma \ref{18/11/05/10:29},  \eqref{20/11/2/9:39} 
 and $\delta(\alpha(t_{1}))=t_{1}^{-1}\alpha(t_{1})$ (see \eqref{20/11/2/9:40}),   
 we see that  
\begin{equation}\label{20/11/10/12:7}
\begin{split}
&
\delta( \alpha(t_{1}) ) 
|\tau_{1} -\tau_{2} |
\|\mathscr{A}_{1}^{-1} 
(1-\Pi) 
(-\Delta+ \alpha( \tau_{2}))^{-1} 
(-\Delta+ \alpha( \tau_{1}))^{-1} 
V \mathfrak{g}_{2} \|_{L^{q}}
\\[6pt]
&\lesssim 
\delta( \alpha(t_{1}) ) 
|\tau_{1} -\tau_{2} |
\| (-\Delta+ \alpha( \tau_{2}))^{-1} 
(-\Delta+ \alpha( \tau_{1}))^{-1} 
V \mathfrak{g}_{2} \|_{L^{q}}
\\[6pt]
&\lesssim 
\delta( \alpha(t_{1}) ) 
|\tau_{1} -\tau_{2} |
\| (-\Delta+ \alpha(\tau_{1}))^{-1} 
V \mathfrak{g}_{2} 
\|_{L^{\frac{dq}{d+2q}}}
\\[6pt]
&\lesssim 
\delta( \alpha(t_{1}) ) 
|\tau_{1} -\tau_{2} |
\alpha( \tau_{1})^{\frac{d-2}{2}-\frac{d}{2q}-1}  
\| V \mathfrak{g}_{2} \|_{L^{1}}
\\[6pt]
&\lesssim
\delta( \alpha(t_{1}) ) 
|\tau_{1} -\tau_{2} |
\alpha(t_{1})^{\Theta_{q}-1}  
\|\mathfrak{g}_{2}  \|_{L^{q}} 
=
|\tau_{1} -\tau_{2} |
t_{1}^{-1} \alpha(t_{1})^{\Theta_{q}}  
\|\mathfrak{g}_{2}  \|_{L^{q}} 
,
\end{split} 
\end{equation}
where the implicit constants depend only on $d$ and $q$.  
 
Plugging  \eqref{20/11/5/11:10} and \eqref{20/11/10/12:7} into \eqref{20/10/26/11:43}, 
 and using \eqref{20/11/3/9:26} and \eqref{20/11/7/13:2}, 
 we see that  
\begin{equation}\label{21/2/7/15:12}
\begin{split} 
&
\| 
\mathscr{A}_{1}^{-1} 
\big\{
(-\Delta+ \alpha(\tau_{2}))^{-1} 
-
(-\Delta+ \alpha(\tau_{1}))^{-1} 
\big\} 
V \mathfrak{g}(t_{2};\tau_{2},\eta_{2})
 \|_{L^{q}}
\\[6pt]
&\lesssim
 |\tau_{1}-\tau_{2}| t_{1}^{-1}
\|   \mathfrak{g}_{2} \|_{L^{q}}
+
|\tau_{1} -\tau_{2} |
t_{1}^{-1} \alpha(t_{1})^{\Theta_{q}}  
\|\mathfrak{g}_{2}  \|_{L^{q}} 
\lesssim 
|\tau_{1} -\tau_{2} |
t_{1}^{-1} 
\|\mathfrak{g}_{2}  \|_{L^{q}} 
\\[6pt]
&\lesssim 
\big\{ o_{t_{1}}(1) |t_{2}-t_{1}|
+
t_{1}
\alpha (t_{1})^{-\Theta_{q}+\theta_{q}}
\|\eta_{2} -\eta_{1}\|_{L^{q}} \big\}
t_{1}^{-1} 
\|   \mathfrak{g}_{2} \|_{L^{q}}
\\[6pt]
&\lesssim 
o_{t_{1}}(1)t_{1}^{-1}  \alpha(t_{1})^{\Theta_{q}} |t_{2}-t_{1}|
+
\alpha (t_{1})^{\theta_{q}}
\|\eta_{2} -\eta_{1}\|_{L^{q}} 
, 
\end{split}
\end{equation}
where the implicit constant depends only on $d$, $p$ and $q$. 
 Thus, we have proved \eqref{20/11/11/11:25} and completed the proof of the lemma. 

\end{proof}

%%%%%%%%%%%%

Now, we are in a position to  prove Proposition \ref{19/01/21/08:01}:

\begin{proof}[Proof of Proposition \ref{19/01/21/08:01}]
First,  we claim  that  if $t>0$ is sufficiently small depending only on $d$, $p$, $q$ and $R$, 
 then for any $\eta \in Y_{q}(R,t)$,  the interval $I(t)$ admits one and only one element $\tau(t;\eta)$ such that 
\begin{equation}\label{21/2/7/17:55}
\tau(t;\eta)=\mathfrak{s}(t; \tau(t;\eta), \eta).
\end{equation}   
We prove this claim by applying the Banach fixed-point theorem to the mapping $\tau \in I(t) \mapsto \mathfrak{s}(t; \tau,\eta)$, as in \cite{CG}. 
  By Lemma \ref{20/10/20/10:23} (see also Remark \ref{20/11/15/16:23}), 
  we see that if $t$ is sufficiently small depending only on $d$, $p$, $q$ and $R$,   
   then  for any $\eta \in Y_{q}(R,t)$ and $\tau \in I(t)$,   $\mathfrak{s}(t; \tau, \eta) \in I(t)$.    
    Furthermore,  applying Lemma \ref{20/10/27/11:21} as $t_{1}=t_{2}=t$ and $\eta_{1}=\eta_{2}=\eta$ (see also Remark \ref{20/11/6/11:50}),
     we see that  if $t$ is sufficiently small depending only on $d$, $p$, $q$ and $R$, then for any $\eta \in Y_{q}(R,t)$ and $\tau_{1},\tau_{2}\in I(t)$, 
\begin{equation}\label{21/2/7/17:18}
\big| 
\mathfrak{s}(t; \tau_{2}, \eta)
-
\mathfrak{s}(t; \tau_{1},  \eta)
\big|
\lesssim  
o_{t}(1)
|\tau_{1}- \tau_{2} |
,
\end{equation} 
where the implicit constant depends only on $d$, $p$ and $q$.  
 Then,  the Banach fixed-point theorem 
 shows that  the claim is true.  
 
%%%%%%%%%%%%%%%%%%%%%%%%%%%%
 
Next, we claim that if $t>0$ is sufficiently small depending only on $d$, $p$, $q$ and $R$, 
then $Y_{q}(R,t)$ admits one and only one element $\eta_{t}$ such that 
\begin{equation}\label{21/2/7/18:3}
\eta_{t}=\mathfrak{g}(t; \tau(t; \eta_{t}), \eta_{t}),
\end{equation}
where $\tau(t;\eta)$ denotes the unique number in $I(t)$ satisfying \eqref{21/2/7/17:55}. 
 Indeed,  applying Lemma \ref{20/10/30/11:47} (use \eqref{20/10/30/11:48} as $t_{j}=t$ and $\tau_{j}=\tau(t;\eta_{j})$),  we see that the mapping  $\eta \in Y_{q}(R,t) \mapsto  \mathfrak{g}(t; \tau(t; \eta), \eta)$ is contraction in the metric induced from $L^{q}(\mathbb{R}^{d})$.  
 Hence, the claim follows from  the Banach fixed-point theorem.  

%%%%%%%%%%%%%%%%%

By the above claims, we conclude that if $t>0$ is sufficiently small depending only on $d$, $p$, $q$ and $R$, then 
  $ I(t) \times Y_{q}(R,t)$ admits at least one solution to \eqref{19/01/13/14:55}.  
 Furthermore,  the uniqueness of such a solution follows from  Lemma  \ref{20/10/27/11:21}  and Lemma \ref{20/10/30/11:47}; 
 indeed,  by these lemmas and \eqref{21/2/6/15:27}, 
 we see that if $(\tau_{t}, \eta_{t})$ and $(\widetilde{\tau}_{t}, \widetilde{\eta}_{t})$ are solutions to \eqref{19/01/13/14:55} in $I(t) \times Y_{q}(R,t)$,
 then 
\begin{equation}\label{21/2/8/9:46}
\begin{split}
&
|\tau_{t}-\widetilde{\tau}_{t}|
+
\|\eta_{t} - \widetilde{\eta}_{t}\|_{L^{q}}
\\[6pt]
&=
| \mathfrak{s}(t; \tau_{t}, \eta_{t})
-
\mathfrak{s}(t; \widetilde{\tau}_{t},  \widetilde{\eta}_{t}) |
+
\| 
\mathfrak{g}(t;  \tau_{t}, \eta_{t}) 
-
\mathfrak{g}(t;  \widetilde{\tau}_{t},  \widetilde{\eta}_{t}) \|_{L^{q}}
\\[6pt]
&\lesssim  
o_{t}(1)
|\tau_{t}- \widetilde{\tau}_{t} |
+
\alpha(t)^{\theta_{q}} 
\|\eta_{t} -\widetilde{\eta}_{t} \|_{L^{q}}
,
\end{split}
\end{equation} 
where the implicit constants depend only on $d$, $p$ and $q$.  This estimate implies 
that  if $t>0$ is sufficiently small depending only on $d$, $p$, $q$ and $R$, then the uniqueness holds. 
 
%%%%%%%%%%%%%%%%

Now,  for  a sufficiently small $t>0$,  we  use $(\tau_{t}, \eta_{t})$ to denote the unique solution to \eqref{19/01/13/14:55} in $I(t)\times Y_{q}(R,t)$. 
 
We shall show that  $\tau_{t}$ and $\eta_{t}$ are continuous with respect to $t$.  
   Let $t_{0},t>0$ be sufficiently small numbers with  $\frac{1}{2}t \le t_{0} \le 2t$.   
   Then,  by Lemma  \ref{20/10/27/11:21} (see also Remark \ref{20/11/6/11:50}), \eqref{21/2/6/15:27}  and  Lemma \ref{20/10/30/11:47}, 
 we see that 
\begin{equation}\label{20/11/19/9:57}
\begin{split}
&
|\tau_{t} - \tau_{t_{0}}|
+
\| \eta_{t} - \eta_{t_{0}} \|_{L^{q}}
\\[6pt]
&=
|\mathfrak{s}(t;\tau_{t},\eta_{t}) - \mathfrak{s}(t_{0};\tau_{t_{0}}, \eta_{t_{0}})|
+
\| \mathfrak{g}(t; \tau_{t}, \eta_{t}) - \mathfrak{g}(t_{0}; \tau_{t_{0}}, \eta_{t_{0}}) \|_{L^{q}}
\\[6pt]
&\lesssim 
o_{t_{0}}(1) |\tau_{t} - \tau_{t_{0}}| 
+ 
\{ 1+o_{t_{0}}(1) t_{0}^{-1}\alpha(t_{0})^{\Theta_{q}} \} |t-t_{0}| 
+
\alpha(t_{0})^{\theta_{q}}
\| \eta_{t} - \eta_{t_{0}} \|_{L^{q}}
,
\end{split}
\end{equation}
where the implicit constant depends only on $d$, $p$ and $q$. Note that if $t_{0}$ is sufficiently small, then
 the first and last terms on the right-hand side of \eqref{20/11/19/9:57} can be absorbed into the left-hand side. 
 Thus,   \eqref{20/11/19/9:57} implies the continuity of $\tau_{t}$ and $\eta_{t}$.

It remains to prove  that  $\tau_{t}$ is strictly increasing with respect to $t$. Let $0< t_{1} <t_{2}$ be sufficiently small numbers to be specified later, dependently on $d$, $p$, $q$ and $R$.   
 By the continuity of $\tau_{t}$ with respect to $t$,  
 we may assume that $t_{1}$ and $t_{2}$ are sufficiently close; In particular, we may assume $t_{2} \le 2t_{1}$. Then, by $\tau_{t_{j}}=\mathfrak{s}(t_{j};\tau_{t_{j}},\eta_{t_{j}})$, Lemma \ref{20/10/27/11:21} (see also Remark \ref{20/11/6/11:50}), \eqref{21/2/6/15:27},  
 we see that  
\begin{equation}\label{20/11/19/11:9}
\begin{split} 
\tau_{t_{2}} - \tau_{t_{1}} 
&\ge 
\frac{K_{p}}{A_{1}}(t_{2} -t_{1})
-
o_{t_{1}}(1) |t_{2}-t_{1}| 
-
o_{t_{1}}(1) |\tau_{t_{2}}-\tau_{t_{1}}| 
\\[6pt]
&\quad -
C_{0} t_{1} \alpha(t_{1})^{-\Theta_{q}+\theta_{q}} 
\| \eta_{t_{2}} -\eta_{t_{1}} \|_{L^{q}},
\end{split} 
\end{equation}
where $C_{0}>0$ is some constant depending only on $d$, $p$ and $q$. 
 Furthermore, we find from Lemma \ref{20/10/30/11:47} 
 and $\eta_{t_{j}}=\mathfrak{g}(t_{j}; \tau_{t_{j}}, \eta_{t_{j}})$
 that 
\begin{equation}\label{20/11/19/12}
\begin{split}
\| \eta_{t_{2}} - \eta_{t_{1}} \|_{L^{q}}
&\lesssim 
o_{t_{1}}(1) t_{1}^{-1} 
\alpha(t_{1})^{\Theta_{q}}
|t_{2}-t_{1}| 
+
|t_{2}-t_{1}|
,
\end{split} 
\end{equation}
where the implicit constant depends only on $d$, $p$ and $q$.  
 Plugging \eqref{20/11/19/12} into \eqref{20/11/19/11:9}, and using 
 \eqref{21/2/6/15:27}, we see that 
\begin{equation}\label{21/2/8/14:59}
\begin{split}
\tau_{t_{2}} - \tau_{t_{1}} 
&\ge 
\frac{K_{p}}{A_{1}}(t_{2} -t_{1})
-
o_{t_{1}}(1) |t_{2}-t_{1}| 
-
o_{t_{1}}(1) |\tau_{t_{2}}-\tau_{t_{1}}| 
\\[6pt]
&\quad -
C_{0} o_{t_{1}}(1) \alpha(t_{1})^{\theta_{q}} 
|t_{2}-t_{1}| 
+
C_{0} \alpha(t_{1})^{\theta_{q}} 
|t_{2}-t_{1}|
,
\end{split}
\end{equation}
where the implicit constant depends only on $d$, $p$ and $q$. 
This implies that $\tau_{t}$ is strictly increasing with respect to $t$.
 Thus, we have completed the proof of the proposition. 
\end{proof}

%%%%%%%%%%%%%%%%%%%%%%%%%%%%%%%%%%%%%%%%%%%%%%%%%%%

\subsection{Proof of Proposition \ref{19/02/11/16:56}}\label{19/02/17/09:57}

In this section, we prove  Proposition \ref{19/02/11/16:56}. 

Recall that $\widetilde{\mathcal{S}}_{t}$ and $\widetilde{\mathcal{N}}_{t}$ denote the action and Nehari functional associated with \eqref{19/02/11/17:03},     respectively (see \eqref{20/11/25/7:1} and \eqref{20/11/25/7:2}).   
 Furthermore, recall that $\widetilde{\mathcal{G}}_{t}$ denotes the set of positive radial minimizers for \eqref{19/02/11/17:03}.

The existence of a minimizer for \eqref{19/02/11/17:03} follows from that for \eqref{eq:1.1} via the $H^{1}$-scaling (see \eqref{20/12/6/16:37});  
 Precisely, we have:  
\begin{lemma}\label{20/9/21/10:29}
Assume $d=3,4$ and $\frac{4}{d-2}-1<p<\frac{d+2}{d-2}$.   
 Let $\max\{2^{*},\frac{2^{*}}{p-1}\}< q <\infty$ and $R>0$, and let $T_{1}(q, R)$ be the number given in Proposition \ref{19/01/21/08:01}; hence 
  for any $0<t <T_{1}(q,R)$,  $I(t) \times Y_{q}(R,t)$ admits  a unique solution to \eqref{19/01/13/14:55}, say $(\tau_{t}, \eta_{t})$.
 Furthermore,  let $0<t <T_{1}(q,R)$,  and define $\lambda(t)$ and $\omega(t)$ by   
\begin{equation}\label{20/9/21/16:19}
\lambda(t)^{-2^{*}+p+1} 
=t ,
\qquad 
\omega(t)
=
\alpha(\tau_{t}) \lambda(t)^{2^{*}-2}
=
\alpha(\tau_{t}) t^{\frac{-(2^{*}-2)}{ 2^{*}- (p+1)}}. 
\end{equation}
Then, the following hold:
\begin{enumerate}
\item 
Let $\Phi \in \mathcal{G}_{\omega(t)}$; Note that  by Proposition \ref{proposition:1.1},  $\mathcal{G}_{\omega(t)} \neq \emptyset$.
 Then,  $T_{\lambda(t)}[\Phi] \in \widetilde{\mathcal{G}}_{t}$; Hence,  $\widetilde{\mathcal{G}}_{t} \neq \emptyset$.

\item  
Let $u \in \widetilde{\mathcal{G}}_{t}$. 
Then,  $T_{\lambda(t)^{-1}}[u] \in \mathcal{G}_{\omega(t)}$, and 
\begin{equation}\label{20/9/22/11:38}
\widetilde{\mathcal{S}}_{t}(u) 
=
\frac{1}{d} \|\nabla u \|_{L^{2}}^{2}
\le 
\frac{1}{d} \|\nabla W\|_{L^{2}}^{2}
 . 
\end{equation}
\end{enumerate} 
\end{lemma}

\begin{proof}[Proof of Lemma \ref{20/9/21/10:29}]
We shall prove the first claim.  
 Observe from a computation involving the scaling that 
\begin{equation}\label{20/11/21/13:22}
\widetilde{\mathcal{N}}_{t}( T_{\lambda(t)}[\Phi])
=
\mathcal{N}_{\omega(t)}(\Phi)
=0
,
\end{equation}
which implies that 
\begin{equation}\label{20/9/21/16:24}
\widetilde{\mathcal{S}}_{t}( T_{\lambda(t)}[\Phi])
\ge 
\inf\big\{ 
\widetilde{\mathcal{S}}_{t}(u) 
 \colon u \in H^{1}(\mathbb{R}^{d})\setminus \{0\},
\  
\widetilde{\mathcal{N}}_{t}(u)=0 
\big\}
.
\end{equation} 
If the equality failed in \eqref{20/9/21/16:24}, then we could take a function $u_{0} \in H^{1}(\mathbb{R}^{d})\setminus \{0\}$ such that  
\begin{equation}\label{20/9/21/16:45}
\widetilde{\mathcal{S}}_{t}( u_{0} )
<
\widetilde{\mathcal{S}}_{t}( T_{\lambda(t)}[\Phi] ), 
\qquad 
\widetilde{\mathcal{N}}_{t}(u_{0})
=0
.
\end{equation}
Furthermore,  \eqref{20/9/21/16:45} together with computations involving the scaling shows that 
\begin{align}
\label{20/9/21/16:57}
\mathcal{S}_{\omega(t)}(T_{\lambda(t)^{-1}}[u_{0}])
&=\widetilde{\mathcal{S}}_{t}( u_{0} )
<
\widetilde{\mathcal{S}}_{t}(T_{\lambda(t)}[\Phi] )
=
\mathcal{S}_{\omega(t)}(\Phi), 
\\[6pt]
\label{20/9/21/16:58} 
\mathcal{N}_{\omega(t)}(T_{\lambda(t)^{-1}}[u_{0}])
&=
\widetilde{\mathcal{N}}_{t}(u_{0})
=0
.
\end{align}
However, \eqref{20/9/21/16:57} together with \eqref{20/9/21/16:58} contradicts  $\Phi \in \mathcal{G}_{\omega(t)}$.  
 Thus, the equality must hold in \eqref{20/9/21/16:24}, which together with \eqref{20/11/21/13:22} shows that 
 $T_{\lambda(t)}[\Phi] \in \widetilde{\mathcal{G}}_{t}$.   
  
Move on to  the second claim.   
 As well as the first claim, we can verify  that   
  $T_{\lambda(t)^{-1}}[u] \in \mathcal{G}_{\omega(t)}$. 
 It remains to prove \eqref{20/9/22/11:38}. 
 Observe from computations involving the scaling that  
\begin{equation}\label{20/11/21/12:1}
\widetilde{\mathcal{S}}_{t}(u) 
=
\mathcal{S}_{\omega(t)} 
( T_{\lambda(t)^{-1} }[u]   ),
\qquad 
\|\nabla u \|_{L^{2}}
=
\|\nabla  T_{\lambda(t)^{-1} }[u]\|_{L^{2}}
. 
\end{equation}
Then,  by Lemma 2.2 of \cite{AIIKN},  and \eqref{20/11/21/12:1}, we find that \eqref{20/9/22/11:38} holds. 
\end{proof}

The following fact follows from  the second claim in Lemma \ref{20/9/21/10:29} (cf. Proposition \ref{proposition:1.1}):
\begin{corollary}\label{20/9/22/17:47}
Assume $d= 3,4$ and $\frac{4}{d-2}-1<p<\frac{d+2}{d-2}$. 
 Let $\max\{2^{*}, \frac{2^{*}}{p-1}\} <q <\infty$, $R>0$ 
 and let $T_{1}(q,R)$ be the number given in Proposition \ref{19/01/21/08:01}; 
 hence 
  for any $0<t <T_{1}(q,R)$,  $I(t) \times Y_{q}(R,t)$ admits  a unique solution to \eqref{19/01/13/14:55}, say $(\tau_{t}, \eta_{t})$.  
  Furthermore, let $0<t < T_{1}(q,R)$ and $u_{t}\in \widetilde{\mathcal{G}}_{t}$. 
 Then, $u_{t}$ is a positive radial solution to \eqref{19/02/25/12:07}. 
 Furthermore, $u_{t}\in C^{2}(\mathbb{R}^{d})\cap H^{2}(\mathbb{R}^{d})$, and $u_{t}(x)$ is strictly decreasing as a function of $|x|$.  
\end{corollary}

%%%%%%%%%%%%%%%%%%%
 
Next,  we give  properties of ground states to \eqref{19/02/25/12:07} (Lemma \ref{20/9/21/19:3} and Lemma \ref{20/9/22/16:51}):   

\begin{lemma}\label{20/9/21/19:3}
Assume  $d=3,4$ and $\frac{4}{d-2}-1<p<\frac{d+2}{d-2}$.  
 Let $\max\{2^{*},\frac{2^{*}}{p-1}\}<q <\infty$, $R>0$ and  
  let $T_{1}(q,R)$ be the number given in Proposition \ref{19/01/21/08:01}; 
 hence for any $0<t <T_{1}(q,R)$,  $I(t) \times Y_{q}(R,t)$ admits  a unique solution to \eqref{19/01/13/14:55}, say $(\tau_{t}, \eta_{t})$.  
  Then, the following hold:
\begin{align}
\label{20/9/22/11:09}
&
\lim_{t\to 0}  
\alpha(\tau_{t})^{ \frac{-\{2^{*}-(p+1)\}}{2^{*}-2} } 
t
=0,
\\[6pt]
\label{20/9/22/13:51}
&
\lim_{t\to 0}\sup_{u \in \widetilde{\mathcal{G}}_{t} } t \|u\|_{L^{p+1}}^{p+1}
=
\lim_{t\to 0}\sup_{u \in \widetilde{\mathcal{G}}_{t} } \alpha(\tau_{t})  \|u\|_{L^{2}}^{2}
=
0
.
\end{align}
\end{lemma}
\begin{proof}[Proof of Lemma \ref{20/9/21/19:3}]
By \eqref{20/10/15/16:46},  the assumption about $p$ and 
 $\tau_{t} \in I(t)$ (hence $\tau_{t} \sim t$), we see that  \eqref{20/9/22/11:09} holds.   

We shall prove \eqref{20/9/22/13:51}.  
 Let $0< t<T_{1}(q,R)$ be a number to be taken $t\to 0$, and let $u_{t}  \in \widetilde{\mathcal{G}}_{t}$.
  Note that $\widetilde{\mathcal{N}}_{t}(u_{t})=0$. Then, by \eqref{20/9/22/11:38} in Lemma \ref{20/9/21/10:29}, 
 we see that  
\begin{equation}\label{20/9/21/19:55}
\begin{split}
&
\frac{1}{d}\|\nabla W \|_{L^{2}}^{2}
\ge 
\widetilde{\mathcal{S}}_{t}(u_{t})
-
\frac{1}{p+1}\widetilde{\mathcal{N}}_{t}(u_{t})
\\[6pt]
&
=
\frac{p-1}{2(p+1)}
\alpha(\tau_{t})
\|u_{t}\|_{L^{2}}^{2}
+
\frac{p-1}{2(p+1)} \|\nabla u_{t}\|_{L^{2}}^{2}
+
\frac{ 2^{*}- (p+1) }{ 2^{*}(p+1) }
\|u_{t}\|_{L^{2^{*}}}^{2^{*}}
. 
\end{split} 
\end{equation}
Furthermore,  by H\"older's inequality, we see  that   
\begin{equation}\label{20/9/22/10:52}
\begin{split}
t \|u_{t}\|_{L^{p+1}}^{p+1}
&\le 
t \|u_{t}\|_{L^{2}}^{\frac{2\{ 2^{*}-(p+1)\}}{2^{*}-2}} 
\|u_{t}\|_{L^{2^{*}}}^{\frac{2^{*}(p-1)}{2^{*}-2}}
\\[6pt]
&= 
\alpha(\tau_{t})^{\frac{-\{ 2^{*}-(p+1) \}}{2^{*}-2} }  t 
\big\{ \alpha(\tau_{t}) 
\|u_{t}\|_{L^{2}}^{2} 
\big\}^{ \frac{2^{*}-(p+1)}{2^{*}-2} }
\|u_{t}\|_{L^{2^{*}}}^{\frac{2^{*}(p-1)}{2^{*}-2}}
.
\end{split} 
\end{equation}
Then,  we find from \eqref{20/9/22/11:09},  \eqref{20/9/21/19:55} and \eqref{20/9/22/10:52}  that
\begin{equation}\label{20/9/22/12}
\lim_{t\to 0} \sup_{u \in \widetilde{\mathcal{G}}_{t}} t\|u \|_{L^{p+1}}^{p+1}=0.
\end{equation}
Furthermore,  a standard argument (see \cite{Berestycki-Lions}) shows  that the following ``Pohozaev identity'' holds
 for all $u \in \widetilde{\mathcal{G}}_{t}$:
\begin{equation}\label{20/9/22/12:06}
\frac{1}{d}\alpha( \tau_{t} )
\|u \|_{L^{2}}^{2}
=
\frac{2^{*}-(p+1)}{2^{*}(p+1)}
t \|u \|_{L^{p+1}}^{p+1} 
.
\end{equation} 
Putting \eqref{20/9/22/12} and \eqref{20/9/22/12:06} together, we find that \eqref{20/9/22/13:51} is true.  
\end{proof}

%%%%%%%%%%%%%%%%%%%

\begin{lemma}\label{20/9/22/16:51}
Assume $d=3,4$ and $\frac{4}{d-2}-1<p<\frac{d+2}{d-2}$.   
 Let $\max\{2^{*}, \frac{2^{*}}{p-1}\} <q <\infty$,  $R>0$ and 
  let $T_{1}(q,R)$ be the number given in Proposition \ref{19/01/21/08:01}; 
 hence for any $0<t <T_{1}(q,R)$,  $I(t) \times Y_{q}(R,t)$ admits  a unique solution to \eqref{19/01/13/14:55}, say $(\tau_{t}, \eta_{t})$.    
 Then, all of the following hold:
\begin{enumerate}
\item 
\begin{equation}\label{20/9/22/16:57}
\lim_{t\to 0}\sup_{u \in \widetilde{\mathcal{G}}_{t}}
\| T_{u(0)} [u]  -W \|_{\dot{H}^{1}}
=0
.
\end{equation}

\item
If $0< t <T_{1}(q,R)$ is sufficiently small depending only on $d$, $p$, $q$ and $R$, then   
\begin{equation}\label{20/9/27/11:57}
\sup_{u \in \widetilde{\mathcal{G}}_{t}} | T_{u(0)}[u](x)| \lesssim (1+|x|)^{-(d-2)},
\end{equation}
 where the implicit constant depends only on $d$ and $p$; in particular,  for any $u \in \widetilde{\mathcal{G}}_{t}$ and $\frac{d}{d-2}<r \le \infty$, 
\begin{equation}\label{20/9/27/12:11}
\| T_{u(0)}[u] \|_{L^{r}}
\lesssim 1
, 
\end{equation}
where the implicit constant dpends only on $d$, $p$ and $r$. 

\item  
\begin{equation}\label{20/9/27/17:21}
\lim_{t\to 0} \sup_{u \in \widetilde{\mathcal{G}}_{t}} u(0)^{-2^{*}+2} \alpha(\tau_{t})=0. 
\end{equation}
\end{enumerate} 
\end{lemma}
\begin{remark}\label{20/9/27/16:11}
By \eqref{20/9/22/16:57} and \eqref{20/9/27/12:11}, we  see that 
\begin{equation}\label{20/9/27/12:7}
\lim_{t\to 0}\sup_{u \in \widetilde{\mathcal{G}}_{t}}
\| T_{u(0)} [u]  -W \|_{L^{r}}
=0
\quad 
\mbox{for all $\frac{d}{d-2}<r <\infty$}
.
\end{equation} 
\end{remark}
\begin{proof}[Proof of Lemma \ref{20/9/22/16:51}]
Let $0< t <T_{1}(q,R)$ be a number to be taken $t\to 0$,  
 and let $u \in \widetilde{\mathcal{G}}_{t}$.  
 Define $\lambda(t)$ and $\omega(t)$ as in \eqref{20/9/21/16:19}.  
  Furthermore, put $\Psi_{\omega(t)}:= T_{\lambda(t)^{-1}}[ u ]$. 
  Then,  Lemma \ref{20/9/21/10:29} shows that $\Psi_{\omega(t)}  \in \mathcal{G}_{\omega(t)}$.  
 Observe that 
\begin{equation}\label{20/9/23/8:51}
T_{\Psi_{\omega(t)}(0)}[ \Psi_{\omega(t)}]
=
T_{u(0)} [u] 
. 
\end{equation}
Furthermore, observe from \eqref{20/9/22/11:09} in Lemma \ref{20/9/21/19:3} that    
\begin{equation}\label{20/9/23/8:1}
\lim_{t\to 0}\omega(t) 
= 
\lim_{t\to 0}
\{\alpha(\tau_{t})^{\frac{-\{2^{*}-(p+1)\}}{2^{*}-2}} t
\}^{\frac{-(2^{*}-2)}{2^{*}-(p+1)}}
=
\infty
.
\end{equation}
Then, \eqref{20/9/22/16:57} follows from $\Psi_{\omega(t)}  \in \mathcal{G}_{\omega(t)}$, \eqref{20/9/23/8:51} and Lemma \ref{theorem:3.1}.  
 Furthermore,   \eqref{20/9/27/11:57} follows  from Lemma \ref{18/09/05/01:05}. 
 
 %%%%%%%%%

It remains to prove the last claim \eqref{20/9/27/17:21}.  Suppose for contradiction that the claim is false. Then, we can take  a sequence 
 $\{t_{n}\}$ with the following property:  $\lim_{n\to \infty}t_{n}=0$,  and   
 for any $n\ge 1$  there exists $u_{n} \in  \widetilde{\mathcal{G}}_{t_{n}}$ such that 
\begin{equation}\label{20/9/27/17:32}
\lim_{n\to \infty} u_{n}(0)^{-2^{*}+2} \alpha(\tau_{t_{n}})>0
.
\end{equation} 
Observe from a computation involving the scaling, and \eqref{20/9/22/13:51} in Lemma \ref{20/9/21/19:3} that  
\begin{equation}\label{20/9/27/17:42}
\lim_{n\to \infty} u_{n}(0)^{-2^{*}+2} 
\alpha(\tau_{t_{n}}) 
\|T_{u_{n}(0)}[u_{n}]\|_{L^{2}}^{2}
=
\lim_{n\to \infty} 
\alpha(\tau_{t_{n}}) 
\| u_{n} \|_{L^{2}}^{2}
=0.
\end{equation}
By \eqref{20/9/27/17:32} and \eqref{20/9/27/17:42},  we see that 
\begin{equation}\label{20/9/28/7:56}
\lim_{n\to \infty} 
\|T_{u_{n}(0)}[u_{n}]\|_{L^{2}}
=
0
.
\end{equation}
Furthermore, by \eqref{20/9/22/16:57},  \eqref{20/9/27/12:11} and  \eqref{20/9/28/7:56}, we see  that  
\begin{equation}\label{20/9/28/7:1}
\begin{split}
\|W\|_{L^{2^{*}}} 
&=
\lim_{n\to \infty}
\| T_{u_{n}(0)}[u_{n}] \|_{L^{2^{*}}}
\\[6pt]
&\le 
\lim_{n\to \infty}
\| T_{u_{n}(0)}[u_{n}] \|_{L^{2}}^{\frac{d-2}{d}} 
\|T_{u_{n}(0)}[u_{n}]\|_{L^{\infty}}^{\frac{2}{d}}
\lesssim 
\lim_{n\to \infty}
\| T_{u_{n}(0)}[u_{n}] \|_{L^{2}}^{\frac{d-2}{d}}
= 0
.
\end{split} 
\end{equation}
This is a contradiction. Thus, the claim \eqref{20/9/27/17:21} is true. 
\end{proof}

%%%%%%%%%%%%%%%%%%%

We will prove Proposition \ref{19/02/11/16:56} by showing that  
 a minimizer for \eqref{19/02/11/17:03} coincides with the unique solution to \eqref{19/01/13/14:55}  obtained in Proposition \ref{19/01/21/08:01}.  
The first step to accomplish this  plan is the following lemma (cf. Lemma 3.8 in \cite{CG}):   

%%%%%%%%%%%%%%%%%%%%

\begin{lemma}\label{20/9/23/9:8}
Assume $d=3,4$ and $\frac{4}{d-2}-1 <p<\frac{4}{d-2}+1$. 
 Let  $\max\{2^{*},\frac{2^{*}}{p+3-2^{*}} \} < q <\infty$. 
 Furthermore, for $R>0$,   let $T_{1}(q,R)$ denote the number given in Proposition \ref{19/01/21/08:01}. 
 Then, there exist $R_{*}>0$ and $0< T_{*}<T_{1}(q, R_{*})$, both depending only on $d$, $p$ and $q$,  with the following property:   
  For any $0<t < T_{*}$ and  $u_{t} \in \widetilde{\mathcal{G}}_{t}$,   
  there exists $\mu(t)=1+o_{t}(1)$ such that,   
  denoting  a unique solution to \eqref{19/01/13/14:55} in $I(t) \times Y_{q}(R_{*},t)$ by $(\tau_{t}, \eta_{t})$ 
(see Proposition \ref{19/01/21/08:01}), and  defining $\mu^{\dagger}(t)$,  $t^{\dagger}$,  $\alpha^{\dagger}(t)$ and $\eta_{t}^{\dagger}$ as  
\begin{align}
\label{20/9/23/10:33}
&
\mu^{\dagger}(t)
:=
\mu(t) u_{t}(0), 
\\[6pt]
\label{20/9/23/10:45}
&
t^{\dagger}
:=
\mu^{\dagger}(t)^{-\{2^{*}-(p+1)\}}t
,
\qquad 
\alpha^{\dagger}(t)
:=
\mu^{\dagger}(t)^{-(2^{*}-2)}\alpha (\tau_{t})
, 
\\[6pt]
\label{20/9/23/9:26}
&
\eta_{t}^{\dagger}
:=
T_{\mu^{\dagger}(t)}[u_{t}] - W
,
\end{align}
we have the following:    
\begin{enumerate}
\item 
\begin{align}
\label{20/9/23/9:12}
&
\langle  
(-\Delta
+ \alpha^{\dagger}(t) )^{-1}
F( \eta_{t}^{\dagger}; 
\alpha^{\dagger}(t),  t^{\dagger} )
, V\Lambda W  
\rangle 
=0
,
\\[6pt]
\label{20/10/4/11:37}
&
\lim_{t\to 0}
\sup_{u_{t} \in \widetilde{\mathcal{G}}_{t}}
\alpha^{\dagger}(t) =0,
\qquad 
\lim_{t\to 0} 
\sup_{u_{t} \in \widetilde{\mathcal{G}}_{t}}
t^{\dagger}=0,
\\[6pt]
\label{20/10/7/11:41}
&
\lim_{t\to 0}
\sup_{u_{t} \in \widetilde{\mathcal{G}}_{t}}
 \alpha^{\dagger}(t) ^{\frac{-\{ 2^{*}-(p+1)\}}{2^{*}-2}}
 t^{\dagger}  
= 0
, 
\\[6pt]
\label{20/9/25/11:37}
&
\lim_{t\to 0}
\| \eta_{t}^{\dagger} \|_{\dot{H}^{1}}
=0,
\quad 
\lim_{t\to 0}
\| \eta_{t}^{\dagger} \|_{L^{r}}
=
0
\quad 
\mbox{for all $\frac{d}{d-2}<r<\infty$}
.
\end{align}
\item 
\begin{equation}\label{20/10/5/9:1}
(\beta( \alpha^{\dagger}(t)),  \eta_{t}^{\dagger} ) 
 \in I( t^{\dagger}) \times Y_{q}(R_{*},  t^{\dagger}) 
 .
\end{equation}

\item  If  $0<t<T_{*}$, then $0< t^{\dagger} <T_{1}(q,R_{*})$.   Furthermore,  let  $0<t<T_{*}$, and 
 let $(\tau_{t^{\dagger}}, \eta_{t^{\dagger}})$ denote a unique solution to \eqref{19/01/13/14:55} 
 with $t=t^{\dagger}$ in $I( t^{\dagger}) \times Y_{q}(R_{*}, t^{\dagger})$. Then, 
\begin{equation}\label{20/11/23/18:15}
(\beta( \alpha^{\dagger}(t)), \eta^{\dagger}_{t} )
=
(\tau_{t^{\dagger}}, \eta_{t^{\dagger}}) 
.
\end{equation}
\end{enumerate} 
\end{lemma}

\begin{proof}[Proof of Lemma \ref{20/9/23/9:8}]
Since the proof is long, we divide it  into three steps:
\\ 
\noindent 
{\bf Step 1.}~Let $R>0$. Then,  we shall show that  there exists $0< T_{2}(q,R) <T_{1}(q,R)$, depending only on $d$, $p$, $q$ and $R$, with the 
 following property:  For any $0<t<T_{2}(q,R)$ and $u_{t} \in \widetilde{\mathcal{G}}_{t}$,   
 there exists $\mu(t)=1+o_{t}(1)$ such that  defining  $\mu^{\dagger}(t)$, $t^{\dagger}$, $\alpha^{\dagger}(t)$ and $\eta_{t}^{\dagger}$ in the same manner as \eqref{20/9/23/10:33} through \eqref{20/9/23/9:26}, we have  \eqref{20/9/23/9:12} through \eqref{20/9/25/11:37}. 
 
First, observe from  \eqref{20/9/27/17:21} in Lemma \ref{20/9/22/16:51} that 
\begin{equation}\label{20/11/22/11:35}
\lim_{t\to 0}  \sup_{u_{t} \in \widetilde{\mathcal{G}}_{t}} 
\sup_{\frac{1}{2}\le \mu \le \frac{3}{2}}
\{  \mu u_{t}(0)\}^{-(2^{*}-2)} \alpha(\tau_{t})
=0
.
\end{equation}
Furthermore,  observe from \eqref{20/9/22/11:09} in Lemma \ref{20/9/21/19:3} 
 that  
\begin{equation}\label{20/10/7/11:21}
\begin{split}
&
\lim_{t\to 0}
\sup_{u_{t} \in \widetilde{\mathcal{G}}_{t}} 
\sup_{\mu >0}
\big( 
\{ \mu u_{t}(0)\}^{-(2^{*}-2)}
 \alpha(\tau_{t})
\big)^{\frac{-\{ 2^{*}-(p+1)\}}{2^{*}-2}}
\{\mu u_{t}(0)\}^{-\{ 2^{*}-(p+1)\}} 
t
\\[6pt]
&=
\lim_{t\to 0}
\alpha(\tau_{t})^{\frac{-\{2^{*}-(p+1)\}}{2^{*}-2}}  t
=0
.
\end{split} 
\end{equation}  
  
Next,  let $0< T_{2}<T_{1}(q,R)$ be a sufficiently small number to be specified in the middle of the proof, dependently on $d$, $p$,
 $q$ and $R$ (we will see that  the dependence of $T_{2}$ on  $q$ and $R$ comes from  the restriction $T_{2}<T_{1}(q,R)$ only). 
  Furthermore, let $0< t< T_{2}$,  $u_{t}\in  \widetilde{\mathcal{G}}_{t}$ 
 and $\frac{1}{2}<\mu <\frac{3}{2}$.  
By Corollary \ref{20/9/22/17:47} and a computation involving the scaling, 
 we see that  
\begin{equation}\label{20/9/23/10:1}
\begin{split}
&
-\Delta T_{\mu u_{t}(0)}[u_{t}] 
+
\{\mu u_{t}(0)\}^{-(2^{*}-2)} \alpha(\tau_{t}) 
T_{\mu u_{t}(0)}[u_{t}] 
\\[6pt]
&\qquad -
\{\mu u_{t}(0)\}^{-\{ 2^{*}-(p+1)\}} t \,   
T_{\mu u_{t}(0)}[u_{t}]^{p}
-
T_{\mu u_{t}(0)}[u_{t}]^{\frac{d+2}{d-2}}
= 0
.
\end{split}
\end{equation}

Define $t(\mu)$,  $\alpha_{t}(\mu)$  and $\eta_{t}(\mu)$ as 
\begin{align}
\label{21/2/15/11:22}
&t(\mu)
:=
\{\mu u_{t}(0)\}^{-\{2^{*}-(p+1)\}}t,  
\qquad 
\alpha_{t}(\mu)
:=
\{\mu u_{t}(0)\}^{-(2^{*}-2)} \alpha(\tau_{t})
,
\\[6pt]
\label{21/2/15/11:28}
&
\eta_{t}(\mu):=  T_{\mu u_{t}(0)}[u_{t}] - W
.
\end{align}
By \eqref{20/11/22/11:35}, we may assume that 
 \begin{equation}\label{21/2/14/11:30}
\alpha_{t}(\mu)< 1, 
\qquad 
\delta(\alpha_{t}(\mu))< 1
.
\end{equation} 
Observe from  \eqref{20/9/23/10:1}  and \eqref{20/8/19/15:59}  that 
\begin{equation}\label{20/10/6/10:27}
(-\Delta + \alpha_{t}(\mu) +V ) 
\eta_{t}(\mu)
=
F(\eta_{t}(\mu); \alpha_{t}(\mu), t(\mu) )
.
\end{equation}
Furthermore, observe from $(-\Delta + V) \Lambda W =0$ (see \eqref{20/9/13/9:52}) that 
\begin{equation}\label{21/2/13/14:19}
V\Lambda W 
= 
\alpha_{t}(\mu) \Lambda W 
- 
(-\Delta + \alpha_{t}(\mu) )\Lambda W
.
\end{equation}
By \eqref{20/10/6/10:27} and \eqref{21/2/13/14:19}, 
 we see that 
\begin{equation} \label{20/9/24/21:29} 
\begin{split}
&
\langle  
(-\Delta
+ 
\alpha_{t}(\mu) )^{-1}
F(\eta_{t}(\mu); 
\alpha_{t}(\mu), t(\mu))
, V\Lambda W \rangle 
\\[6pt]
&=
\langle 
(-\Delta + \alpha_{t}(\mu) )^{-1} 
(-\Delta + \alpha_{t}(\mu) +V ) \eta_{t}(\mu) 
,
V\Lambda W
\rangle 
\\[6pt]
&=
\langle 
\eta_{t}(\mu)
,
V\Lambda W
\rangle 
+
\langle 
(-\Delta + \alpha_{t}(\mu) )^{-1} 
V \eta_{t}(\mu)
, V\Lambda W
\rangle 
\\[6pt]
&=
\langle 
(-\Delta + \alpha_{t}(\mu) )^{-1}V \eta_{t}(\mu), 
\alpha_{t}(\mu)  
\Lambda W
\rangle 
=
\langle 
\eta_{t}(\mu) 
,
\alpha_{t}(\mu) V (-\Delta + \alpha_{t}(\mu) )^{-1}
\Lambda W
\rangle 
.
\end{split} 
\end{equation}
Note that  $\eta_{t}(\mu)$ can be written as 
\begin{equation}\label{20/9/27/16:41}
\eta_{t}(\mu) = T_{\mu } [ T_{u_{t} (0)}[u_{t}] - W] + T_{\mu}[W] -W. 
\end{equation}
Plugging \eqref{20/9/27/16:41} into \eqref{20/9/24/21:29},  
 we see that for any $\mu>0$, 
\begin{equation} \label{20/9/24/21:50}
\begin{split}
&
\langle  
(-\Delta
+ 
\alpha_{t}(\mu) )^{-1}
F(\eta_{t}(\mu); 
\alpha_{t}(\mu)
, t(\mu))
, V\Lambda W  
\rangle 
\\[6pt]
&=
\langle
T_{\mu} [T_{u_{t}(0)}[u_{t}]-W]  
,~\alpha_{t}(\mu) V  (-\Delta + \alpha_{t}(\mu) )^{-1} \Lambda W
\rangle 
\\[6pt]
&\quad +
\langle
T_{\mu}[W] - W +\frac{2}{d-2}(\mu-1)\Lambda W 
,~\alpha_{t}(\mu) V  (-\Delta +  \alpha_{t}(\mu) )^{-1} \Lambda W
\rangle 
\\[6pt]
&\quad -
\frac{2}{d-2}(\mu-1)
\langle
\Lambda W 
,~\alpha_{t}(\mu) V  (-\Delta +  \alpha_{t}(\mu) )^{-1} \Lambda W
\rangle 
.
\end{split}
\end{equation}
Furthermore,  by \eqref{20/9/26/8:41}, we may rewrite \eqref{20/9/24/21:50} as 
\begin{equation}\label{20/9/26/9:3}
\begin{split}
&
\langle  
(-\Delta
+ 
\alpha_{t}(\mu) )^{-1}
F(\eta_{t}(\mu); 
\alpha_{t}(\mu)
, t(\mu)), V\Lambda W  
\rangle 
\\[6pt]
&=
\langle
T_{\mu} [T_{u_{t}(0)}[u]-W]  
,~
\alpha_{t}(\mu) V  (-\Delta + \alpha_{t}(\mu) )^{-1} \Lambda W
\rangle 
\\[6pt]
&\quad 
+\frac{2}{d-2}
\langle
\int_{1}^{\mu} 
\int_{1}^{\nu}
\lambda^{-2} 
T_{\lambda}\big[ 
2 \Lambda W
+
\frac{2}{d-2}
 x\cdot \nabla \Lambda W
\big]
d \lambda d\nu 
,~\alpha_{t}(\mu) V  (-\Delta + \alpha_{t}(\mu) )^{-1} \Lambda W
\rangle 
\\[6pt]
&\quad -
\frac{2}{d-2}(\mu-1)
\langle
 \Lambda W
,~
\alpha_{t}(\mu) V  (-\Delta + \alpha_{t}(\mu) )^{-1} \Lambda W
\rangle 
.
\end{split} 
\end{equation}
Observe from the duality, Lemma \ref{18/12/02/11:25},
 \eqref{21/2/14/11:30} and $\beta(s):=\delta(s)^{-1}s$ that 
\begin{equation} \label{20/9/26/11:43}
\begin{split}
&
\| \alpha_{t}(\mu) 
V  (-\Delta +  \alpha_{t}(\mu) )^{-1} \Lambda W
\|_{L^{\frac{2d}{d+2}}} 
= 
\sup_{{g\in L^{2^{*}}}\atop {\|g\|_{L^{2^{*}}}=1}}
| 
\langle 
\alpha_{t}(\mu) V(-\Delta + \alpha_{t}(\mu) )^{-1} \Lambda W
,~
g 
\rangle  
|
\\[6pt]
&=
\alpha_{t}(\mu)
\sup_{{g\in L^{2^{*}}}\atop {\|g\|_{L^{2^{*}}}=1}}
|
\langle 
(-\Delta + \alpha_{t}(\mu) )^{-1}  V g 
,~
\Lambda W
\rangle 
| 
\\[6pt]
&\lesssim 
\alpha_{t}(\mu)
\sup_{{g\in L^{2^{*}}}\atop {\|g\|_{L^{2^{*}}} =1}}
\big\{ 
\delta( \alpha_{t}(\mu))^{-1} 
\|V\|_{L^{\frac{2d}{d+2}}} 
\|g\|_{L^{2^{*}}} + 
\|g\|_{L^{2^{*}}}
\big\} 
\lesssim 
\beta( \alpha_{t}(\mu)) 
,
\end{split} 
\end{equation}
where the implicit constants depend only on $d$. Furthermore, 
 by H\"older's inequality, a computation involving the scaling  and \eqref{20/9/26/11:43}, the first term on the right-hand side of \eqref{20/9/26/9:3} is estimated as follows: 
\begin{equation}\label{18/06/29/13:57}
\begin{split}
&
|
\langle
T_{\mu} [T_{u_{t}(0)}[u_{t}]-W]  
,~
\alpha_{t}(\mu) V  (-\Delta + \alpha_{t}(\mu) )^{-1} \Lambda W
\rangle 
|
\\[6pt]
&\lesssim 
\|T_{u_{t}(0)}[u_{t}]-W  \|_{L^{2^{*}}}
\beta(\alpha_{t}(\mu)) 
,
\end{split}
\end{equation}
where the implicit constant depends only on $d$. Similarly,
 the second term on the right-hand side of \eqref{20/9/26/9:3} is estimated as follows:    
\begin{equation} \label{20/9/26/13:5}
\begin{split}
&
\Big|
\langle
\int_{1}^{\mu} 
\int_{1}^{\nu}
\lambda^{-2} 
T_{\lambda}\big[ 
2\Lambda W
+
\frac{2}{d-2}
 x\cdot \nabla \Lambda W
\big]
\,d \lambda d\nu 
,~
\alpha_{t}(\mu) V  (-\Delta + \alpha_{t}(\mu) )^{-1} \Lambda W
\rangle 
\Big|
\\[6pt]
&\lesssim  
\beta( \alpha_{t}(\mu))   
\| \int_{1}^{\mu} 
\int_{1}^{\nu}
\lambda^{-2} 
T_{\lambda}\big[ 2 \Lambda W
+
\frac{2}{d-2}
 x\cdot  \nabla \Lambda W
\big]
\,d \lambda d\nu 
 \|_{L^{2^{*}}}
\\[6pt]
&\le
\beta(\alpha_{t}(\mu))   
\int_{1}^{\mu} 
\int_{1}^{\nu}
\lambda^{-2}
\| 2 \Lambda W 
+
\frac{2}{d-2}
x\cdot \nabla \Lambda W
\|_{L^{2^{*}}}
\, d\lambda d\nu 
\\[6pt]
&\lesssim 
\beta(\alpha_{t}(\mu))  
\int_{1}^{\mu} 
\int_{1}^{\nu}
\lambda^{-2}
\, d\lambda d\nu
\lesssim 
\beta(\alpha_{t}(\mu))  
|\mu- 1 |^{2} 
,
\end{split}
\end{equation}
where the implicit constants depend only on $d$. 
 Consider the last  term on the right-hand side of \eqref{20/9/26/9:3}. 
Note that 
\begin{equation}\label{21/2/13/16:25}
\langle
 \Lambda W
,~
\alpha_{t}(\mu) V  (-\Delta +\alpha_{t}(\mu) )^{-1} \Lambda W
\rangle 
=
\alpha_{t}(\mu) 
\langle
 (-\Delta +\alpha_{t}(\mu) )^{-1}  V\Lambda W
,~
\Lambda W
\rangle 
.
\end{equation}
Furthermore, observe from \eqref{20/12/22/14:18} that $\Re{\mathcal{F}[V\Lambda W](0)} = - \frac{d+2}{d-2}\langle W^{\frac{4}{d-2}}, \Lambda W \rangle >0$. 
 Then, applying Lemma \ref{18/12/02/11:25} to \eqref{21/2/13/16:25} as $g=\Lambda W$ and $q=2^{*}$, and using $\beta(s):=\delta(s)^{-1}s$, 
 we see that there exists  $\mathscr{C}_{0}>0$ depending only on $d$ such that 
\begin{equation}\label{20/9/26/12:1}
\begin{split}
&
\Big|
-\frac{2}{d-2}(\mu-1)
\langle
 \Lambda W
,~
\alpha_{t}(\mu) V  (-\Delta +\alpha_{t}(\mu) )^{-1} \Lambda W
\rangle 
-
\mathscr{C}_{0} 
\beta(\alpha_{t}(\mu))
(\mu-1) 
\Big|
\\[6pt]
&\lesssim 
\delta(\alpha_{t}(\mu))
\beta(\alpha_{t}(\mu)) 
|\mu-1 |
,
\end{split} 
\end{equation}
where the implicit constant depends only on $d$.   

Plugging \eqref{18/06/29/13:57}, \eqref{20/9/26/13:5} and \eqref{20/9/26/12:1} into \eqref{20/9/26/9:3},  
we see that 
\begin{equation}\label{20/9/26/13:33}
\begin{split}
&
\big|
\langle  
(-\Delta
+ 
\alpha_{t}(\mu) )^{-1}
F(\eta_{t}(\mu); 
\alpha_{t}(\mu)
,  t(\mu) )
, V\Lambda W  
\rangle 
-\mathscr{C}_{0}  
\beta( \alpha_{t}(\mu))  
\{ \mu-1\} 
\big|
\\[6pt]
&\le C_{1}   
\beta(\alpha_{t}(\mu))
\big\{  
\sup_{u\in \widetilde{\mathcal{G}}_{t}}\|T_{u(0)}[u]-W  \|_{L^{2^{*}}}
+
|\mu - 1 |^{2} 
+
\delta(\alpha_{t}(\mu)) 
| \mu-1 |
\big\}
,
\end{split} 
\end{equation}
where $C_{1}>0$ is some constant  depending only on $d$.   
 Furthermore, by \eqref{20/9/26/13:33},  \eqref{20/9/22/16:57} in Lemma \ref{20/9/22/16:51} and $\lim_{t\to 0} \delta(\alpha_{t}(\mu))=0$ (see \eqref{20/11/22/11:35}), we see that  there  exists $T_{2}>0$ depending only on $d$ and $p$ such that for any $0<t<T_{2}$, $u_{t}\in  \widetilde{\mathcal{G}}_{t}$ and $\frac{1}{2}< \mu <\frac{3}{2}$,   
\begin{equation}\label{21/2/16/9:24}
\begin{split} 
&
\big|
\langle  
(-\Delta
+ 
\alpha_{t}(\mu) )^{-1}
F(\eta_{t}(\mu); 
\alpha_{t}(\mu)
,  t(\mu) )
, V\Lambda W  
\rangle 
-
\mathscr{C}_{0}  
\beta( \alpha_{t}(\mu))  
\{ \mu-1\} 
\big|
\\[6pt]
&\le   
\beta(\alpha_{t}(\mu))
\Big\{ \frac{\mathscr{C}_{0}}{4(1+C_{1})}  
+
C_{1} |\mu - 1 |^{2} 
+
 \frac{\mathscr{C}_{0}}{4}
| \mu-1 |
\Big\}
,
\end{split} 
\end{equation} 
which together with the intermediate value theorem implies that  
  for any $0< t< T_{2}$ and $u_{t}\in \widetilde{\mathcal{G}}_{t}$,  
 there exists $\mu(t)>0$ with  $|\mu(t) -1| \le \frac{\min\{ \mathscr{C}_{0},1 \}}{4(1+C_{1})}$ 
  such that 
\begin{equation}\label{21/2/16/9:24}
\langle  
(-\Delta
+ 
\alpha_{t}(\mu(t)) )^{-1}
F(\eta_{t}(\mu(t)); 
\alpha_{t}(\mu(t))
,  t(\mu(t)) )
, V\Lambda W  
\rangle 
=0
.
\end{equation}
Furthermore, when $\mu=\mu(t)$,   \eqref{20/9/26/13:33}  together with  \eqref{21/2/16/9:24}  implies that 
\begin{equation}\label{21/2/17/10:37}
\lim_{t\to 0}|\mu(t) -1 |
=0.  
\end{equation}
Thus, we have proved  that  for any $R>0$, there exists $0< T_{2}(q,R) <T_{1}(q,R)$ such that for any 
  $0<t<T_{2}(q,R)$ and $u_{t}\in \widetilde{\mathcal{G}}_{t}$, there exists  $\mu(t)=1+o_{t}(1)$ for which  \eqref{20/9/23/9:12} holds. 
  Define $\mu^{\dagger}(t)$, $t^{\dagger}$, $\alpha^{\dagger}(t)$ and $\eta_{t}^{\dagger}$ as in \eqref{20/9/23/10:33} through \eqref{20/9/23/9:26}.  
 Then, the claims \eqref{20/10/4/11:37} and \eqref{20/10/7/11:41} follows from  \eqref{20/11/22/11:35} and \eqref{20/10/7/11:21};  
 We remark that $\lim_{t\to 0}\sup_{u_{t} \in \widetilde{\mathcal{G}}_{t}}t^{\dagger}=0$ follows from  $\lim_{t\to 0}\sup_{u_{t} \in \widetilde{\mathcal{G}}_{t}}\alpha^{\dagger}(t)=0$ and \eqref{20/10/7/11:41}.

It remains to prove \eqref{20/9/25/11:37}.  
 By the fundamental theorem of calculus (see \eqref{20/9/25/11:45}), a computation involving the scaling, \eqref{20/9/22/16:57} in Lemma \ref{20/9/22/16:51}, and $\mu(t)=1+o_{t}(1)$, 
 we see that   
\begin{equation}\label{20/9/28/8:52}
\begin{split}
\|\eta_{t}^{\dagger}\|_{\dot{H}^{1}}
&\le 
\| T_{\mu(t)}[ T_{u_{t}(0)}[u_{t}] -W] \|_{\dot{H}^{1}}
+
\| T_{\mu(t)}[ W] -W \|_{\dot{H}^{1}}
\\[6pt]
&\lesssim 
\| T_{u_{t}(0)}[u_{t}] -W  \|_{\dot{H}^{1}}
+
\int_{1}^{\mu(t)} 
\lambda^{-1}  \| \Lambda W \|_{\dot{H}^{1}} 
\,d \lambda
\to 0
\quad 
\mbox{as $t\to 0$}
.
\end{split}
\end{equation}
Similarly,  using  \eqref{20/9/27/12:7} instead of \eqref{20/9/22/16:57}, 
 we see  that   for any $\frac{d}{d-2}<r < \infty$, 
\begin{equation}\label{20/10/1/10:30}
\begin{split}
\|\eta_{t}^{\dagger}\|_{L^{r}}
&\lesssim
\mu(t)^{\frac{2^{*}-r}{r}}  
\| T_{u_{t}(0)}[u_{t}] -W  \|_{L^{r}}
+
\int_{1}^{\mu(t)} 
\lambda^{-1 + \frac{2^{*}-r}{r}}  \| \Lambda W \|_{L^{r}} 
\,d \lambda
\\[6pt]
&\to 0
\quad 
\mbox{as $t\to 0$}
.
\end{split}
\end{equation}
Thus, we have proved  \eqref{20/9/25/11:37}.

\noindent 
{\bf Step 2.}~Let $R>0$, and let  $T_{2}(q,R)$ denote the number given in the previous step. 
  Furthermore, let   $0< t<T_{2}(q,R)$,  $u_{t} \in \widetilde{\mathcal{G}}_{t}$, and  let  $\mu(t)=1+o_{t}(1)$ be a number such that  \eqref{20/9/23/9:12} through \eqref{20/9/25/11:37} hold.  Define $\mu^{\dagger}(t)$, $t^{\dagger}$, $\alpha^{\dagger}(t)$ and $\eta_{t}^{\dagger}$ as in \eqref{20/9/23/10:33} through \eqref{20/9/23/9:26}.  Then,  our aim in this step is to show  that, taking  $T_{2}(q,R)$ even smaller dependently on  $d$, $p$ and $q$,  we have     
\begin{equation}\label{20/9/27/16:35}
\|\eta_{t}^{\dagger}\|_{L^{q}}
\lesssim 
\alpha^{\dagger}(t)^{\Theta_{q}} 
+
t^{\dagger} 
,
\end{equation}
where the implicit constant depends only on $d$, $p$ and $q$.   

 Note that $\eta_{t}^{\dagger}$ satisfies 
\begin{equation}\label{20/11/23/17:21}
\eta_{t}^{\dagger}
=
\{ 1+(-\Delta + \alpha^{\dagger}(t))^{-1} V  \}^{-1}
(-\Delta + \alpha^{\dagger}(t))^{-1} 
F(\eta_{t}^{\dagger}; \alpha^{\dagger}(t), t^{\dagger})
.
\end{equation}
By \eqref{19/01/13/15:17},   \eqref{18/11/24/13:56} in Lemma \ref{18/11/05/10:29} with $q_{1}=\frac{d}{d-2}$,  
 Lemma \ref{18/11/23/17:17}, $\frac{pdq}{d+2q}>\frac{d}{d-2}$ and \eqref{20/11/8/11:39}, 
 we see that  
\begin{equation}\label{20/9/28/9:38}
\begin{split}
&\| (-\Delta+ \alpha^{\dagger}(t))^{-1} 
F(\eta_{t}^{\dagger}; \alpha^{\dagger}(t), t^{\dagger})
\|_{L^{q}}
\\[6pt]
&\le 
\alpha^{\dagger}(t) \| (-\Delta+\alpha^{\dagger}(t))^{-1} W \|_{L^{q}}
+
t^{\dagger} \| (-\Delta+\alpha^{\dagger}(t))^{-1} W^{p} \|_{L^{q} }
\\[6pt]
&\quad +
\| (-\Delta+ \alpha^{\dagger}(t))^{-1}  N(\eta_{t}^{\dagger}; t^{\dagger}) \|_{L^{q} }
\\[6pt]
&\lesssim 
\alpha^{\dagger}(t)^{\Theta_{q}} 
+
t^{\dagger}   
\\[6pt]
&\quad +
\| (-\Delta+ \alpha^{\dagger}(t))^{-1}  D(\eta_{t}^{\dagger}, 0)\|_{L^{q} }
+
t^{\dagger}   
\| (-\Delta+\alpha^{\dagger}(t))^{-1}  E(\eta_{t}^{\dagger}, 0)\|_{L^{q} }
,
\end{split}
\end{equation}
where the implicit constant depends only on $d$, $p$ and $q$. 
 Consider the third term on the right-hand side of \eqref{20/9/28/9:38}.  
 By \eqref{20/9/15/14:54}, \eqref{19/02/19/11:43},  
Lemma \ref{18/11/23/17:17},  H\"older's inequality,  
\eqref{20/9/25/11:37}, $\frac{(6-d)dq}{(d-2)(2q-d)}>\frac{d}{d-2}$ and $q>\frac{d}{d-2}$, 
 we see  that  
\begin{equation}\label{20/9/30/11:42}
\begin{split}
&\| (-\Delta + \alpha^{\dagger}(t))^{-1} D(\eta_{t}^{\dagger},0) \|_{L^{q}}
\lesssim
\| 
(W+ |\eta_{t}^{\dagger}|)^{\frac{6-d}{d-2}}
|\eta_{t}^{\dagger}|^{2}
\|_{ L^{\frac{dq}{d+2q}}}
\\[6pt]
&\lesssim 
\| (W + |\eta_{t}^{\dagger}|)^{\frac{6-d}{d-2}} \|_{ L^{\frac{dq}{2q-d}}} 
\| |\eta_{t}^{\dagger}|^{2} \|_{L^{\frac{q}{2}}}
\lesssim 
\|\eta_{t}^{\dagger} \|_{L^{q}}^{2}
=
o_{t}(1) 
\|\eta_{t}^{\dagger} \|_{L^{q}}
,
\end{split}
\end{equation}
where the implicit constants depend only on $d$ and $q$.
 Consider  the last term on the right-hand side of \eqref{20/9/28/9:38}. 
 Introduce an exponent $q_{1}$ as  
\begin{equation}\label{20/10/2/9:29} 
q_{1}:= \frac{2^{*}q}{(p-1)q +2^{*} } =\frac{2^{*}}{p-1+ \frac{2^{*}}{q}}
.
\end{equation} 
Observe from $q>\max\{2^{*}, \frac{2^{*}}{p-1}\}$ that  
\begin{equation}\label{20/10/2/9:13}
1<q_{1} < \frac{q}{2},
\qquad 
\frac{d}{2}\Big( \frac{1}{q_{1}} -\frac{1}{q} \Big) -1 
= 
\frac{-\{ 2^{*}-(p+1)\}}{2^{*}-2},
\qquad 
\frac{d}{d-2}  < \frac{(p-1) q_{1}q}{q-q_{1}}
.
\end{equation} 
Furthermore,  observe from the last  condition in \eqref{20/10/2/9:13} and \eqref{20/9/25/11:37}
 that 
\begin{equation}\label{20/10/2/9:38}
\|W\|_{L^{\frac{(p-1) q_{1}q}{q-q_{1}}}}\lesssim 1, 
\qquad 
\|\eta_{t}^{\dagger}\|_{L^{\frac{(p-1) q_{1}q}{q-q_{1}}}}=o_{t}(1)
,
\end{equation}
where the implicit constant depends only on $d$, $p$ and $q$.
 Then, by  \eqref{20/9/15/14:54},  \eqref{19/02/19/17:02}, Lemma \ref{18/11/05/10:29},  H\"older's inequality, the second condition in \eqref{20/10/2/9:13}, 
 \eqref{20/10/2/9:38} and \eqref{20/10/7/11:41},  
 we see that 
\begin{equation}\label{20/10/3/18:2}
\begin{split}
&
t^{\dagger}  
\| (-\Delta+\alpha^{\dagger}(t))^{-1} E(\eta_{t}^{\dagger}, 0)  \|_{L^{q}}
\\[6pt]
&\lesssim
t^{\dagger} 
\alpha^{\dagger}(t)^{\frac{d}{2}(\frac{1}{q_{1}}-\frac{1}{q})-1}
\|   (W+|\eta_{t}^{\dagger}|)^{p-1} |\eta_{t}^{\dagger}|  \|_{L^{q_{1}}}
\\[6pt]
&\le 
t^{\dagger} 
\alpha^{\dagger}(t)^{\frac{-\{ 2^{*}-(p+1)\}}{2^{*}-2}}
\| (W + |\eta_{t}^{\dagger}|)^{p-1} \|_{L^{\frac{q_{1}q}{q-q_{1}}}}
\|\eta_{t}^{\dagger}\|_{L^{q}}
\\[6pt]
&\lesssim 
t^{\dagger} 
\alpha^{\dagger}(t)^{\frac{-\{ 2^{*}-(p+1)\}}{2^{*}-2}}
\|\eta_{t}^{\dagger}\|_{L^{q}}
=
o_{t}(1)
\|\eta_{t}^{\dagger}\|_{L^{q}}
,
\end{split}
\end{equation}
where the implicit constants depend only on $d$, $p$ and $q$.

Plugging  \eqref{20/9/30/11:42} and \eqref{20/10/3/18:2} into \eqref{20/9/28/9:38},  
 we see  that 
\begin{equation}\label{20/11/23/16:58}
\| (-\Delta+ \alpha^{\dagger}(t))^{-1} 
F(\eta_{t}^{\dagger}; \alpha^{\dagger}(t), t^{\dagger})
\|_{L^{q}}
\lesssim 
\alpha^{\dagger}(t)^{\Theta_{q}}
+
t^{\dagger} 
+
o_{t}(1) \|\eta_{t}^{\dagger}\|_{L^{q}}
,
\end{equation}
where  the implicit constant depends only on $d$, $p$ and $q$. 
 Furthermore, by \eqref{20/11/23/17:21}, \eqref{20/9/23/9:12},  Proposition \ref{18/11/17/07:17} and \eqref{20/11/23/16:58}, we see that     
\begin{equation}\label{20/9/28/9:35}
\begin{split}
\|\eta_{t}^{\dagger}\|_{L^{q}}
&\lesssim  
\| (-\Delta+ \alpha^{\dagger}(t))^{-1} 
F(\eta_{t}^{\dagger}; \alpha^{\dagger}(t), t^{\dagger})
\|_{L^{q}}
\lesssim 
\alpha^{\dagger}(t)^{\Theta_{q}}
+
t^{\dagger} 
+
o_{t}(1) \|\eta_{t}^{\dagger}\|_{L^{q}}
,
\end{split} 
\end{equation}
where the implicit constants depend only on $d$, $p$ and $q$.  Since the last term on the right-hand side of \eqref{20/9/28/9:35} can be absorbed into the left-hand side, \eqref{20/9/28/9:35} implies \eqref{20/9/27/16:35}. 

%%%%%%%%%%%%%%%%%%%%%%%

\noindent 
{\bf Step 3.}~ We shall finish the proof of the lemma.  
 Let $R>0$, and let $T_{2}(q,R)$ denote the same number as in  the previous step.    Then, for any $0< t <T_{2}(q,R)$ and $u_{t}\in \widetilde{\mathcal{G}}_{t}$,  we can take 
  $\mu(t)=1+o_{t}(1)$ for which \eqref{20/9/23/9:12} through \eqref{20/9/25/11:37} and \eqref{20/9/27/16:35} hold.  

Note that since $\alpha$ is the inverse function of $\beta$,  we have  
\begin{equation}\label{20/11/23/18:5}
\alpha^{\dagger}(t)= \alpha (\beta(\alpha^{\dagger}(t)))
.
\end{equation}
Observe from  \eqref{20/9/23/9:12}  that 
\begin{equation}\label{20/10/4/11:21}
\begin{split}
0&=\langle  
(-\Delta
+ 
\alpha^{\dagger}(t) )^{-1}
F(\eta_{t}^{\dagger}; 
\alpha^{\dagger}(t)
,  
t^{\dagger} )
, V\Lambda W  
\rangle 
\\[6pt]
&=
-
\alpha^{\dagger}(t)
\langle  
(-\Delta+\alpha^{\dagger}(t) )^{-1}
W
,  
V\Lambda W  
\rangle 
+ 
t^{\dagger} 
\langle  
(-\Delta+\alpha^{\dagger}(t) )^{-1}W^{p}
, V\Lambda W  
\rangle 
\\[6pt]
&\quad +
\langle  
(-\Delta + \alpha^{\dagger}(t) )^{-1}
D(\eta_{t}^{\dagger}, 0) 
,
V\Lambda W  
\rangle 
+
t^{\dagger}
\langle  
(-\Delta + \alpha^{\dagger}(t) )^{-1}
E(\eta_{t}^{\dagger}, 0) 
,
V\Lambda W  
\rangle 
.
\end{split} 
\end{equation}
Recall that $\beta(s)= \delta(s)^{-1}s$. 
Then,  by Lemma \ref{20/12/22/13:55} and Lemma \ref{18/12/19/01:00}, 
 we see that 
\begin{align}
\label{20/10/4/11:41}
\alpha^{\dagger}(t)
\langle  
(-\Delta+\alpha^{\dagger}(t) )^{-1}
W
,  
V\Lambda W  
\rangle 
&=
A_{1} \beta(\alpha^{\dagger}(t))
+
O(\alpha^{\dagger}(t) )
,
\\[6pt]
\label{20/10/4/11:42}
t^{\dagger} 
\langle  
(-\Delta+ \alpha^{\dagger}(t) )^{-1}W^{p}
, V\Lambda W  
\rangle 
&= 
K_{p} t^{\dagger} 
+
o_{t}(1)t^{\dagger}
.
\end{align}
By \eqref{20/9/30/11:42} and \eqref{20/9/27/16:35}, 
 we see that    
\begin{equation}\label{20/10/4/17:17}
\begin{split}
&
|
\langle  
(-\Delta +  \alpha^{\dagger}(t) )^{-1}
D(\eta_{t}^{\dagger}, 0) 
,
V\Lambda W  
\rangle 
|
\\[6pt]
&\le  
\|(-\Delta + \alpha^{\dagger}(t) )^{-1}
D(\eta_{t}^{\dagger}, 0) 
\|_{L^{q}}
\|V\Lambda W   \|_{L^{\frac{q}{q-1}}}
\lesssim 
\|\eta_{t}^{\dagger} \|_{L^{q}}^{2}
\lesssim 
\alpha^{\dagger}(t)^{2 \Theta_{q}} + o_{t}(1) t^{\dagger}
, 
\end{split} 
\end{equation}
where the implicit constants depend only on $d$, $p$ and $q$. 
 Observe from $q>\max\{2^{*}, \frac{2^{*}}{p+3-2^{*}}\}$ and $p> \frac{4}{d-2}-1$ that \begin{equation}\label{21/2/16/14:50}
\Theta_{q}-\frac{2^{*}-(p+1)}{2^{*}-2}>0,
\qquad 
2\Theta_{q} > \frac{d-2}{2}
.
\end{equation} 
By \eqref{20/10/3/18:2}, \eqref{20/9/27/16:35}, \eqref{21/2/16/14:50} and 
 \eqref{20/10/7/11:41}, we see 
\begin{equation}\label{20/10/4/17:18}
\begin{split} 
&
t^{\dagger} 
| 
\langle  
(-\Delta +  \alpha^{\dagger}(t) )^{-1}
E(\eta_{t}^{\dagger}, 0) 
,
V\Lambda W  
\rangle 
|
\lesssim  
t^{\dagger} \|(-\Delta + \alpha^{\dagger}(t) )^{-1}
E(\eta_{t}^{\dagger}, 0) 
 \|_{L^{q}}
\\[6pt]
&= 
t^{\dagger} 
\alpha^{\dagger}(t)^{\frac{-\{ 2^{*}-(p+1)\}}{2^{*}-2}}
\|\eta_{t}^{\dagger} \|_{ L^{q} } 
\lesssim 
t^{\dagger} 
\alpha^{\dagger}(t)^{\frac{-\{ 2^{*}-(p+1)\}}{2^{*}-2}}
\big\{ 
\alpha^{\dagger}(t)^{\Theta_{q}}+ t^{\dagger} 
\big\}
= 
o_{t}(1) t^{\dagger}
,
\end{split} 
\end{equation}
where the implicit constants depend only on $d$, $p$ and $q$. 

Plugging  \eqref{20/10/4/11:41}, \eqref{20/10/4/11:42}, \eqref{20/10/4/17:17} and \eqref{20/10/4/17:18} into \eqref{20/10/4/11:21}, and using \eqref{20/11/23/18:5}, \eqref{21/2/16/14:50} and \eqref{20/10/15/16:46},  
 we see that 
\begin{equation}\label{20/10/4/17:55}
\begin{split}
\Big|
\beta( \alpha^{\dagger}(t))
-
\frac{K_{p}}{A_{1}}
t^{\dagger} 
\Big|
&\le  
o_{t}(1) t^{\dagger} + O(\alpha^{\dagger}(t)) 
+
\alpha^{\dagger}(t)^{2\Theta_{q}}
\le 
o_{t}(1) t^{\dagger} 
+
o_{t}(1) \beta( \alpha^{\dagger}(t)).
\end{split} 
\end{equation}
This implies that $\beta( \alpha^{\dagger}(t)) \in I(t^{\dagger})$. 
 Furthermore, by \eqref{20/9/27/16:35}, \eqref{20/11/23/18:5}, 
 \eqref{21/2/6/15:27} and $\beta( \alpha^{\dagger}(t)) \in I(t^{\dagger})$ (hence $\beta(\alpha^{\dagger}(t)) \sim t^{\dagger}$), we see that  
\begin{equation}\label{20/11/23/15:57}
\|\eta_{t}^{\dagger}\|_{L^{q}}
\lesssim 
\alpha^{\dagger}(t)^{\Theta_{q}} 
+
t^{\dagger} 
\lesssim
\alpha( \beta(\alpha^{\dagger}(t)) )^{\Theta_{q}}
+
\alpha(t^{\dagger})^{\Theta_{q}}
\lesssim 
\alpha(t^{\dagger})^{\Theta_{q}}
,
\end{equation}
where the implicit constants depend only on $d$, $p$ and $q$, so that 
 there exists $R_{*}>0$ depending only on $d$, $p$ and $q$ such that 
\begin{equation}\label{21/2/16/16:16}
\eta_{t}^{\dagger}\in Y_{q}(R_{*}, t^{\dagger}) 
. 
\end{equation}
Since $R_{*}$ depends only on $d$, $p$ and $q$, we may take $R=R_{*}$ from 
 the beginning of the proof;  Put $T_{*}:=T_{2}(q,R_{*})$.  
 Thus, we have proved the claims \eqref{20/9/23/9:12} through \eqref{20/10/5/9:1}. 

 It remains to prove \eqref{20/11/23/18:15}.  
 By \eqref{20/11/23/18:5},  \eqref{20/11/16/10:42} and \eqref{19/01/14/13:30}, 
  we may rewrite  \eqref{20/9/23/9:12} as  
 \begin{equation}\label{20/11/23/17:52}
\beta(\alpha^{\dagger}(t))
=
\mathfrak{s}( t^{\dagger}; \beta(\alpha^{\dagger}(t) ),  \eta_{t}^{\dagger} )
. 
\end{equation}
Furthermore, by \eqref{20/11/23/18:5}, 
 we may write  \eqref{20/11/23/17:21} as 
\begin{equation}
\label{20/11/23/17:25}
\eta_{t}^{\dagger}
=
\mathfrak{g}(t^{\dagger}; \beta( \alpha^{\dagger}(t)) ,  \eta_{t}^{\dagger})
.
\end{equation}
By \eqref{20/10/4/11:37}, choosing $T_{*}$ even smaller dependently only on $d$, $p$ and $q$,  
  we may assume that  $t^{\dagger} <T_{1}(q,R_{*})$ for all $0<t<T_{*}$. 
 Then,  the claim \eqref{20/11/23/18:15} follows from the uniqueness of solutions to \eqref{19/01/13/14:55} in 
 $I(t^{\dagger})\times Y_{q}(R_{*},t^{\dagger})$ (see Proposition \ref{19/01/21/08:01}) and \eqref{20/10/5/9:1}.

Thus, we have completed the proof of the lemma. 
\end{proof} 

%%%%%%%%%%%%%%

Now, we are in a position to prove Proposition \ref{19/02/11/16:56}:  
\begin{proof}[Proof of Proposition \ref{19/02/11/16:56}]
We employ the idea  from \cite{CG} (see the proof of Lemma 3.10 in \cite{CG}).   For $R>0$, let $T_{1}(q,R)$ denote the number given by Proposition \ref{19/01/21/08:01}. Furthermore, for $0<t< T_{1}(q,R)$ and $0< \lambda <\frac{T_{1}(q,R)}{t}$, let $(\tau_{\lambda t}, \eta_{\lambda t})$ denote a 
 unique solution to \eqref{19/01/13/14:55} with $t=\lambda t$ in $I( \lambda t) \times Y_{q}(R, \lambda t)$ (see Proposition \ref{19/01/21/08:01}).  
 Then, for $0<t< T_{1}(q,R)$, we define a function $\Omega_{t} \colon (0, \frac{T_{1}(q,R)}{t}) \to (0,\infty)$  by 
\begin{equation}\label{20/10/7/7:17}
\Omega_{t}(\lambda)
:=
\lambda^{\frac{-(2^{*}-2)}{2^{*}-(p+1)}}  
\alpha( \tau_{\lambda t} )
.
\end{equation}
Observe from the continuity of $\tau_{t}$ with respect to $t$ (see Proposition \ref{19/01/21/08:01})  
that  $\Omega_{t}(\lambda)$ is continuous with respect to $\lambda$. 

%%%

We claim: 
\\
\noindent 
{\bf Claim.}~Let $R>0$, and let $0<t<T_{1}(q,R)$; We may choose $T_{1}(q,R)$ even smaller, dependently on $d$, $p$, $q$ and $R$, if necessary. 
 Then, $\Omega_{t}(\lambda)$ is injective on $(0,\frac{T_{1}(q,R)}{t})$. 
 
To prove this claim,  it suffices to show that  $\Omega_{t}(\lambda)$ is strictly decreasing with respect to $\lambda$ on $(0,\frac{T_{1}(q,R)}{t})$.
 In addition, by the continuity of $\Omega_{t}(\lambda)$, it suffices to prove the following: 
 There exists $0< \varepsilon <1$ depending only on $d$, $p$ and $q$ such that 
 if $0< \lambda_{1} <\lambda_{2} <\frac{T_{1}(q,R)}{t}$ and  
 $\lambda_{2}  < (1+\varepsilon) \lambda_{1}$, then 
\begin{equation}\label{20/10/26/9:11}
\Omega_{t}(\lambda_{2}) 
<
\Omega_{t}(\lambda_{1})
.
\end{equation} 
 
Let $0< \varepsilon <1$ be a constant to be specified later, dependently on $d$, $p$ and $q$. Furthermore,  let $0< \lambda_{1}<\lambda_{2}<\frac{T_{1}(q,R)}{t}$, and assume that $\lambda_{2} < (1+\varepsilon)\lambda_{1}$. 
 Observe that 
\begin{equation}\label{20/10/7/10:51}
\begin{split}
\Omega_{t}(\lambda_{2}) -\Omega_{t}(\lambda_{1})
&=
\big\{
\lambda_{2}^{\frac{-(2^{*}-2)}{2^{*}-(p+1)}} 
- 
\lambda_{1}^{\frac{-(2^{*}-2)}{2^{*}-(p+1)}}  
\big\} 
\alpha( \tau_{\lambda_{1}t} )
\\[6pt]
&\quad +
\lambda_{2}^{\frac{-(2^{*}-2)}{2^{*}-(p+1)}} 
\big\{   
\alpha( \tau_{\lambda_{2}t})
- 
\alpha( \tau_{\lambda_{1}t})
\big\} 
.
\end{split} 
\end{equation}
Note that for each $j=1,2$, $(\tau_{\lambda_{j}t}, \eta_{\lambda_{j}t} )$ is a 
 unique solution to \eqref{19/01/13/14:55} with $t=\lambda_{j}t$ in 
 $I( \lambda_{j}t) \times Y_{q}(R, \lambda_{j}t)$; in particular, we have 
\begin{equation}\label{20/10/19/6:2} 
\tau_{\lambda_{j}t}
=
\mathfrak{s}( \lambda_{j}t ;  \tau_{\lambda_{j}t},  \eta_{\lambda_{j}t})
,
\qquad 
\eta_{\lambda_{j}t}
=
\mathfrak{g}(\lambda_{j}t; \tau_{\lambda_{j}t}, \eta_{\lambda_{j}t}) 
.
\end{equation}  
Furthermore, observe from $\tau_{\lambda_{j}t} \in I(\lambda_{j}t)$ and $\lambda_{1} \le \lambda_{2} \le 2\lambda_{1}$ ($0<\varepsilon <1$) that  
\begin{align}
\label{20/10/27/11}
&
\tau_{\lambda_{1}t} \sim \lambda_{1}t \sim \lambda_{2}t \sim \tau_{\lambda_{2}t}, 
\\[6pt]
\label{20/10/27/10:31}
&
\alpha(\tau_{\lambda_{1}t}) 
\sim 
\alpha(\lambda_{1}t) 
\sim 
\alpha(\lambda_{2}t) 
\sim 
\alpha(\tau_{\lambda_{2}t}),
\end{align} 
where the implicit constants depend only on $d$ and $p$. 

%%%%%%%%%%%%%%%%

We consider the first term on the right-hand side of \eqref{20/10/7/10:51}. 
 By an elementary computation, we see that 
\begin{equation}\label{20/10/19/10:41}
\begin{split} 
&
\lambda_{2}^{\frac{-(2^{*}-2)}{2^{*}-(p+1)}} 
- 
\lambda_{1}^{\frac{-(2^{*}-2)}{2^{*}-(p+1)}}  
=
-\lambda_{2}^{\frac{-(2^{*}-2)}{2^{*}-(p+1)}} 
\Big\{
\frac{\lambda_{2}^{\frac{2^{*}-2}{2^{*}-(p+1)}}  }{\lambda_{1}^{\frac{2^{*}-2}{2^{*}-(p+1)}}}  
-1
\Big\} 
\\[6pt]
&=
-\lambda_{2}^{\frac{-(2^{*}-2)}{2^{*}-(p+1)}} 
\frac{2^{*}-2}{2^{*}-(p+1)}
\lambda_{1}^{-1} ( \lambda_{2} - \lambda_{1}) 
+
\lambda_{2}^{\frac{-(2^{*}-2)}{2^{*}-(p+1)}} 
o\big( \lambda_{1}^{-1} |\lambda_{2} - \lambda_{1}| 
\big) 
.
\end{split} 
\end{equation}
 Furthermore, by $\alpha(t)=\delta(\alpha(t))t$ (see \eqref{20/11/2/9:40}),  \eqref{20/10/19/6:2} and Lemma \ref{20/10/20/10:23}, 
 we see that 
\begin{equation}\label{20/11/11/17:55}
\alpha(\tau_{\lambda_{1}t}) 
=
\delta(\alpha(\tau_{\lambda_{1}t})) 
\tau_{\lambda_{1}t}
=
\delta(\alpha(\tau_{\lambda_{1}t})) 
\Big\{
\frac{K_{p}}{A_{1}} \lambda_{1}t
+ 
o(\lambda_{1}t) 
\Big\} 
.
\end{equation} 
Putting \eqref{20/10/19/10:41} and \eqref{20/11/11/17:55} together, 
 we find  that 
\begin{equation}\label{20/11/12/17:27}
\begin{split}
&
\big\{
\lambda_{2}^{\frac{-(2^{*}-2)}{2^{*}-(p+1)}} 
- 
\lambda_{1}^{\frac{-(2^{*}-2)}{2^{*}-(p+1)}}
\big\} 
\alpha(\tau_{\lambda_{1}t}) 
\\[6pt]
&=
-\delta(\alpha(\tau_{\lambda_{1}t}) )
\lambda_{2}^{\frac{-(2^{*}-2)}{2^{*}-(p+1)}} 
\frac{2^{*}-2}{2^{*}-(p+1)}
\Big\{
\frac{K_{p}}{A_{1}}\lambda_{1}t
+
o(\lambda_{1}t)
\Big\}
\lambda_{1}^{-1}  ( \lambda_{2} - \lambda_{1}) 
\\[6pt]
&\quad +
\delta(\alpha(\tau_{\lambda_{1}t}))
\lambda_{2}^{\frac{-(2^{*}-2)}{2^{*}-(p+1)}} 
\Big\{
\frac{K_{p}}{A_{1}}  \lambda_{1}t
+
o(\lambda_{1}t)
\Big\}
o( \lambda_{1}^{-1} |\lambda_{2} - \lambda_{1}| ) 
.
\end{split}  
\end{equation}

Next, we consider the second term on 
 the right-hand side of \eqref{20/10/7/10:51}. 
 Since $\tau_{t}$ is strictly increasing with respect to $t$ (see Proposition \ref{19/01/21/08:01}),   
 $\tau_{\lambda_{2}t} -  \tau_{\lambda_{1}t}=|\tau_{\lambda_{2}t}-  \tau_{\lambda_{1}t}|$. 
 Hence, by $\alpha(t)=\delta(\alpha(t))t$ (see \eqref{20/11/2/9:40}), 
 \eqref{19/01/03/11:39} in Lemma \ref{19/01/03/16:41}, 
 we see that  
\begin{equation}\label{20/11/12/17:46}
\begin{split}
&\alpha(\tau_{\lambda_{2}t}) 
-
\alpha(\tau_{\lambda_{1}t})
\\[6pt]
&=
\{ 
\delta( \alpha(\tau_{\lambda_{2}t}) )
-
\delta( \alpha(\tau_{\lambda_{1}t}) )
\}
\tau_{\lambda_{2}t}
+
\delta( \alpha(\tau_{\lambda_{1}t}) )
\{ \tau_{\lambda_{2}t} - \tau_{\lambda_{1}t} \}
\\[6pt]
&\le 
\Big\{ 
(1+o_{\lambda_{1}t}(1)) \frac{\tau_{\lambda_{2}t}}{\tau_{\lambda_{1}t}} 
\delta( \alpha( \tau_{\lambda_{2}t}) )^{d-3}
+
1 
\Big\} 
\delta( \alpha( \tau_{\lambda_{1}t}) ) 
\big\{  \tau_{\lambda_{2}t} - \tau_{\lambda_{1}t} \big\}
. 
\end{split}
\end{equation}
Here, observe from $\tau_{\lambda_{j}t}=\mathfrak{s}(\lambda_{j}t; \tau_{\lambda_{j}t}, \eta_{\lambda_{j}t})$ (see \eqref{20/10/19/6:2}) and Lemma \ref{20/10/27/11:21} (see Remark \ref{20/11/6/11:50}) that 
\begin{equation}\label{21/2/17/17}
|\tau_{\lambda_{2}t}- \tau_{\lambda_{1}t}|
\lesssim 
|\lambda_{2}t-\lambda_{1}t |
 +
\lambda_{1}t
\alpha(\lambda_{1}t)^{-\Theta_{q}+\theta_{q}} 
\|\eta_{\lambda_{1}t} -\eta_{\lambda_{2}t}\|_{L^{q}}
, 
\end{equation}
where the implicit constant depends only on $d$, $p$ and $q$.  
Furthermore, Lemma \ref{20/10/27/11:21} (see Remark \ref{20/11/6/11:50}) together with \eqref{21/2/17/17} shows that 
\begin{equation}\label{20/11/12/9:40}
\begin{split}
&\Big| 
\tau_{\lambda_{2}t}- \tau_{\lambda_{1}t}
- 
\frac{K_{p}}{A_{1}}\{ \lambda_{2}t-\lambda_{1}t \}
\Big|
\\[6pt]
&\lesssim 
o_{\lambda_{1}t}(1)
|\lambda_{2}t-\lambda_{1}t|
+
\lambda_{1}t
\alpha (\lambda_{1}t)^{-\Theta_{q}+\theta_{q}}
\|\eta_{\lambda_{1}t} -\eta_{\lambda_{2}t}\|_{L^{q}}
, 
\end{split} 
\end{equation}
where the implicit constant depends only on $d$, $p$ and $q$.  
 Observe from $\eta_{\lambda_{j}t}=\mathfrak{g}(\lambda_{j}t; 
 \tau_{\lambda_{j}t}, \eta_{\lambda_{j}t})$ (see \eqref{20/10/19/6:2}),  
 and \eqref{20/10/30/11:48} in Lemma \ref{20/10/30/11:47} that 
\begin{equation}\label{20/11/12/10:56}
\|\eta_{\lambda_{1}t} -\eta_{\lambda_{2}t} \|_{L^{q}}
\lesssim 
\{
o_{\lambda_{1}t}(1) (\lambda_{1}t )^{-1} 
\alpha(\lambda_{1}t)^{\Theta_{q}}
+1 
\} 
|\lambda_{1}t- \lambda_{2}t|  
,
\end{equation}
where the implicit constant depends only on $d$, $p$ and $q$. 
Plugging \eqref{20/11/12/10:56} into \eqref{20/11/12/9:40}, 
and using \eqref{21/2/6/15:27}, 
  we see that  
\begin{equation}\label{20/11/17/17:43}
\tau_{\lambda_{2}t}- \tau_{\lambda_{1}t}
=
\frac{K_{p}}{A_{1}} ( \lambda_{2}-\lambda_{1} )t
+
o_{\lambda_{1}t}(1)
(\lambda_{2}-\lambda_{1}) t 
.
\end{equation}
Furthermore, dividing both sides of \eqref{20/11/17/17:43} by $\tau_{\lambda_{1}t}$, and using \eqref{20/10/27/11} and $\lambda_{2}\le (1+\varepsilon )\lambda_{1}$, 
 we see that 
\begin{equation}\label{20/11/18/10:3}
\frac{\tau_{\lambda_{2}t} }{\tau_{\lambda_{1}t}} -1
\lesssim 
\frac{ \lambda_{2}-\lambda_{1} }{\lambda_{1}}
\le 
\varepsilon 
,
\end{equation}
where the implicit constant depends only on $d$ and $p$. We may write \eqref{20/11/18/10:3} as 
\begin{equation}\label{20/11/18/10:45}
\frac{\tau_{\lambda_{2}t} }{\tau_{\lambda_{1}t}} =1 +O(\varepsilon).
\end{equation}
Putting \eqref{20/11/12/17:46}, \eqref{20/11/17/17:43} and \eqref{20/11/18/10:3} together, we find that 
\begin{equation}\label{20/11/18/10:48}
\begin{split}
&
\lambda_{2}^{\frac{-(2^{*}-2)}{2^{*}-(p+1)}} 
\big\{   
\alpha( \tau_{\lambda_{2}t})
- 
\alpha( \tau_{\lambda_{1}t})
\big\} 
\\[6pt]
&\le  
\lambda_{2}^{\frac{-(2^{*}-2)}{2^{*}-(p+1)}} 
\Big\{ 
(1+o_{\lambda_{1}t}(1) + O(\varepsilon) ) 
\delta( \alpha( \tau_{\lambda_{2}t}) )^{d-3}
+
1 
\Big\} 
\delta( \alpha( \tau_{\lambda_{1}t}) )  
\\[6pt]
&\qquad \times 
\Big\{ 
\frac{K_{p}}{A_{1}}
+
o_{\lambda_{1}t}(1)
\Big\}
(\lambda_{2}-\lambda_{1}) t 
.
\end{split}
\end{equation}

%%%%%%%%%%%%%%%%%%%%

Now, plugging \eqref{20/11/12/17:27} and \eqref{20/11/18/10:48} into \eqref{20/10/7/10:51}, and making a simple computation, 
we see that  
\begin{equation}\label{21/2/18/12:6}
\begin{split}
&\Omega_{t}(\lambda_{2}) -\Omega_{t}(\lambda_{1})
\\[6pt]
&\le 
-\delta(\alpha(\tau_{\lambda_{1}t}) )
\lambda_{2}^{\frac{-(2^{*}-2)}{2^{*}-(p+1)}} 
\frac{2^{*}-2}{2^{*}-(p+1)}
\frac{K_{p}}{A_{1}} 
( \lambda_{2} - \lambda_{1}) t
\\[6pt]
&\quad +
\delta(\alpha(\tau_{\lambda_{1}t}))
\lambda_{2}^{\frac{-(2^{*}-2)}{2^{*}-(p+1)}} 
o(|\lambda_{2} - \lambda_{1}|t )
\\[6pt]
&\quad +
\delta( \alpha( \tau_{\lambda_{1}t}) )
\lambda_{2}^{\frac{-(2^{*}-2)}{2^{*}-(p+1)}} 
\Big\{ \delta( \alpha( \tau_{\lambda_{2}t}) )^{d-3}
+
1 
\Big\} 
\frac{K_{p}}{A_{1}}(\lambda_{2}-\lambda_{1}) t
\\[6pt]
&\quad + 
\delta( \alpha( \tau_{\lambda_{1}t}) )
\lambda_{2}^{\frac{-(2^{*}-2)}{2^{*}-(p+1)}} 
\big\{ 
o_{\lambda_{1}t}(1) 
+ 
O(\varepsilon) \delta( \alpha( \tau_{\lambda_{2}t}) )^{d-3}
\big\}
|\lambda_{2}-\lambda_{1}| t
.
\end{split}
\end{equation}
We find from \eqref{21/2/18/12:6} that if $t$ and $\varepsilon$ are sufficiently small depending only on $d$, $p$ and $q$, then $\Omega_{t}(\lambda_{2}) -\Omega_{t}(\lambda_{1})<0$. 
 Thus, we have proved that $\Omega_{t}(\lambda)$ is strictly decreasing with respect to $\lambda$, which implies that $\Omega_{t}$ is injective.

Now, we shall finish the proof of the proposition:
\\
\noindent 
{\bf End of the proof.}~Let $R_{*}$ and $T_{*}$ be the numbers given in Lemma \ref{20/9/23/9:8}. Furthermore, let $0<t<T_{*}$, and let $(\tau_{t},\eta_{t})$ be  a unique solution to \eqref{19/01/13/14:55} in $I(t) \times Y_{q}(R_{*}, t)$ (see Proposition \ref{19/01/21/08:01}). 
 Then, what we need to prove is that if $u_{t} \in \widetilde{\mathcal{G}}_{t}$, then $u_{t}=W+\eta_{t}$.        
 
Let $u_{t} \in \widetilde{\mathcal{G}}_{t}$. Then, Lemma \ref{20/9/23/9:8} shows that there exists $\mu(t)=1+o_{t}(1)$ such that, 
  defining $t^{\dagger}$, $\alpha^{\dagger}(t)$  and $\eta_{t}^{\dagger}$ as 
\begin{align}
\label{20/10/6/12:9}
&t^{\dagger}
:=\{ \mu(t) u_{t}(0) \}^{-2^{*}+p+1}t, 
\qquad 
\alpha^{\dagger}(t)
:=\{ \mu(t) u_{t}(0) \}^{-2^{*}+2} \alpha(\tau_{t})
,
\\[6pt]
\label{20/10/9/10:46}
&
\eta_{t}^{\dagger}
:= T_{\mu(t)u_{t}(0)}[u_{t}] -W, 
\end{align}
we have the following: 
\begin{align}
\label{20/10/9/11:17}
&
\lim_{t\to 0}\alpha^{\dagger}(t)=0,
\qquad 
\lim_{t\to 0}t^{\dagger} =0
, 
\\[6pt]
\label{20/10/6/12:2}
&(\tau_{t^{\dagger}}, \eta_{t^{\dagger}}) 
=
(\beta(\alpha^{\dagger}(t)), \eta_{t}^{\dagger}) 
\in 
I(t^{\dagger}) \times Y_{q}(R_{*},t^{\dagger})
, 
\end{align}
where $(\tau_{t^{\dagger}}, \eta_{t^{\dagger}})$ denotes a unique solution to \eqref{19/01/13/14:55} with $t=t^{\dagger}$ in $I(t^{\dagger}) \times Y_{q}(R_{*},t^{\dagger})$.  
Since $\alpha$ is the inverse function of $\beta$, 
  \eqref{20/10/6/12:2} shows  that 
\begin{equation}\label{20/10/6/12:21}
\alpha(\tau_{t^{\dagger}})
=
\alpha(\beta ( \alpha^{\dagger}(t) ) )
=
\alpha^{\dagger}(t)
.
\end{equation}
Note that 
\begin{equation}\label{20/10/9/11:16}
\frac{t^{\dagger}}{t}
=
\{\mu(t) u_{t}(0)\}^{-2^{*}+p+1} 
.
\end{equation} 
Furthermore, observe from \eqref{20/10/9/11:16} and \eqref{20/10/6/12:21} that 
\begin{equation}\label{20/10/7/7:19}
\Omega_{t}\Big( \frac{t^{\dagger}}{t} \Big)
 = 
\{ \mu(t) u_{t}(0)\}^{2^{*}-2} 
\alpha^{\dagger}(t) 
=
\alpha(\tau_{t})
=
\Omega_{t}(1)
.
\end{equation}
Then,  the injectivity of $\Omega_{t}(\lambda)$ with respect to $\lambda$, 
 together with \eqref{20/10/7/7:19} and \eqref{20/10/9/11:16}, 
 implies that $\mu(t)u_{t}(0)=1$. Hence, $t^{\dagger}=t$. 
 Furthermore, by $t^{\dagger}=t$, $\mu(t)u_{t}(0)=1$, \eqref{20/10/9/10:46} and \eqref{20/10/6/12:2}, we see that $\eta_{t}=\eta_{t^{\dagger}}=u_{t}-W$. 
Thus, we have completed the proof.

\end{proof}

%%%%%%%%%%%%%%%%%%%%%%%%%%%%%%%%%%%%%%%%%%%%%%%%%%%
%%%%%%%%%%%%%%%%%%%%%%%%%%%%%%%%%%%%%%%%%%%%%%%%%%%

\appendix 

%%%%%%%%%%%%%%%%%%%%%%%%%%%%%%%%%%%%%%%%%%%%%%%%%%%%%%%%%%%%%%%%%%%%%%%%%%%%%%%

\section{Compactness in Lebesgue spaces}\label{18/11/11/15:30}

We record a compactness theorem in Lebesgue spaces (see Proposition A.1 of \cite{Killip-Visan-book}):

\begin{lemma}\label{18/11/10/13:29} 
Assume $d\ge 1$, and let $1\le q \le \infty$. A sequence $\{f_{n}\}$ in $L^{q}(\mathbb{R}^{d})$ has a convergent subsequence 
 if and only if 
 it satisfies the following properties:
\begin{enumerate}
\item 
\begin{equation}\label{18/11/07/10:16}
\sup_{n\ge 1} \|f_{n} \|_{L^{q}}<\infty 
.
\end{equation}

\item 
\begin{equation}\label{18/11/07/10:17}
\lim_{y \to 0}
\sup_{n\ge 1} \| f_{n}(\cdot)-f_{n}(\cdot+y) \|_{L^{q}} =0
.
\end{equation}

\item 
\begin{equation}\label{18/11/07/10:18}
\lim_{R\to \infty} \sup_{n\ge 1} \|   f_{n} \|_{L^{q}(|x|\ge R)} =0
.
\end{equation}
\end{enumerate} 
\end{lemma}

\begin{remark}\label{20/8/14/14:15}
Although Proposition A.1 of \cite{Killip-Visan-book} does not refer to the case $q=\infty$,
 we may include  this case in the statement. 
 Note that when $q=\infty$,  \eqref{18/11/07/10:17} implies that the sequence $\{f_{n}\}$ is continuous; hence the limit along the subsequence is also continuous.
 \end{remark}

%%%%%%%%%%%%%%%%%%%%%%%%%%%%%%%%%%%%%%%%%%%%%%%%%%%%%%%%%%%%%%%%%%%%%%%%%%%%%%

\section{Completeness of $\boldsymbol{Y_{q}(R,t)}$}\label{19/03/02/16:22}

Let $\frac{d}{d-2}<q < \infty$, $R>0$,  $0< t <1$,  and let $Y_{q}(R,t)$ be a set defined by \eqref{19/01/13/15:35}. 
 Then, we shall show that  $Y_{q}(R,t)$  is complete with the metric  induced from $L^{q}(\mathbb{R}^{d})$. 
\par 
Let $\{ \eta_{n} \}$ be a Cauchy sequence in $Y_{q}(R,t)$.  Since $L_{\rm rad}^{q}(\mathbb{R}^{d})$ is complete, 
there exists a radial function $\eta_{\infty} \in L^{q}(\mathbb{R}^{d})$ such that 
\begin{equation}\label{19/03/02/16:30}
\lim_{n\to \infty}\eta_{n}=\eta_{\infty}
\qquad \mbox{strongly in $L^{q}(\mathbb{R}^{d})$}
.
\end{equation}
Furthermore, we see that 
\begin{equation}\label{19/03/02/19:50} 
\|\eta_{\infty}\|_{L^{q}} 
=
\lim_{n\to \infty} \| \eta_{n}\|_{L^{q}}
\le R\alpha(t)^{\frac{d-2}{2}-\frac{d}{2q}},
\end{equation}
so that $\eta_{\infty} \in Y_{q}(R,t)$. Thus, we have proved the completeness.  

%%%%%%%%%%%%%%%%%%%%%%%%%%%%%%%%%%%%%%%%%%%%%%%%%%%%%%%%%%%%%%%%%%%%%%%%%%%%%%%

\section{Elementary computations}\label{20/10/14/9:12}

\begin{lemma}\label{20/12/30/8:13}
Define a function $\beta$ on $(0,\infty)$ by $\beta(s):=s \log{(1+s^{-1})}$. 
 Then,  $\beta$ is strictly increasing  on $(0,\infty)$, and the image of $(0,\infty)$ by $\beta$ is $(0,1)$; 
 hence the inverse is defined on $(0,1)$, say $\alpha \colon (0,1) \to (0,\infty)$.  
 Furthermore, there exists $0< T_{0} < 1$ such that if $0<t\le T_{0}$, then 
\begin{equation}\label{20/12/30/8:19}
\alpha(t)  \sim   \frac{t}{\log{(1+t^{-1})}}
.
\end{equation} 
\end{lemma}
\begin{proof}[Proof of Lemma \ref{20/12/30/8:19}]
We may write $\beta$ as  
\begin{equation}\label{20/12/30/8:25}
\beta(s)
=
s\int_{0}^{1}\frac{d}{d\theta}\log{(s+\theta)}\,d\theta 
=
1-\int_{0}^{1}\frac{\theta}{s+\theta}\,d\theta. 
\end{equation}
This formula shows that $\beta$ is strictly increasing on $(0,\infty)$, 
 $\lim_{s\to 0}\beta(s)=0$, and $\lim_{s\to \infty}\beta(s)=1$.
 In particular, the image of $(0,\infty)$ by $\beta$ is $(0,1)$. 

Next, let $\alpha\colon (0,1)\to (0,\infty)$ be the inverse of $\beta$.  
 Note that $\alpha$ is strictly increasing, continuous, and $\lim_{t\to 0}\alpha(t)=0$. Hence, there exists $0< T_{0} < 1$ such that 
  $\alpha(t) \le e^{-1} $ for all $0<t\le T_{0}$.  
 Let $0< t\le T_{0}$, and put $s:=\alpha(t)$. Note that  $0<s\le e^{-1}$ and $\beta(s)=t$.  
 Then, observe from $\beta(s):=s \log{(1+s^{-1})}$ that  $s \le t \le s^{\frac{1}{2}}$, which implies that 
\begin{equation}\label{20/12/30/13:47}
\frac{1}{2}\log{(s^{-1})} \le \log{(t^{-1})} \le \log{(s^{-1})}.  
\end{equation} 
Furthermore, it follows from  $0< s,t \le 1$, \eqref{20/12/30/13:47} and $s=\alpha(t)$ that 
\begin{equation}\label{20/12/30/13:53}
t = \beta(s) \sim s \log{(s^{-1}) } \sim s   \log{(t^{-1})} \sim \alpha(t) \log{(1+t^{-1})}.  
\end{equation} 
Thus, we have prove the lemma. 
 \end{proof}

\begin{lemma}\label{19/02/16/10:53}
Assume $d=3,4$. Let $0<s_{1}\le s_{2}<1$. Then, we have 
\begin{equation}\label{19/02/16/10:55}
\int_{\mathbb{R}^{d}}
\int_{|\xi| \le 1 } 
\frac{\min\{|x||\xi|, \, |x|^{2}|\xi|^{2}\}}{|\xi|^{2}(|\xi|^{2}+s_{1})
(|\xi|^{2}+s_{2})}
(1+|x|)^{-(d+2)}
\,d\xi dx 
\lesssim 
\log{(1+s_{2}^{-1})}\delta(s_{2})^{-1}
,
\end{equation}
where the implicit constant depends only on $d$. 
\end{lemma}
\begin{proof}[Proof of Lemma \ref{19/02/16/10:53}]
Observe that 
\begin{equation}\label{19/02/16/10:54}
\begin{split}
&
\int_{\mathbb{R}^{d}}
\int_{|\xi| \le 1 } 
\frac{\min\{|x||\xi|, \, |x|^{2}|\xi|^{2}\}}{|\xi|^{2}(|\xi|^{2}+s_{1})
(|\xi|^{2}+s_{2})}
(1+|x|)^{-(d+2)}
\,d\xi dx 
\\[6pt]
&\lesssim 
\int_{0}^{1}
\int_{0}^{1} 
\frac{r^{d+1} \lambda^{d-1}}{(1+r)^{d+2}(\lambda^{2}+s_{1})
(\lambda^{2}+s_{2})}
\,d\lambda dr
\\[6pt]
&\quad + 
\int_{1}^{\infty}
\int_{0}^{\frac{1}{r}} 
\frac{r^{d+1} \lambda^{d-1}}{(1+r)^{d+2}(\lambda^{2}+s_{1})
(\lambda^{2}+s_{2})}
\,d\lambda dr
\\[6pt]
&\quad + 
\int_{1}^{\infty}
\int_{\frac{1}{r}}^{1} 
\frac{r^{d} \lambda^{d-2}}{(1+r)^{d+2}(\lambda^{2}+s_{1})
(\lambda^{2}+s_{2})}
\,d\lambda dr,
\end{split}
\end{equation}
where the implicit constant depends only on $d$. 
 Consider the first term on the right-hand side of \eqref{19/02/16/10:54}. 
 By an elementary computation, and \eqref{20/12/12/11:12} (only for $d=3$), 
we see that
\begin{equation}\label{19/02/16/11:10}
\begin{split}
&
\int_{0}^{1}
\int_{0}^{1} 
\frac{r^{d+1} \lambda^{d-1}}{(1+r)^{d+2}(\lambda^{2}+s_{1})
(\lambda^{2}+s_{2})}
\,d\lambda dr
\\[6pt]
&\lesssim
 \int_{0}^{1} 
\frac{\lambda^{d-1}}{(\lambda^{2}+s_{1})
(\lambda^{2}+s_{2})}
\,d\lambda 
\lesssim
 \int_{0}^{1} 
\frac{\lambda^{d-3}}{\lambda^{2}+s_{2}}
\,d\lambda 
\lesssim 
\delta(s_{2})^{-1},
\end{split}
\end{equation}
where the implicit constant depends only on $d$. Next, consider the second term on the right-hand side of \eqref{19/02/16/10:54}. 
 By integration by substitution ($\lambda=\sqrt{s_{2}}\mu$), elementary computations, and \eqref{20/12/12/11:12} (only for $d=3$), we see that 
\begin{equation}\label{19/02/16/11:25}
\begin{split}
&
\int_{1}^{\infty}
\int_{0}^{\frac{1}{r}} 
\frac{r^{d+1} \lambda^{d-1}}{(1+r)^{d+2}(\lambda^{2}+s_{1})
(\lambda^{2}+s_{2})}
\,d\lambda dr 
\\[6pt]
&\le
\int_{1}^{\infty}\int_{0}^{\frac{1}{\sqrt{s_{2}} r}} 
\frac{s_{2}^{\frac{d-4}{2}} \mu^{d-1}}{r (\mu^{2}+\frac{s_{1}}{s_{2}}) (\mu^{2}+1) }
\,d\mu dr
\\[6pt]
&= 
\int_{1}^{\frac{1}{\sqrt{s_{2}}}} 
\int_{0}^{1} 
\frac{s_{2}^{\frac{d-4}{2}} \mu^{d-1}}{r (\mu^{2}+\frac{s_{1}}{s_{2}}) (\mu^{2}+1)}
\,d\mu dr
+
\int_{1}^{\frac{1}{\sqrt{s_{2}}}} 
\int_{1}^{\frac{1}{\sqrt{s_{2}} r}} 
\frac{s_{2}^{\frac{d-4}{2}} \mu^{d-1}}{r (\mu^{2}+\frac{s_{1}}{s_{2}}) (\mu^{2}+1) }
\,d\mu dr
\\[6pt]
&\quad +
 \int_{\frac{1}{\sqrt{s_{2}}}}^{\infty}
 \int_{0}^{\frac{1}{\sqrt{s_{2}} r}} 
\frac{s_{2}^{\frac{d-4}{2}} \mu^{d-1}}{r (\mu^{2}+\frac{s_{1}}{s_{2}}) (\mu^{2}+1)}
\,d\mu dr
\\[6pt]
&\le 
\int_{1}^{\frac{1}{\sqrt{s_{2}}}} 
\frac{s_{2}^{\frac{d-4}{2}}}{r} 
\,dr
+
\int_{1}^{\frac{1}{\sqrt{s_{2}}}} 
\int_{1}^{\frac{1}{\sqrt{s_{2}}}} 
\frac{s_{2}^{\frac{d-4}{2}} \mu^{d-3}}{r (\mu^{2}+1)}
\,d\mu dr
+
 \int_{\frac{1}{\sqrt{s_{2}}}}^{\infty}
 \int_{0}^{\frac{1}{\sqrt{s_{2}} r}} 
\frac{s_{2}^{\frac{d-4}{2}}}{r}
\,d\mu dr
\\[6pt]
&\lesssim 
\delta(s_{2})^{-1}\log{(1+s_{2}^{-1})},
\end{split}
\end{equation}
where the implicit constant depends only on $d$.
 Finally, we consider the last term on the right-hand side of \eqref{19/02/16/10:54}. By change of the order of integrals, and \eqref{20/12/12/11:12} (only for $d=3$), we see that 
\begin{equation}\label{19/02/16/14:50}
\begin{split}
&
\int_{1}^{\infty}
\int_{\frac{1}{r}}^{1} 
\frac{r^{d} \lambda^{d-2}}{(1+r)^{d+2}(\lambda^{2}+s_{1})
(\lambda^{2}+s_{2})}
\,d\lambda dr
\\[6pt]
&\le 
\int_{1}^{\infty}
\int_{\frac{1}{r}}^{1} 
\frac{\lambda^{d-2}}{r^{2}(\lambda^{2}+s_{1})
(\lambda^{2}+s_{2})}
\,d\lambda dr
=
\int_{0}^{1}\frac{\lambda^{d-2}}{(\lambda^{2}+s_{1})
(\lambda^{2}+s_{2})} \int_{\frac{1}{\lambda}}^{\infty} \frac{1}{r^{2}}
\,dr d\lambda
\\[6pt]
&=
\int_{0}^{1}\frac{\lambda^{d-1}}{(\lambda^{2}+s_{1})
(\lambda^{2}+s_{2})}\,d\lambda
\le 
\int_{0}^{1}\frac{\lambda^{d-3}}{\lambda^{2}+s_{2}}\,d\lambda
\lesssim 
\delta(s_{2})^{-1},
\end{split}
\end{equation}
where the implicit constant depends only on $d$. 

Putting \eqref{19/02/16/10:54} through \eqref{19/02/16/14:50} together, 
 we find that \eqref{19/02/16/10:55} holds.
\end{proof}

%%%

The following lemma is quoted from \cite{Jensen-Kato}: 
\begin{lemma}[Lemma 3.12 of \cite{Jensen-Kato}]\label{21/1/12/9:48}
Let $\mathscr{X}$, $\mathscr{Y}$, $\mathbf{X}$, $\mathbf{Y}$ be vector spaces. 
 Let $\mathscr{A}\colon \mathscr{X} \to \mathscr{Y}$, $B\colon \mathbf{X}\to \mathscr{X}$,  $C \colon \mathscr{Y} \to \mathbf{Y}$ be linear operators. Define $\mathbf{A}:=C \mathscr{A} B$. If $\mathbf{A}^{-1}$ exists, then $\mathscr{A}^{-1}=B \mathbf{A}^{-1} C$ provided $B$ is surjective and $C$ is injective. 
\end{lemma}

%%%%%%%%%%%%%%%%%%%%%%%%%%%%%%%%%%%%%%%%%%%%%%%%%%%%%%%%%%%%%%%%%%%%%%%%%%%%%%%

\bibliographystyle{plain}

%%%%%%%%%%%%%%%%%%%%%%%%%%%%%%%%%%%%%%%%%%%%%
\end{document}